\documentclass[american]{amsart}

\usepackage[english]{babel}
\usepackage[utf8]{inputenc}
\usepackage{graphicx}
\usepackage{subfigure}
\usepackage{hyperref}
\usepackage{color}
\usepackage{amssymb}

\usepackage{xspace}

\definecolor{Chris}{rgb}{0.5,0,1}
\definecolor{dubiousEnglish}{rgb}{1,0,.2}

\DeclareMathOperator{\dR}{{\mathrm{dR}}}
\DeclareMathOperator{\End}{End}
\DeclareMathOperator{\ev}{ev}
\DeclareMathOperator{\GL}{GL}

\DeclareMathOperator{\ImaginaryPart}{Im}

\DeclareMathOperator{\RealPart}{Re}

\DeclareMathOperator{\Span}{span}

\DeclareMathOperator{\CS}{\mathrm{CS}}
\DeclareMathOperator{\tr}{\mathrm{tr}}

\newcommand{\conSum}{{\mathbin{\#}}}
\newcommand{\restricted}[2]{{\left.{#1}\right|_{#2}}}
\newcommand{\lcan}{{\lambda_{\mathrm{can}}}}
\newcommand{\ocan}{{\omega_{\mathrm{can}}}}
\newcommand{\p}{\partial}
\newcommand{\lie}[1]{{\mathcal{L}_{#1}}}
\newcommand{\abs}[1]{{\left\lvert #1\right\rvert}}
\newcommand{\norm}[1]{{\lVert #1\rVert}}
\newcommand{\0}{{\mathbf{0}}}

\newcommand{\flow}{\varphi}

\renewcommand{\epsilon}{\varepsilon}

\makeatletter
\newcommand{\BLOB}{\textsf{bLob}\xspace}
\makeatother

\newcommand{\DD}{{\mathbb D}}
\renewcommand{\SS}{{\mathbb{S}}}
\newcommand{\RR}{{\mathbb{R}}}
\newcommand{\QQ}{{\mathbb{Q}}}
\newcommand{\TT}{{\mathbb{T}}}
\newcommand{\CC}{{\mathbb{C}}}
\newcommand{\ZZ}{{\mathbb{Z}}}
\newcommand{\NN}{{\mathbb{N}}}
\newcommand{\fF}{{\mathcal F}}
\newcommand{\jJ}{{\mathcal J}}
\newcommand{\hH}{{\mathcal H}}
\newcommand{\nN}{{\mathcal N}}
\newcommand{\oO}{{\mathcal{O}}}
\newcommand{\uU}{{\mathcal U}}

\newcommand{\IntSig}{\mathring \Sigma}

\newcommand{\lambdaGT}{\lambda_{{\operatorname{GT}}}}
\newcommand{\xiGT}{\xi_{{\operatorname{GT}}}}

\newcommand{\cocore}{\mathcal{K}'}
\newcommand{\defin}[1]{\textbf{#1}}

\newcommand{\mM}{{\mathcal M}}
\newcommand{\blowdown}{{\mathbin{/\!/}}}
\newcommand{\Wcob}{W}

\DeclareMathOperator{\Aff}{Aff}

\renewcommand{\k}{\mathbf{k}}

\newcommand{\RRpx}{\mathbb{R}^*_+}
\newcommand{\CCx}{\mathbb{C}^*}
\newcommand{\Ok}{\mathcal{O}_\k}
\newcommand{\Okx}{\Ok^\times}
\newcommand{\Okxp}{\Ok^{\times, +}}
\newcommand{\Uk}{\mathbb{U}_\k}
\newcommand{\Ukp}{\mathbb{U}_\k^+}
\newcommand{\Gammak}{\Gamma_\k}
\newcommand{\h}{\mathfrak{h}}
\newcommand{\Gk}{G_\k}
\newcommand{\GrsOrig}{\mathcal{G}^{r,s}}
\newcommand{\GrsunOrig}{\mathcal{G}^{r,s}_1}

\newcommand{\gR}{\mathfrak{g}_\RR}
\newcommand{\gC}{\mathfrak{g}_\CC}

\newcommand{\SL}{\operatorname{SL}}

\newcounter{CobStep}
\newcommand{\X}{\Theta}
\newcommand{\T}{T}
\newcommand{\R}{U}
\newcommand{\Ss}{V}
\newcommand{\U}{X}
\newcommand{\V}{Y}

\newcommand{\step}[1]{\refstepcounter{CobStep}{\vspace{\baselineskip}\noindent
\itshape Step~\arabic{CobStep}: #1}}

\theoremstyle{plain}

\newcounter{maintheorem}

\newtheorem{main_theorem}[maintheorem]{Theorem}

\newtheorem{propIntro}{Proposition}

\newtheorem{corIntro}[propIntro]{Corollary}

\newtheorem{theorem}{Theorem}[section]
\newtheorem{lemma}[theorem]{Lemma}
\newtheorem{conjecture}[theorem]{Conjecture}
\newtheorem{corollary}[theorem]{Corollary}
\newtheorem*{coru}{Corollary}

\newtheorem*{qu}{Question}
\newtheorem{proposition}[theorem]{Proposition}

\theoremstyle{remark}
\newtheorem{remark}[theorem]{Remark}
\newtheorem{example}[theorem]{Example}

\newtheorem{exIntro}[propIntro]{Example}

\theoremstyle{definition}
\newtheorem*{definition}{Definition}
\newtheorem{definitionNumbered}[theorem]{Definition}
\newtheorem{defnIntro}[propIntro]{Definition}

\numberwithin{equation}{section}

\title[Weak and strong fillability in higher dimensions]{Weak 
and strong fillability of higher dimensional contact manifolds}

\author[P.\ Massot]{Patrick Massot}

\email[P.\ Massot]{patrick.massot@math.u-psud.fr}

\address[P.\ Massot]{Département de mathématiques\\
  Université Paris Sud\\
  F-91405 Orsay Cedex\\
  FRANCE}

\author[K.\ Niederkrüger]{Klaus Niederkrüger}

\email[K.\ Niederkrüger]{niederkr@math.univ-toulouse.fr}

\address[K.\ Niederkrüger]{
  Institut de mathématiques de Toulouse\\
  Université Paul Sabatier -- Toulouse III\\
  118 route de Narbonne\\
  F-31062 Toulouse Cedex 9\\
  FRANCE}

\author[C.\ Wendl]{Chris Wendl}

\email[C.\ Wendl]{c.wendl@ucl.ac.uk}

\address[C.\ Wendl]{
Department of Mathematics \\
University College London\\
Gower Street\\
London WC1E 6BT\\
United Kingdom}

\begin{document}

\begin{abstract}
For contact manifolds in dimension three, the notions of weak and strong
symplectic fillability and tightness are all known to be inequivalent.
We extend these facts to higher dimensions: in particular, we define
a natural generalization of weak 
fillings and prove that it is indeed weaker (at least in dimension five), 
while also being obstructed by
all known manifestations of ``overtwistedness''.  We also find the first
examples of contact manifolds in all dimensions that are not symplectically 
fillable but also cannot be called overtwisted in any reasonable sense.
These depend on a higher dimensional analogue of Giroux torsion, which we
define via the existence in all dimensions of exact symplectic manifolds 
with disconnected contact boundary.
\end{abstract}

\maketitle

\section*{Introduction}
\setcounter{maintheorem}{0}

Contact structures in dimension $2n-1$ arise naturally from symplectic
structures in dimension~$2n$ by considering symplectic manifolds with
a convexity condition at the boundary.  It has been understood
since the work of Gromov \cite{Gromov_HolCurves} and
Eliashberg \cite{Eliashberg_filling} in the late 1980's that not every
contact structure arises in this way, i.e.~not all contact structures
are \emph{symplectically fillable}.  Moreover, in dimension three, there
are distinct notions of \emph{strong} and \emph{weak} fillability,
and they are both closely related to the deep dichotomy between \emph{tightness}
and \emph{overtwistedness}, which plays a crucial role in the problem
of classifying contact structures.  One has in particular the following
inclusions among classes of contact structures on $3$-manifolds:
$$
\{ \text{strongly fillable} \} \subset \{ \text{weakly fillable} \}
\subset \{ \text{tight} \}.
$$
Both of these are proper inclusions: in the first case this was shown by
Eliashberg \cite{Eliashberg3Torus}, and in the second by
Etnyre and Honda \cite{EtnyreHonda_tightNonfillable}, though today
a simple alternative construction is also available using the notion
of \emph{Giroux torsion}.  This invariant, introduced by Giroux in
\cite{Giroux_2000}, measures the amount that a contact structure ``twists''
in neighborhoods of certain embedded
$2$-tori; it does not imply overtwistedness but
does obstruct strong \cite{Gay_GirouxTorsion} and sometimes also
weak \cite{GhigginiHonda_twisted} fillability.  It also plays a key role in
several classification results for tight contact structures,
such as the ``coarse'' classification due to Colin, Giroux
and Honda \cite{ColinGH}.

Most of the above discussion only makes sense so far in dimension three.
This is partly because it is not known whether the tight/overtwisted 
dichotomy extends to  higher dimensions, although recent work of the
second author and others
(e.g.~\cite{NiederkruegerPlastikstufe, BourgeoisContactHomologyLeftHanded})
has revealed hints of ``overtwistedness'' in certain classes of examples.
It also has not been clear up to now whether the notions of
weak filling and Giroux torsion have any interesting higher dimensional
counterparts.
One of our main goals in this paper is to answer the latter question
in the affirmative: we will show that several well known three-dimensional
phenomena, such as the existence of tight but non-fillable or weakly
but not strongly fillable contact manifolds, also occur in higher
dimensions.

Let us begin the discussion with the phenomenon of contact structures
that are tight but not (strongly) fillable.
The emblematic example is the family of contact structures on $\TT^3$
defined for $k\in \NN$ by
\begin{equation*}
  \xi_k := \ker \bigl(\cos ks\; d\theta + \sin ks\; dt\bigr) \;,
\end{equation*}
where we define $\TT^3$ as $(\RR / 2\pi\ZZ) \times (\RR / \ZZ)^2$ with
coordinates $(s,t,\theta)$.
These contact structures are all tight due to Bennequin's theorem
\cite{Bennequin}, since they are covered by the standard contact structure on
$\RR^3$, but Eliashberg \cite{Eliashberg3Torus} showed that only $\xi_1$ has a
strong symplectic filling.
Despite this lack of fillability, they share other important properties that are
incompatible with overtwistedness.
For example, they are \emph{hypertight}, i.e. they allow  Reeb vector fields 
without contractible closed orbits, in contrast to Hofer's theorem
\cite{HoferWeinstein} that such orbits always exist in the overtwisted case.
More importantly, they are not ``flexible,'' meaning they are all 
homotopic as plane fields yet not isotopic \cite{Giroux_infinity}, 
whereas overtwisted contact structures are maximally flexible due to
Eliashberg's classification theorem \cite{Eliashberg_Overtwisted}.

In higher dimensions, it is an open question whether one can define a
reasonable notion of tightness, but of course flexibility and
contractible Reeb orbits are easy to define.  Strong fillability can also
be defined in the same way as in dimension three, by considering symplectic
manifolds with convex boundary (see Definition~\ref{defn:sympFilling} below).
This allows us to
compare the properties of the contact structures~$\xi_k$ on
$\TT^3$ discussed above with the following statement.

\begin{main_theorem}\label{thm:generalisationTori}
Identify the torus $\TT^2$ with $(\RR / 2\pi\ZZ) \times (\RR / \ZZ)$
with coordinates $(s,t)$.
In any odd dimension, there is a closed manifold~$M$ carrying two
contact forms~$\alpha_+$ and $\alpha_-$ such that the formula
\begin{equation*}
	\xi_k := \ker\left(\frac{1 + \cos ks}{2}\, \alpha_+ + 
		\frac{1 - \cos ks}{2}\, \alpha_- + \sin ks \; dt \right)
\end{equation*}
for $k \in \NN$ defines a family of contact structures on $\TT^2\times M$ 
with the following properties:
\begin{enumerate}
\item They all admit Reeb vector fields without contractible closed orbits.
\item They are all homotopic as almost contact structures but not
contactomorphic.
\item $(\TT^2 \times M,\xi_k)$ is strongly fillable only for $k=1$.
\end{enumerate}
\end{main_theorem}

We recover the $3$-dimensional case discussed above by taking $M = \SS^1 :=
\RR / \ZZ$ and
$\alpha_\pm = \pm d\theta$ in the theorem.

The non-fillability of the above contact structures on $\TT^3$ was
later recognized to be a consequence of the positivity of their Giroux
torsion, and we'd next like to generalize this fact.
Let us briefly recall the definition of Giroux torsion, in language
that is suitable for generalization to higher dimensions.
Denote by $(A,\beta)$ the cylinder $A := \RR\times \SS^1$ with coordinates
$(s,\theta)$, together with the $1$-form $\beta := s\, d\theta$, which
makes it the completion of a Liouville domain 
(see Definition~\ref{defn:sympFilling}).
The \defin{contactization}\footnote{Our use of the term
``contactization'' is slightly nonstandard, as the word is typically 
used in the literature to mean a product of a Liouville domain with 
$\RR$ instead
of with~$\SS^1$.  In this paper, we shall go back and forth between
both meanings of the term---it should always be clear from context
which one is meant.} of $(A, \beta)$ is the manifold $A
\times \SS^1 = \RR\times\SS^1 \times \SS^1 = \RR\times \TT^2$ equipped
with the contact structure $\ker (dt + \beta)$, where $t$ denotes the
coordinate on the new $\SS^1$-factor.
This contact structure is tangent to the $\RR$-factor, and it makes a
half twist as we move from $s = -\infty$ to $s = +\infty$.
One can then compactify this domain by identifying it with the interior of
$[0,\pi] \times \SS^1 \times \SS^1$ with coordinates $(s,t,\theta)$
and contact structure
\begin{equation*}
  \ker \bigl(\cos s\; d\theta + \sin s\; dt\bigr) \;.
\end{equation*}
This last contact manifold is called a \defin{Giroux $\pi$-torsion}
domain (or sometimes Giroux \emph{half-torsion} domain).
Such domains can be glued along boundary tori to achieve any number of
half turns.
The Giroux torsion of a contact $3$-manifold $(V, \xi)$ is defined to
be the supremum of all integers~$n$ such that $(V, \xi)$ contains $2n$
Giroux $\pi$-torsion domains glued together.
The idea described above can be conveniently rephrased in terms of
\emph{ideal Liouville domains}, a notion recently introduced by
Giroux.
We will review the precise definition in \S\ref{section:giroux_domains}, but in
a nutshell, an ideal Liouville domain is the compactification of a complete
Liouville manifold that appears naturally e.g.~as the closure of a page of a
supporting open book decomposition, or more generally, the closure of any
component of a $\xi$-convex hypersurface minus its dividing set.
With this notion, a Giroux $\pi$-torsion domain can be viewed
directly as the contactization of an ideal Liouville domain.
In this paper, we shall refer to contactizations of ideal Liouville
domains as \defin{Giroux domains}.
The fact that Giroux torsion is an obstruction to strong fillability
\cite{Gay_GirouxTorsion} is then generalized to the following theorem.

\begin{main_theorem}\label{thm:1torsion}
If a contact manifold contains a connected codimension~$0$ submanifold with
nonempty boundary obtained by gluing together two Giroux domains, then it is not
strongly fillable.
\end{main_theorem}

Observe that at least one of the Giroux domains in Theorem~\ref{thm:1torsion}
must always have disconnected boundary.
The existence of Liouville domains with disconnected boundary in
dimensions four and higher is itself a nontrivial fact: the first
examples were found by McDuff \cite{McDuff_contactType} in dimension
four, and more were found by Geiges in dimensions four
\cite{Geiges_disconnected4} and six \cite{Geiges_disconnected},
and Mitsumatsu \cite{Mitsumatsu_Anosov} in dimension four.
The following notion generalizes the construction of Geiges:

\begin{defnIntro}
  A \defin{Liouville pair} on an oriented $(2n-1)$-dimensional
  manifold $M$ is a pair $(\alpha_+,\alpha_-)$ of contact forms such
  that $\pm \alpha_\pm \wedge d\alpha_\pm^{n-1} > 0$, and the
  $1$-form
  \begin{equation*}
    \beta := e^{-s} \alpha_- + e^{s} \alpha_+
  \end{equation*}
  on $\RR \times M$ satisfies $d\beta^n > 0$.
\end{defnIntro}

A Liouville pair allows us to construct Liouville domains 
with two boundary components (in fact, by attaching Stein
$1$-handles to these examples one can obtain 
examples with any number of boundary components).
These manifolds can then be used to build Giroux domains of the form
$[0,\pi] \times \SS^1\times M$ with contact form
\begin{equation}\label{eqIntro:HigherDimTorison}
  \lambdaGT = \frac{1 + \cos s}{2}\, \alpha_+ + 
  \frac{1 - \cos s}{2}\, \alpha_- + \sin s \; dt \;,
\end{equation}
which can be stacked together to produce the examples described in
Theorem~\ref{thm:generalisationTori}.
In order to state an existence result%
\footnote{Our proof of Theorem~\ref{thm:LiouvilleExists} owes a
considerable debt to Yves Benoist, who explained to us how to use number 
theory to find lattices in the groups considered by Geiges in 
\cite{Geiges_disconnected}.}
for Liouville pairs, recall that
a \defin{number field of degree~$n$} is a field that is an
$n$-dimensional vector space over~$\QQ$.
Recall also that $\RR$ contains number fields of arbitrary
degree.

\begin{main_theorem}\label{thm:LiouvilleExists}
  One can associate canonically to any number field~$\k$ of degree~$n$
  a $(2n - 1)$-dimensional closed contact manifold $(M_\k, \xi_\k)$.
  If $\k$ can be embedded into $\RR$, then $M_\k$ also admits a Liouville
  pair, hence $\RR \times M_\k$ is Liouville.
\end{main_theorem}
\begin{coru}
There exist Liouville domains with disconnected boundary in all
even dimensions.
\end{coru}

This corollary provides a source of examples%
\footnote{Actually this construction provides infinitely many examples with 
pairwise distinct fundamental groups. We thank Gaëtan Chenevier for arithmetic
discussions clarifying this.}
that can be plugged into Theorem~\ref{thm:1torsion} to construct
nonfillable contact manifolds in all dimensions, and a special case of
this leads to the examples of 
Theorem~\ref{thm:generalisationTori} as well as the higher dimensional 
version of Giroux torsion
discussed in~\S\ref{sec:torsion}.  The proof of Theorem~\ref{thm:1torsion}
is in fact a generalization to higher
dimensions of a construction that was used by the third author in
\cite{WendlCobordisms} to show that every contact $3$-manifold with
Giroux torsion is weakly symplectically cobordant to one that is
overtwisted. In higher dimensions, the overtwistedness will come from a
generalization of the work of Atsuhide Mori in \cite{Mori_Lutz}.
Note that already in dimension three, the cobordism argument requires the
fact that overtwistedness obstructs \emph{weak} (not only strong)
fillability, a notion that has not previously been defined in any
satisfactory way in higher dimensions.  In dimension three of course,
the subtle differences between weak and strong fillings are of interest
in themselves, not only as a tool for understanding strong fillability.

As preparation for the definition of weak fillability that we will propose
here, let us first have a look at the realm of (almost) complex manifolds.

\begin{defnIntro}\label{def:tamed_pseudo_convex}
One says that a contact manifold $(V, \xi)$ is the \defin{tamed pseudoconvex
boundary} of an almost complex manifold $(W, J)$ if $V = \p W$ and
  \begin{itemize}
  \item $\xi$ is the hyperplane field $TV \cap JTV$ of $J$-complex
    tangencies,
  \item $W$ admits a symplectic form~$\omega$ taming $J$, and
  \item $V$ is $J$-convex.
  \end{itemize}
  The last point means that if we orient $V$ as the boundary of $W$,
  then for any $1$-form~$\lambda$ defining~$\xi$ (i.e.~$\lambda$ is a
  $1$-form with $\xi = \ker \lambda$ as oriented hyperplanes), we
  have $d\lambda(v, Jv) > 0$ for every nonzero vector $v \in \xi$.
\end{defnIntro}

Note that there is no direct relation in the definition between
the taming form~$\omega$ and the contact structure~$\xi$.
It must also be pointed out that the existence of $(W, J)$ is not very
restrictive without the taming condition.
For instance, the overtwisted contact structure on $\SS^3$ that is
homotopic to the standard contact structure can be realized as a
pseudoconvex boundary of the ball for some almost complex structure,
but the Eliashberg-Gromov theorem implies that this structure can
never be tamed.
We now recall the standard definitions on the symplectic side.
\begin{defnIntro}
\label{defn:sympFilling}
Let $V$ be a closed oriented manifold with a positive and co-oriented
contact structure~$\xi$.  We say that a compact symplectic manifold
$(W,\omega)$ is a \defin{symplectic filling} of $(V,\xi)$ if
$\p W = V$ as oriented manifolds and $\omega$ admits a primitive
$\lambda$ (a \defin{Liouville form}) near~$\p W$ which restricts to~$V$ as 
a contact form for~$\xi$.  We call $(W,\omega)$ an \defin{exact filling}
of $(V,\xi)$, or a \defin{Liouville domain}, if the Liouville form~$\lambda$
extends globally over~$W$.
\end{defnIntro}
Note that a Liouville form $\lambda$ gives rise (via the $\omega$-dual)
to a \defin{Liouville vector field}, whose flow is a symplectic dilation,
and the condition that $\restricted{\lambda}{TV}$ be a positive
contact form means that the Liouville vector field points transversely
outward at the boundary.  For this reason we say in this case that
$(W,\omega)$ has (symplectically) \defin{convex} boundary.
In dimension three, it is customary to distinguish this notion from
the weaker version discussed below by calling $(W,\omega)$ a
\defin{strong filling} of $(V,\xi)$, and we shall also apply this
convention to higher dimensions in the present paper.
To obtain a weaker notion of symplectic filling, recall that every
co-oriented contact structure $\xi$ carries a natural conformal class
$\CS_\xi$ of symplectic structures: indeed, if $\lambda$ is any
contact form for~$\xi$, then $\restricted{d\lambda}{\xi}$ defines a
symplectic bundle structure that is independent of the
choice of~$\lambda$ up to scaling.
If $(W, \omega)$ is a symplectic manifold and $V = \p W$ carries a
positive contact structure~$\xi$, one says, following
\cite{EliashbergGromov_convex}, that $\omega$
\defin{dominates}~$\xi$ if the restriction $\omega_\xi :=
\restricted{\omega}{\xi}$ belongs to~$\CS_\xi$.
This is always the case if $(W,\omega)$ is a strong filling of $(V,\xi)$,
and in dimension three it defines a strictly weaker notion of symplectic
fillability, e.g.~it is obstructed by overtwistedness but not by Giroux
torsion.
A symplectic $4$-manifold $(W,\omega)$ dominating a contact structure~$\xi$
at its boundary $V = \p W$ is therefore called a \defin{weak filling}
of $(V,\xi)$.
However, McDuff proved \cite[Lemma~2.1]{McDuff_contactType} that
from dimension~$5$ upward, the dominating condition already
implies that $(W,\omega)$ is a \emph{strong} filling.
In this paper, we propose the following weak filling condition for all 
dimensions.

\begin{defnIntro}
\label{def:weakfilling}
  Let $\xi$ be a co-oriented contact structure on a manifold $V$.
  Denote by $\CS_\xi$ the canonical conformal class of symplectic
  structures on $\xi$.
  Let $(W,\omega)$ be a symplectic manifold with $\p W = V$ as oriented
  manifolds and
  denote by $\omega_\xi$ the restriction of $\omega$ to $\xi$.
  We say that $(W, \omega)$ is a \textbf{weak filling} of $(V,\xi)$ 
	(and $\omega$ \defin{weakly dominates} $\xi$)	if
  $\omega_\xi$ is symplectic and $\omega_\xi + \CS_\xi$ is a ray of symplectic
	structures on~$\xi$.
\end{defnIntro}

The weak filling condition is thus equivalent to the requirement that
\begin{equation*}
\alpha \wedge \bigl(d\alpha + \omega_\xi\bigr)^{n - 1} \text{ and } 
\alpha \wedge \omega_\xi^{n - 1}
\end{equation*}
should be positive
volume forms for every choice of contact form~$\alpha$ for $\xi$. 
If one fixes a contact form~$\alpha$, then this is equivalent
to requiring $\alpha \wedge (\omega_\xi + \tau\, d\alpha)^{n-1} > 0$
for all constants $\tau \ge 0$, and it holds
for instance whenever
$$
\alpha \wedge d\alpha^k \wedge \omega_\xi^{n - 1 - k} > 0
$$
for all $k \in \{0,1,\dotsc,n-1\}$.
In dimension three, weak domination is equivalent to domination, hence our 
definition of weak filling reduces to the standard one.
The first important result to state about this new definition is that it is
the purely symplectic counterpart of tamed pseudoconvex boundaries.\footnote{We
  are deeply indebted to Bruno Sévennec and Jean-Claude Sikorav for
  discussions that led to the proof of Theorem~\ref{thm:omegaJ}.}

\begin{main_theorem}\label{thm:omegaJ}
A symplectic manifold $(W,\omega)$ is a weak filling of a contact
manifold $(V,\xi)$ 
(Definition~\ref{def:weakfilling}) if and only
if it admits a smooth almost complex structure~$J$ that is tamed by~$\omega$
and makes $(V,\xi)$ the tamed pseudoconvex boundary of $(W,J)$ 
(Definition~\ref{def:tamed_pseudo_convex}).
\end{main_theorem}

By contrast, weak fillings are not automatically strong fillings.
Indeed, weak domination of a fixed $\xi$ is an open
condition on $\omega$, so one can easily construct weak fillings that are
non-exact at the boundary by taking small perturbations of strong fillings.
The following less trivial examples of weak fillings non-exact at
the boundary are inspired by
Giroux's construction \cite{Giroux_plusOuMoins} of weak fillings for the tight
contact structures~$\xi_k$ on~$\TT^3$.

\begin{exIntro}
  Starting from a closed contact manifold $(V,\xi)$ and a supporting
	open book decomposition \cite{Giroux_ICM}, Frédéric Bourgeois constructed in
  \cite{BourgeoisTori} a contact structure on $V \times \TT^2$.
It can be written as the kernel of the $1$-form
\begin{equation*}
  \alpha_\epsilon = \alpha_V + \epsilon f\, dx_1 + \epsilon g\, dx_2
\end{equation*}
for any $\epsilon > 0$, where $(x_1,x_2)$ are the coordinates on
$\TT^2 = \SS^1 \times \SS^1$, $\alpha_V$ is a contact form on $V$ compatible
with the given open book, and $f,g\colon V \to \RR$ are functions
associated to the open book.
Now if $(W,\omega)$ is a weak filling of $(V,\xi)$, one can check by
examining the limit $\epsilon \to 0$ that
the Bourgeois contact structure on $V \times \TT^2$ is weakly filled by
$(W\times \TT^2,\omega \oplus \omega_{\TT^2})$, where $\omega_{\TT^2}$
is an area form on~$\TT^2$.
\end{exIntro}

The next result extends the fact that weak fillability is \emph{strictly}
weaker than strong fillability beyond dimension three.  Though we prove
this only for dimension five, it is presumably true in all dimensions; see
\S\ref{sec:torsion} for further discussion.

\begin{main_theorem}\label{thm:exist_weak_not_strong}
There exist $3$-manifolds $M$ with Liouville pairs $(\alpha_+,\alpha_-)$ 
such that the contact manifolds
$(\TT^2 \times M,\xi_k)$ of Theorem~\ref{thm:generalisationTori}
are all weakly fillable.  In particular,
there exist contact $5$-manifolds that are weakly but not
strongly fillable.
\end{main_theorem}

As in dimension three, one should expect that any notion of ``overtwistedness''
one might define in higher dimensions obstructs
the existence of a weak filling.  Here we have two possible notions
in mind: recall first that the second author \cite{NiederkruegerPlastikstufe}
has introduced 
a higher dimensional generalization of the overtwisted disk, called the
\emph{plastikstufe}.  We shall introduce in \S\ref{sec:blobs} a natural
generalization of this, called a \emph{bordered Legendrian open book}
(or ``\BLOB'' for short), and refer to contact manifolds that contain such
objects as \emph{$PS$-overtwisted}.  An alternative (though not necessarily
inequivalent) notion emerges from the observation that a
contact $3$-manifold is overtwisted if and only if it has a supporting
open book that is the negative stabilization of another open book.
The corresponding condition in higher dimensions is known to imply
\emph{algebraic} overtwistedness, i.e.~vanishing contact homomology
\cite{BourgeoisContactHomologyLeftHanded}.  We will show that each of
these conditions gives an obstruction to semipositive\footnote{In
Theorem~\ref{thm:non_weakly_fillable} and several other results in this
paper, we write the word ``semipositive'' in parentheses: this means that
the condition is presently necessary for technical reasons, but should be
removable in the future using the polyfold technology of Hofer-Wysocki-Zehnder,
cf.~\cite{Hofer_polyfoldSurvey}.
Note that in dimensions~$4$ and~$6$, symplectic manifolds are always semipositive.}
weak fillings:
\begin{main_theorem}
\label{thm:non_weakly_fillable}
  If $(V,\xi)$ is a closed contact manifold that either
  \begin{itemize}
  \item [(i)] contains a contractible $PS$-overtwisted subdomain, or
  \item [(ii)] is obtained as the negative stabilization of an open book,
  \end{itemize}
  then $(V, \xi)$ has no (semipositive) weak filling.

  Hence any contact structure on a closed manifold $V$ with $\dim V \ge 3$
  can be modified within its homotopy class of almost contact structures
  to one that admits no (semipositive) weak fillings.
\end{main_theorem}

We will also show in \S\ref{section:weak_condition} that the weak filling
condition is conveniently amenable to deformations near the boundary.
An often used fact in dimension three, due originally to Eliashberg
\cite{EliashbergContactProperties}, is that any weak filling which is
exact near the boundary can be deformed to a strong filling.
This was extended in \cite{NiederkrugerOvertwistedAnnulus} to show
that every weak filling can be deformed to make the boundary a
\emph{stable hypersurface}, so that weak fillings can be studied using
the machinery of Symplectic Field Theory (SFT).
Extending this idea to higher dimensions led to the notion of a
\emph{stable symplectic filling} defined in \cite{LatschevWendl}, and
we will show:

\begin{propIntro}\label{prop:stabilizeBoundary}
  Any weak filling can be deformed near its boundary to a stable
  filling.
  Moreover, if the symplectic form is exact near the boundary, then it
  can be deformed to a strong filling.
\end{propIntro}
The fact that weak fillings can be ``stabilized'' means that they are
obstructed by the invariants defined in \cite{LatschevWendl}, known as
\emph{algebraic torsion}.  The following corollary, which we will not
use in this paper, comes of course with the standard caveat about the
analytical foundations of SFT:
\begin{corIntro}\label{cor:AT}
If $(V,\xi)$ has fully twisted algebraic torsion in the sense of
\cite{LatschevWendl}, then it is not weakly fillable.
In particular, this is the case if $(V,\xi)$ has vanishing contact
homology with fully twisted coefficients.
\end{corIntro}

The contact structures defined in
\eqref{eqIntro:HigherDimTorison} can be used to define a higher
dimensional version of the standard $3$-dimensional Lutz twist
along a pre-Lagrangian torus.
Notably, whenever $(V,\xi)$ contains a hypersurface~$H$ that is
isomorphic to one of the boundary components of the domain $[0,2\pi]
\times \SS^1\times M$ with the contact structure given by $\lambdaGT$,
we can cut $V$ open along $H$ and glue in an arbitrary number of
such domains to modify the contact structure on~$V$.
The contact structure obtained from this operation will never be strongly
fillable, and in some cases it is not even weakly fillable:

\begin{main_theorem}\label{thm:stacking}
By inserting contact domains of the form
$([0,2\pi k] \times \SS^1\times M,\ker\lambdaGT)$ for various $k \in \NN$, 
one can construct closed manifolds
in any dimension $2n-1 \ge 3$ which admit
infinite families of hypertight but not weakly fillable contact structures 
that are homotopic as almost
contact structures but not contactomorphic.
\end{main_theorem}

We will also discuss in \S\ref{sec:torsion} a ``blown down'' version of the 
above operation, which generalizes both the classical Lutz twist along
transverse knots in dimension three and a $5$-dimensional version recently
introduced by A.~Mori \cite{Mori_Lutz}.  As we shall see, this
operation always produces a contact structure that is in the same homotopy
class of almost contact structures, but is $PS$-overtwisted and thus
not weakly fillable. See also \cite{EtnyreGeneralizedLutzTwists}
for a completely different generalization of the Lutz twist to higher
dimensions.

\subsection*{Organization}
Here is an outline of the remainder of the paper.

In \S\ref{section:weak_condition} we establish some basic properties of
the weak filling condition, including its relation to tamed pseudoconvexity and
behavior under deformations in collar neighborhoods.  
This includes the proofs of
Theorem~\ref{thm:omegaJ} and Proposition~\ref{prop:stabilizeBoundary}.

Section~\ref{sec:negStab} shows that weak fillings are obstructed by
negatively stabilized open books.  The technology here involves finite energy
holomorphic planes in the noncompact completion of a weak filling; it is
a minor adaptation of the contact
homology computation due to Bourgeois and van Koert 
\cite{BourgeoisContactHomologyLeftHanded}.  Instead of appealing to contact
homology, however, we argue directly that the moduli space of holomorphic
planes found in \cite{BourgeoisContactHomologyLeftHanded} cannot exist if there
is a semipositive weak filling.

In \S\ref{sec:blobs}, we introduce the \BLOB as 
a natural generalization of the
plastikstufe and adapt the standard ``Bishop family of holomorphic disks''
argument to prove the remainder of 
Theorem~\ref{thm:non_weakly_fillable}.

The next three sections establish the proof of Theorem~\ref{thm:1torsion},
defining the first higher dimensional filling obstruction that is
distinct from any notion of ``overtwistedness''.  
In \S\ref{section:giroux_domains}, we discuss ideal Liouville domains and
Giroux domains, and state a more precise version of
Theorem~\ref{thm:1torsion} that can also be applied to weak fillings.
The proof requires a surgery construction explained in \S\ref{section:surgery},
which is inspired by the construction in \cite{WendlCobordisms} of symplectic
cobordisms from any contact $3$-manifold with Giroux torsion to one that
is overtwisted.  In our case, we consider a contact manifold $(V,\xi)$ which
contains a region with nonempty boundary consisting of two Giroux 
domains $G_0 = \Sigma_0 \times \SS^1$ and~$G_1 = \Sigma_1 \times \SS^1$
glued together.  It turns out that one can attach along~$G_0$ a 
symplectic ``handle'' 
of the form $\Sigma_0 \times \DD^2$, the effect of which is to replace 
$G_0 \cup G_1$ with
a region that is $PS$-overtwisted, thus a weak filling of $(V,\xi)$ with
suitable cohomological properties at the boundary gives rise to a larger
weak filling of something $PS$-overtwisted and hence a contradiction.
Note that since the new boundary is only weakly filled in general, the
new notion of weak fillability plays a crucial role even just for proving
that $(V,\xi)$ is not \emph{strongly} fillable.  
We shall also provide in \S\ref{sec:applications} an alternative argument
that avoids holomorphic disks and uses the somewhat simpler technology
of closed holomorphic spheres; this allows us to overcome transversality
problems using the recently developed polyfold machinery
\cite{HoferWZ_GW}.

In \S\ref{sec:Liouville} we switch gears and address the existence of
Liouville pairs in all dimensions, proving Theorem~\ref{thm:LiouvilleExists}.
For this we borrow an idea of Geiges from \cite{Geiges_disconnected} to
look for Liouville pairs among left-invariant $1$-forms on noncompact 
Lie groups that admit co-compact lattices and hence compact quotients.
Our examples of left-invariant Liouville pairs on Lie groups are quite
easy to write down (see e.g.~Equation \eqref{eqn:LiouvillePair}), but
in order to find co-compact lattices we'll need to apply some basic algebraic 
number theory.

Finally, \S\ref{sec:torsion} explains the most important special cases
of the filling obstruction from Theorem~\ref{thm:1torsion},
leading to higher dimensional generalizations of Giroux torsion and
the Lutz twist.  From this follow the proofs of
Theorems~\ref{thm:generalisationTori}, \ref{thm:exist_weak_not_strong}
and~\ref{thm:stacking}.

The appendix contains some technical results in symplectic linear algebra needed
for the proof of Theorem~\ref{thm:omegaJ}, relating weak symplectic fillings and
tamed pseudoconvexity.

\subsection*{Notation}

Unless otherwise indicated, throughout this paper 
we will assume $(W,\omega)$ is a compact 
symplectic manifold of dimension $2n \ge 4$, and $(V,\xi)$ is a closed
$(2n-1)$-dimensional contact manifold, with $\xi$ positive and
co-oriented.  In cases where $V$ is identified with $\p W$, we assume
that this identification matches the orientation induced by~$\xi$ to
the natural boundary orientation determined by~$\omega$.  
Also when $V = \p W$, we will often use the abbreviations
$$
\omega_V := \restricted{\omega}{TV} \qquad\text{ and }\qquad
\omega_\xi := \restricted{\omega}{\xi}.
$$

\subsection*{Acknowledgments}

We are grateful to Bruno Sévennec and Jean-Claude Sikorav for e-mails leading to
the proof of Theorem \ref{thm:omegaJ}, Yves Benoist for conversations which were
crucial for the proof of Theorem \ref{thm:LiouvilleExists}, Sylvain Courte for
his proof of Lemma~\ref{lemma:idealCollar}, Yves de Cornulier for his proof of
Lemma~\ref{lemma:deCornulier}, Helmut Hofer and Joel Fish for
explaining to us some details of the polyfold machinery, and Paolo Ghiggini for
many helpful discussions at the beginning of this project.
The idea that some modification of Mori's ideas in \cite{Mori_Lutz}
might lead to a notion of Giroux torsion in higher dimensions was first
suggested to us by John Etnyre.  We would also like to
thank the mathematics department in Nantes for
creating a pleasant working environment which hosted several meetings of the
authors, and a very careful anonymous referee whose comments on the
original version of this article have led to several improvements in the
exposition.
The first and second author were partially supported by the ANR grant
\emph{ANR-10-JCJC 0102}. The third author was supported by an
Alexander von Humboldt Foundation fellowship.

\section{The weak filling condition}
\label{section:weak_condition}

\subsection{Pseudoconvexity and weak filling}

The aim of this section is to show that our definition of a weak
filling (Definition~\ref{def:weakfilling} in the introduction) 
is in a certain sense the
purely symplectic counterpart of a tamed almost complex manifold with
pseudoconvex boundary.

Before proving the main theorem on this subject, we will need some important
properties of complex structures on vector spaces which were
explained to us by Bruno Sévennec and Jean-Claude Sikorav.
We will give proofs of the following two propositions in
Appendix~\ref{section:contractibility_of_space_J} and
\ref{section:linear_algebra} respectively.

\begin{proposition}\label{prop:space_cotamed_contractible}
  The space of complex structures on a vector space $E$ tamed by two
  given symplectic forms $\omega_0$ and $\omega_1$ is either empty or
  contractible.
\end{proposition}

\begin{proposition}\label{prop:carac_cotamed}
  Let $E$ be a real vector space equipped with two symplectic forms
  $\omega_0$ and $\omega_1$.
  The following properties are equivalent:
  \begin{enumerate}
  \item the linear segment between $\omega_0$ and $\omega_1$ consists
    of symplectic forms
  \item the ray starting at $\omega_0$ and directed by $\omega_1$
    consists of symplectic forms
  \item there is a complex structure $J$ on $E$ tamed by both
    $\omega_0$ and~$\omega_1$.
  \end{enumerate}
\end{proposition}

\begin{remark}\label{rk:cocompatible}
  When choosing an almost complex structure~$J$ on a symplectic manifold,
for most applications it makes no difference whether one requires~$J$ to be
\emph{calibrated} (i.e.~compatible with) or \emph{tamed} by the symplectic 
structure, and typically very little attention is paid to this distinction
in the literature.  
Note however that the \emph{cotaming} condition is strictly weaker
than \emph{cocalibrating}, and in many cases it is not possible to
require the latter.
For instance, one can prove  (by hand or using the previous proposition)
that there exists a complex
structure on $\RR^4$ that is cotamed by the two forms
$\omega_0 = dx_1 \wedge dx_3 + dx_2 \wedge dx_4$ and $\omega_1 =
dx_2 \wedge dx_1 + dx_3 \wedge dx_4$.
On the other hand, one can use the fact that $\omega_0 \wedge \omega_1 =  0$
to show that there does not exist any complex structure that
is both calibrated by~$\omega_0$ and tamed by~$\omega_1$.
\end{remark}

The following is a restatement of Theorem~\ref{thm:omegaJ} from
the introduction.

\begin{theorem}
    A symplectic manifold $(W, \omega)$ is a weak filling of $(\p W,
    \xi)$ if and only if there is an almost complex structure $J$ on
    $W$ which is tamed by $\omega$ and such that $(\p W, \xi)$ is the
    strictly pseudoconvex boundary of $(W, J)$.
\end{theorem}
\begin{proof}
  We denote the boundary of $W$ by $V$ and use the notation of the
  introduction.
  Suppose we have a weak filling.
  From Proposition~\ref{prop:carac_cotamed}, using the fact that
  the cotaming property is open, it follows that every point in the
  manifold $V$ has a small neighborhood on which there exists a
  complex structure $J_\xi$ on $\xi$ which is tamed by both
  $\omega_\xi$ and~$\CS_\xi$.  Using the contractibility of the space of such
  $J_\xi$'s (Proposition \ref{prop:space_cotamed_contractible}), we can then
replace $J_\xi$ with a global complex structure on~$\xi$
that has this property.
  Choose any vector field~$X$ on $V$ that spans $\ker \omega_V$, and
  extend it to a collar neighborhood $U$ of $V$.
  Let $Y$ be a vector field on $U$ that lies along $V$ in the
  $\omega$-orthogonal complement of $\xi$ and that satisfies
  $\omega(X, Y) > 0$.
  We extend $J_\xi$ to an almost complex structure $J$ on $U$ by
  setting $JX = Y$.
  Clearly, $J$ is tamed by $\omega$ on a small neighborhood of $V$,
  and we can then extend $J$ to the interior of $W$ to obtain the
  desired tamed almost complex structure on the entire filling $W$.
  By construction, $\xi = TV \cap JTV$, and $V$ is strictly
  $J$-pseudoconvex since $J_\xi$ is tamed by $\CS_\xi$.

  Conversely, assume $W$ has an almost complex structure~$J$ 
  that is tamed by $\omega$ and makes the boundary strictly pseudoconvex,
with $\xi$ as the field of complex tangencies $TV \cap J TV$.
  We can then write $\xi$ as
  the kernel of a nonvanishing $1$-form $\alpha$, and
  pseudoconvexity implies that we can choose the sign of $\alpha$ in
  such a way that $\restricted{d\alpha}{\xi}$ tames
  $\restricted{J}{\xi}$, and such that the natural orientation of~$\xi$
together with its co-orientation defined via $\alpha$ is compatible with
the boundary orientation of~$W$.
Since $\omega$ tames $J$, $\omega_\xi$ also 
tames $\restricted{J}{\xi}$.  We therefore have
  cotaming forms on $\xi$, so the easy implication (3)~$\implies$~(2)
  of Proposition~\ref{prop:carac_cotamed}
guarantees that $(W,\omega)$ is a weak filling of~$(V,\xi)$.
\end{proof}

Suppose now $U$ is a domain inside a symplectic manifold $(W, \omega)$ and 
$V := \partial U$ is pseudoconvex for some tamed $J$.
Using the easy direction of the preceding theorem, we see that 
$(U, \restricted{\omega}{U})$ is a weak filling of $(V, TV \cap JTV)$.
It is not true in general that it is a strong filling.
This was observed first in \cite[p.~158]{EliashbergGromov_convex},
where an example in $\CC^n$ with its standard Kähler structure is
discussed.
In this example, Eliashberg proved that the relevant contact structure
is actually Stein fillable, but weak fillability is much easier to
check (recall that we used the easy direction).  
By Theorem~\ref{thm:non_weakly_fillable} in the introduction, 
this already implies global
information about the contact structure, such as the nonexistence of a
contractible $PS$-overtwisted subdomain, or of
a negatively stabilized supporting open book.

\subsection{Magnetic collars and cones}
\label{subsection:magnetic}

Recall that for any co-oriented hyperplane field $\xi$ on a manifold
$V$, one can consider the annihilator of $\xi$ in $T^*V$:
\begin{equation*}
  S\xi := \bigl\{ \lambda \in T^*V\bigm|\,
  \text{$\ker \lambda = \xi$ and $\lambda (v) > 0$
    if $v$ is positively transverse to $\xi$}\bigr\} \;.
\end{equation*}
The field $\xi$ is a contact structure if and only if $S\xi$ is a
symplectic submanifold of $(T^*V, \ocan)$, and in this case $S\xi$ is
called the \emph{symplectization} of $\xi$.
Any contact form $\alpha$ is a section of this $\RR^*_+$-bundle, and thus
determines a trivialization $S\xi \cong \RR_+^* \times V$. 
In this trivialization, the restriction of the canonical symplectic form
$\ocan$ becomes $d(t\alpha)$, where $t$ is the coordinate in
$\RR^*_+$.

In order to rephrase the definition of weak filling in these terms,
we need to recall one further notion.
Suppose $\omega_V$ is any closed $2$-form on $V$, and denote the
projection from $T^*V$ to $V$ by $\pi$.
The $2$-form $\ocan + \pi^*\omega_V$ is then a symplectic form on $T^*V$,
which is called \defin{magnetic}.

The definition of weak fillings can now be reformulated as follows.

\begin{lemma}\label{lem:weak equals magnetic}
  Let $(W, \omega)$ be a symplectic manifold with $\p W = V$.
  Denote by $\omega_V$ the restriction of $\omega$ to $TV$ and by
  $\omega_\xi$ its restriction to a contact structure $\xi$ on $V$.
  The manifold $(W, \omega)$ is a weak filling of $(V, \xi)$ if and
  only if $\omega_\xi$ is symplectic and $S\xi$ is a symplectic
  submanifold of the magnetic cotangent bundle associated to
  $\omega_V$. \qed
\end{lemma}

In the case where $(W,\omega)$ strongly fills $(V,\xi)$, it admits a
Liouville vector field $X$ near $V$, which induces the contact form
$\alpha = \restricted{\iota_X\omega}{TV}$ on~$V$.  
Let $\flow_t$ denote the flow of $X$
for time~$t$.  For sufficiently small $\epsilon > 0$, the map 
$(t, m) \mapsto \flow_{\ln t}(m)$ embeds 
$\bigl((1 - \epsilon, 1]\times V, \ocan\bigr)$ symplectically into $W$.
This allows the completion of $W$ by adding the positive half 
$(1, \infty)\times V$ of $S\xi$.
To understand this from a magnetic point of view, observe that 
$\omega_V = d\alpha$, so the magnetic form on 
$S\xi$ is $\omega_V + d(t\alpha) =
d((t+1)\,\alpha)$.  Thus $(t,m) \mapsto (t + 1,m)$ is a
symplectomorphism from the magnetic symplectization to the cylindrical end
of the completed strong filling.

In the setting of weak fillings, we would similarly like to be able to complete 
$(W, \omega)$ by adding the
magnetic symplectization. For this we need
a suitable description of a collar neighborhood of the boundary:
the following lemma has an obvious analogue for the situation where~$V$
is an oriented boundary component of a symplectic manifold $(W,\omega)$.

\begin{lemma}\label{lemma:model_tubular_neighborhood}
Suppose $V \subset W$ is an oriented hypersurface in the interior of a
$2n$-dimensional symplectic
  manifold $(W,\omega)$, $\xi \subset TV$ is the co-oriented (and hence
also oriented) hyperplane distribution induced by a nowhere vanishing
$1$-form $\lambda$ on~$V$,
  and the restriction of $\omega$ to $\xi$ is symplectic and induces the
positive orientation.
  Then a neighborhood of~$V$ in $(W,\omega)$ is symplectomorphic to
  \begin{equation*}
    \bigl( (-\epsilon,\epsilon) \times V,\,
    d(t\lambda) + \omega_V \bigr) \;,
  \end{equation*}
for some $\epsilon > 0$, where $\omega_V := \restricted{\omega}{TV}$,
$V$ is identified in the natural way with $\{0\}\times V$, and
  the direction of $\partial_t$ is such that
  $\iota_{\partial_t}\omega^{n} = \lambda \wedge \omega^{n-1}$.
Moreover, the vector field $\partial_t$ in~$W$ can be chosen to extend any given
vector field which has these properties on a neighborhood of 
some part of~$V$.
\end{lemma}
\begin{remark}
The statement about the direction of $\p_t$ means that in the version of this
lemma for the boundary of a weak filling, one obtains a neighborhood
of the form $((-\epsilon,0] \times V , d(t\lambda) + \omega_V)$, so in
particular~$\p_t$ points outwards.  There is a corresponding variation for
negative boundary components of weak symplectic cobordisms, for which~$\p_t$
points inwards.
\end{remark}
\begin{proof}[Proof of Lemma~\ref{lemma:model_tubular_neighborhood}]
  An identical proof has been given for the $3$-dimensional case in
  \cite{NiederkrugerOvertwistedAnnulus}.
  We will first define a collar neighborhood of $V$ by choosing a
  vector field that is transverse to $V$.
  Let $E\subset \restricted{TW}{V}$ be the $\omega$-orthogonal complement
  of $\xi$ along $V$.
  The intersection of $E$ with $TV$ is a $1$-dimensional
  subbundle, and we can uniquely define a $\emph{Reeb-like}$ vector
  field $X_\omega$ by taking the section in $E\cap TV$ that satisfies
  $\lambda(X_\omega) \equiv 1$.
  By our definition, $\restricted{\omega(X_\omega,\cdot)}{TV} = 0$ holds.
  Choose now a second section $Y$ in $E$ that is transverse to $V$,
  and normalize it such that $\omega(Y,X_\omega) \equiv 1$.
Note that if such a section is already given near some subset of~$V$,
then we can choose $Y$ to be an extension of that section.
We now have $\restricted{\omega(Y,\cdot)}{TV} = \lambda$,
  since both forms vanish on $\xi$ and agree on $X_\omega$.

  Extend $Y$ to a smooth vector field in a neighborhood of $V$, and
  use the flow $\flow^Y$ of this vector field to define a smooth
  diffeomorphism
  \begin{equation*}
    \Phi\colon (-\epsilon, \epsilon) \times V \hookrightarrow W, \,
    (t,p) \mapsto \flow^Y_t(p) \;,
  \end{equation*}
  which agrees with the canonical identification on $\{0\}\times V$.
  Next, compare the $2$-forms $\Phi^*\omega$ and $\omega_V + d(t
  \lambda)$ on $(-\epsilon,\epsilon) \times V$.
  Both forms coincide along $\{0\} \times V$, thus the linear
  interpolation of these forms is a path of symplectic structures
  (decreasing $\epsilon>0$ if necessary).  We can then use the Moser trick
  to show that they are all symplectomorphic to each other
(perhaps in a smaller neighborhood) by an isotopy that keeps the level
  set $\{0\} \times V$ fixed.
\end{proof}

\begin{corollary}\label{cor:magnetic_completion}
  If $(W,\omega)$ is a weak filling of $(V, \xi)$, then
  one can extend $W$ to a magnetic completion 
	$(\widehat W,  \widehat\omega)$ with $\widehat W = W \cup S\xi$,
  $\restricted{\widehat \omega}{W} = \omega$ and $\restricted{\widehat
    \omega}{S\xi} = \ocan + \omega_V$.
  
  Moreover, for every positive $t$, ($\{t\}\times V,\xi)$ is then weakly
  filled by $W \cup (0,t]\times V$ equipped with the restriction of
  $\widehat{\omega}$.  \qed
\end{corollary}

In the previous section we proved that whenever
$(W, \omega)$ is a weak filling of $(V, \xi)$, there is a~$J$ on $\xi$ which
is tamed by $CS_\xi$ and also by the restriction of~$\omega$. However, 
it is sometimes desirable to fix a complex structure on $\xi$ in advance.
The following observation allows us to do this, at the price of first adding a
sufficiently large part of the magnetic completion.  The proof is a short
computation using the fact that for $T \gg 0$, the restriction of
$\omega_V + d(t\alpha)$ to $\{T\} \times V$ is dominated by the second
term.

\begin{lemma}
\label{lemma:Jalpha}
Suppose $\omega_V$ is a closed $2$-form on $V$ weakly dominating a 
contact structure~$\xi$, $\alpha$ is a contact form for~$\xi$ and
$R_\alpha$ is its Reeb vector field.  Further,
suppose $J$ is an almost complex structure on $[0, \infty) \times V$ which
preserves $\xi$ such that $\restricted{J}{\xi}$ is tamed by 
$\restricted{d\alpha}{\xi}$ and
$J\partial_t = R_\alpha$, with $t$ denoting the coordinate on $[0, \infty)$.
Then there exists a number $T \ge 0$ 
such that $J$ is tamed by $\omega_V + d(t\alpha)$ on 
$[T, \infty) \times V$. \qed
\end{lemma}

\subsection{Deformations of weak fillings}

We now want to deform completions of weak fillings in order to obtain
some flexibility for~$\omega_V$.

\begin{lemma}\label{thm:cohomologous_deformation_positive_end}
Let $\omega_V$ be a closed $2$-form weakly dominating a contact structure
$\xi = \ker \alpha$ on $V$, and suppose $\omega_V'$ is 
any closed $2$-form on $V$ that is cohomologous to $\omega_V$.
Then the symplectic structure 
$\omega_V + d(t\alpha)$ on $[0,\infty) \times V$ can
be deformed away from $\{0\} \times V$ so that it coincides with 
$\omega_V' + d(t\alpha)$ on $(t_1, \infty) \times V$ for some large 
number $t_1 > 0$ and all levels $(\{t\} \times V,\xi)$ remain weakly filled.
\end{lemma}
\begin{proof}
Since $\omega_V$ and $\omega_V'$ are cohomologous, there exists a
$1$-form $\beta$ on $V$ such that $\omega_V' = \omega_V + d\beta$.
Consider the closed $2$-form $\omega' = d(t\alpha) + \omega_V +
d(\rho\, \beta)$ on $[0,\infty)\times V$, where $\rho\colon
[0,\infty) \to [0,1]$ is a smooth monotone function that is equal
to $0$ near $t=0$ and to~$1$ for large values of~$t$.
We now show that, if the support of $\rho$ is sufficiently far away from $0$
and $\rho$ increases sufficiently slowly, the new structure $\omega'$ will be
symplectic.
Since it is closed by construction, we only need to check nondegeneracy.
We compute:
\begin{equation*}
  (\omega')^n = dt \wedge (\alpha + \rho'\,\beta) \wedge 
  \bigl(t\,d\alpha + \omega_V + \rho\, d\beta\bigr)^{n - 1} \;.
\end{equation*}
To prove that $t\, d\alpha + \omega_V + \rho\, d\beta$ is a symplectic
form on $\ker(\alpha + \rho'\,\beta)$, choose an auxiliary norm on
the space of differential forms on $V$, and set $c_1 := \norm{\beta}$
and $c_2 := \norm{\omega_V} + \norm{d\beta}$.

The map $\Omega^1(V) \times \Omega^2(V) \to \Omega^{2n+1}(V), \,
(\gamma, \eta) \mapsto \gamma \wedge \eta^n$ is continuous, so that we
find constants $\epsilon_1,\epsilon_2 > 0$ such that $\gamma \wedge \eta^n >
0$ for every pair $(\gamma,\eta) \in \Omega^1(V) \times \Omega^2(V)$
with $\norm{\gamma - \alpha} < \epsilon_1$, and $\norm{\eta - d\alpha} <
\epsilon_2$.
Then for $\eta = d\alpha + \omega_V/t + \rho\, d\beta/t$, we obtain
$\norm{\eta - d\alpha} = \norm{\omega_V/t + \rho\, d\beta/t} \le
c_2/t$, and similarly, we find for $\gamma = \alpha + \rho'\, \beta$
that $\norm{\gamma - \alpha} = \rho'\, c_1$.

The nondegeneracy of $\omega'$ is immediate whenever $t$ lies outside
the support of~$\rho$.
If we let $\rho$ increase sufficiently slowly so that $\rho' <
\epsilon_1 / c_1$ and also assume $\rho (t) = 0$ for $t < c_2/\epsilon_2$,
then the above calculation shows that $\omega'$ is nondegenerate everywhere.
By the same reasoning, every hypersurface $\bigl(\{t\}\times V,
\alpha\bigr)$ will be weakly filled in the new manifold.
\end{proof}

\begin{remark}
This lemma implies that a weak filling gives rise to a strong filling whenever
$\restricted{\omega}{TV}$ is exact. 
This does not mean however that $\omega$ is a weak filling of a unique isotopy
class of contact structures on the boundary---there are counter-examples
in dimension~$3$. As explained for instance in
\cite[Section~4.2]{Massot_GCS}, any Seifert 3-manifold $V$ is the boundary of a 
symplectic manifold $(W, \omega)$ such that $\ker\restricted{\omega}{TV}$
is tangent to the fibers. Thus any (positive) contact structure on $V$ which is
transverse to the fibers is weakly filled by $(W, \omega)$. If $V$ is a
Brieskorn sphere $-\Sigma(2, 3, 6n - 1)$, then the results of \cite{Massot_GCS,
Ghiggini_VHM} combine to prove that there are $n - 2$ isotopy classes of
contact structures transverse to the fibers. Since those manifolds
are homology spheres, $\restricted{\omega}{TV}$ is exact.
\end{remark}

We now make the connection between weak fillings and stable hypersurfaces,
establishing Proposition~\ref{prop:stabilizeBoundary} and hence
Corollary~\ref{cor:AT} from the introduction.

\begin{corollary}\label{cor:weak implies stable}
  Any weak filling $(W, \omega)$ of a contact manifold $(V, \xi)$ can
  be deformed to have the additional property that $\ker \omega_V =
  \ker d\alpha$ for some nondegenerate contact form $\alpha$ for~$\xi$.
  In particular, $(\alpha,\omega_V)$ is then a stable Hamiltonian structure
  on~$V$, and $(W,\omega)$ is a stable filling of $(V,\xi)$ in the
sense of \cite{LatschevWendl}.
\end{corollary}
\begin{proof}
Since weak filling is an open condition, we can perturb $\omega$ so that
without loss of generality it represents a rational cohomology class
in $H^2_{\dR}(V)$.  Then by a result of Cieliebak and Volkov 
\cite[Proposition~2.16]{Cieliebak_Volkov_SHS}, $(V,\xi)$ admits a nondegenerate contact
form $\alpha$ and a $2$-form $\omega_V'$ cohomologous to $\omega_V$
such that the pair $(\alpha,\omega_V')$ define a stable Hamiltonian
structure.  The claim now follows by application of the preceding lemma.
\end{proof}

\section{Negative stabilizations}
\label{sec:negStab}

Corollary~\ref{cor:AT} in the introduction, together with the result of
Bourgeois and van Koert \cite{BourgeoisContactHomologyLeftHanded} that negatively
stabilized contact manifolds have vanishing contact homology (with full
group ring coefficients), implies in principle that such manifolds are not
weakly fillable and always admit contractible Reeb orbits.  
In this section we shall show how the computation
from \cite{BourgeoisContactHomologyLeftHanded} can be modified to produce 
direct proofs of these facts without relying on SFT.

The simplest example of a negatively stabilized contact manifold is the
sphere $(\SS^{2n-1},\xi_-)$ that is supported by the open book with page
$T^*\SS^{n-1}$ and monodromy isotopic to a single negative Dehn-Seidel twist.
By an observation due to Giroux, we may for our purposes define an
arbitrary closed $(2n-1)$-dimensional contact manifold to be 
\defin{negatively stabilized} if and only
if it is the contact connected sum of $(\SS^{2n-1},\xi_-)$ with some other
closed contact manifold.  Our goal is thus
to prove the following:

\begin{theorem}\label{thm:negStab2}
For any closed $(2n-1)$-dimensional contact manifold $(M,\xi)$, the
contact connected sum $(M,\xi) \conSum (\SS^{2n-1},\xi_-)$ has no (semipositive)
weak filling, 
and its Reeb vector fields always admit contractible closed orbits.
\end{theorem}

To prepare the proof, recall that a $1$-form $\lambda$ and closed
$2$-form $\Omega$
on an oriented $(2n-1)$-dimensional manifold~$V$ form a 
\defin{stable Hamiltonian
structure} $(\lambda,\Omega)$ if $\lambda \wedge \Omega^{n-1} > 0$ and
$\ker d\lambda \subset \ker \Omega$.  Such a pair always 
determines a unique vector field~$R$ with the properties $\lambda(R) \equiv 1$
and $\Omega(R,\cdot) \equiv 0$.  Note that if $\lambda$ is also a contact
form, then $R$ is simply the Reeb vector field.  We shall say that an almost
complex structure~$J$ on $\RR\times V$ is \defin{adjusted to} 
$(\lambda,\Omega)$ if it is $\RR$-invariant, maps the unit vector in the
$\RR$-direction to the vector field $R$, and restricts to an $\Omega$-tame
complex bundle structure on $\xi := \ker\lambda$.

\begin{lemma}\label{lemma:CH}
Suppose $(V,\xi)$ is a closed $(2n-1)$-dimensional contact manifold
with nondegenerate contact form $\lambda$ and closed $2$-form $\Omega$ 
such that
$\ker\lambda = \xi$ and $(\lambda,\Omega)$ forms a stable Hamiltonian
structure on~$V$.  Suppose moreover that $\RR\times V$ admits
an almost complex structure $J$ adjusted to $(\lambda,\Omega)$ with the
following properties:
\begin{itemize}
\item There exists a finite energy
$J$-holomorphic plane $u_0 \colon \CC \to \RR\times V$ which is Fredholm regular,
has Fredholm index~$1$
and is asymptotic to a simply covered Reeb orbit~$\gamma$.
\item Other than $\RR$-translations of~$u_0$, $\RR \times V$ admits no
finite energy punctured $J$-holomorphic curves of genus zero
with one positive end asymptotic to~$\gamma$ and no
other positive ends.
\end{itemize}
Then $(V,\xi)$ does not admit any (semipositive)
weak filling $(W,\omega)$ for which
$\restricted{\omega}{TV}$ is cohomologous to~$\Omega$.
\end{lemma}
\begin{proof}
Assume the contrary, that there exists a weak filling $(W,\omega)$
with $[\omega_V] = [\Omega] \in H^2_{\dR}(V)$.
By Lemma~\ref{thm:cohomologous_deformation_positive_end}, we can complete
$(W,\omega)$ to an open symplectic manifold $(\widehat{W},\omega)$ 
by attaching a cylindrical end $([0,\infty) \times V,\omega)$
such that for some $T > 0$, $\omega = \Omega + d(t\lambda)$ on
$[T,\infty) \times V$.  Assign to $(\widehat{W},\omega)$
an $\omega$-tame almost complex structure that matches the given
$\RR$-invariant structure~$J$ on 
$[T,\infty) \times M$ and is generic everywhere else; we shall
denote this extension also by~$J$.
The point of assuming $(\lambda,\Omega)$ to be a stable Hamiltonian structure
is that the compactness results of Symplectic Field Theory
\cite{BourgeoisCompactness} are now valid for finite energy
$J$-holomorphic curves in $(\widehat{W},\omega)$.

The $\RR$-translations of the $J$-holomorphic plane 
$u_0 \colon \CC \to \RR \times V$ asymptotic to the
orbit~$\gamma$ now give rise to a
smooth $1$-dimensional family of $J$-holomorphic curves in
$[T,\infty) \times V \subset \widehat{W}$.  Let $\mM$ denote the 
unique connected component of the
moduli space of unparametrized finite energy $J$-holomorphic curves
in $\widehat{W}$ that contains this family.  All curves in~$\mM$ are planes
asymptotic to the simply covered orbit~$\gamma$ and are thus somewhere
injective.  Let $\mM_+ \subset \mM$ denote the subset consisting of
curves whose images are contained entirely in $[T,\infty) \times V$.
By the uniqueness assumption for~$u_0$, all of these are $\RR$-translations
of~$u_0$, thus $\mM_+ \cong [0,\infty)$.
Then by genericity, all curves in~$\mM \setminus \mM_+$ are also Fredholm
regular, hence $\mM$ is a smooth $1$-dimensional manifold (without boundary).
Observe that $\mM\setminus\mM_+$ is an open subset.
Its closure $\overline{\mM \setminus \mM_+} \subset \mM$ has exactly one
boundary point, the unique curve in $\mM_+$ that
touches $\{T\} \times V$.  

We claim that $\overline{\mM \setminus \mM_+}$
is compact.  Indeed, by \cite{BourgeoisCompactness}, any sequence
$u_k \in \overline{\mM \setminus \mM_+}$ has a subsequence convergent to
a $J$-holomorphic building $u_\infty$ of arithmetic genus~$0$, with one
positive end asymptotic to~$\gamma$ and no other ends.  If $u_\infty$ has
any nontrivial upper level, then the uniqueness assumption implies that
this level can only be an $\RR$-translation of~$u_0$, thus it has no
negative ends and the main level of $u_\infty$ must be empty.  But this can
happen only if $u_k$ has its image in $[T,\infty) \times V$ for large~$k$,
hence $u_k \in \mM_+$, giving a contradiction.  Thus $u_\infty$ has only
a main level, and is at worst a nodal $J$-holomorphic curve in~$\widehat{W}$,
including exactly one component that is a plane asymptotic to~$\gamma$,
while all other components are spheres.  The spheres are ruled out by
semipositivity: since $\dim \mM = 1$, any spheres that could appear
in~$u_\infty$ would necessarily be covers of somewhere injective spheres
with negative index, and thus cannot exist since $J$ is generic.
It follows that $u_\infty$ is a smooth $J$-holomorphic plane, hence
$\overline{\mM \setminus \mM_+}$ is compact as claimed.

The above shows that $\overline{\mM \setminus \mM_+}$ is diffeomorphic
to a compact $1$-dimensional manifold whose boundary is a single point.
Since no such space exists, we have a contradiction and conclude that
the filling $(W,\omega)$ cannot exist.
\end{proof}

For the case $\Omega = d\lambda$, there is a variation on the 
above argument using a trick pioneered by Hofer in \cite{HoferWeinstein}.
Instead of considering a completed filling $(\widehat{W},\omega)$, 
one considers an exact cylindrical symplectic cobordism 
$(\RR \times V,\omega)$ with
$\omega = d(e^t\lambda)$ near $+\infty$ and $d(e^t\lambda')$ near $-\infty$,
where $\lambda'$ may be taken to be a constant multiple of any given
contact form for~$\xi$.  Defining a moduli space of $J$-holomorphic
planes in $\RR\times V$ based on the $\RR$-translations of~$u_0$
as above, the same compactness argument goes through and produces a
contradiction unless planes
bubble off in the negative end, which means $\lambda'$ must admit a
contractible Reeb orbit.
Note that in this case it's even easier to rule out sphere bubbling,
as the exact cobordism $(\RR\times M,\omega)$ does not admit \emph{any}
closed holomorphic curves.  This proves:

\begin{lemma}
\label{lemma:notHypertight}
If the assumptions of Lemma~\ref{lemma:CH} are satisfied with
$\Omega = d\lambda$, then every contact form on $(V,\xi)$ admits a
contractible closed Reeb orbit. \qed
\end{lemma}

\begin{proof}[Proof of Theorem~\ref{thm:negStab2}]
For the case of $\Omega$ exact, \cite{BourgeoisContactHomologyLeftHanded}
establishes precisely the hypotheses of Lemma~\ref{lemma:CH},
thus proving that $(M,\xi) \conSum (\SS^{2n-1},\xi_-)$ is neither strongly
fillable nor (by Lemma~\ref{lemma:notHypertight}) hypertight.
Specifically, Bourgeois and van Koert construct a contact form
and suitable complex structure for $(\SS^{2n-1},\xi_-)$ such that there
is a special Reeb orbit $\gamma$, which has smaller period than all
other Reeb orbits in~$\SS^{2n-1}$, 
and is the asymptotic end of a unique $J$-holomorphic
plane~$u_0$.  In the case of the connected sum $(M,\xi) \conSum (\SS^{2n-1},\xi_-)$,
they also observe that $\gamma \subset \SS^{2n-1}$ can be assumed to have
smaller period than all other Reeb orbits except for a special set
of orbits in the tube connecting $\SS^{2n-1}$ to~$M$, and there can be
no holomorphic curves from~$\gamma$ to any combination of these orbits.
To rule out weak fillings $(W,\omega)$ with arbitrary cohomology
$\beta := [\restricted{\omega}{T(M \conSum \SS^{2n-1})}] \in 
H^2_{\dR}(M \conSum \SS^{2n-1})$,
we now argue as follows.  We can first perturb~$\omega$ to
assume without loss of generality that $\beta$ is a rational cohomology
class.  Let $\beta' \in H^2_{\dR}(M)$ denote the image
of~$\beta$ under the natural isomorphism 
$H^2_{\dR}(M \conSum \SS^{2n-1}) \to
H^2_{\dR}(M)$.  Using the construction in
\cite[Proposition~2.16]{Cieliebak_Volkov_SHS},
we can find a stable Hamiltonian structure $(\lambda',\Omega')$
on~$M$ such that $\ker\lambda' = \xi$, $[\Omega'] = \beta'$ and
$\Omega' = d\lambda'$ outside a tubular neighborhood $\nN(\Sigma)$ 
of a contact submanifold
$\Sigma \subset M$ such that $[\Sigma] \in H_{2n-3}(M)$ is
Poincar\'e dual to a multiple of~$\beta'$.  
The contact form $\lambda'$ may also be
chosen freely outside $\nN(\Sigma)$, and we may assume that the
ball deleted from~$M$ to form the connected sum is disjoint
from~$\nN(\Sigma)$.  The stable Hamiltonian structure 
$(\lambda',\Omega')$ can then be extended over $M \conSum \SS^{2n-1}$ as a 
stable Hamiltonian structure $(\lambda,\Omega)$
such that $[\Omega] = \beta$, and outside of $\nN(\Sigma)$,
$\Omega = d\lambda$ with $\lambda$ an arbitrarily chosen contact
form for $\xi \conSum \xi_-$.
This construction can therefore be arranged to guarantee the same 
essential properties of the orbit~$\gamma$ and curve~$u_0$ as in the 
exact case, thus establishing the hypotheses of Lemma~\ref{lemma:CH}.
\end{proof}

\section{Bordered Legendrian open books}
\label{sec:blobs}

In this section, we will first introduce a generalization of the
plastikstufe that is more natural and less restrictive than the
initial version introduced in \cite{NiederkruegerPlastikstufe}.
Subsequently we will prove that these objects, under a
certain homological condition (which is trivially satisfied for the
overtwisted disk), represent obstructions to weak fillability.

\begin{definition}
Let $N$ be a compact manifold with nonempty boundary.  A \defin{relative open
book} on $N$ is a pair~$(B, \theta)$ where:
\begin{itemize}
\item 
the \defin{binding} $B$ is a nonempty codimension~$2$ submanifold in the
interior of $N$ with trivial normal bundle;
\item 
$\theta \colon N \setminus B \to \SS^1$ is a fibration whose fibers are
transverse to $\p N$, and which coincides in a neighborhood $B \times \DD^2$ of 
$B = B \times \{0\}$ with the normal angular coordinate.
\end{itemize}
\end{definition}

\begin{definition}
Let $(V,\xi)$ be a $(2n+1)$-dimensional contact manifold.  A compact 
$(n + 1)$-dimensional submanifold $N\hookrightarrow V$ with boundary is called
a \defin{bordered Legendrian open book} (abbreviated \defin{\BLOB{}}), if it has a
relative open book $(B, \theta)$ such that:
\begin{itemize}
\item [(i)] all fibers of $\theta$ are Legendrian;
\item [(ii)] the boundary of $N$ is Legendrian.
\end{itemize}
\end{definition}

\begin{remark}
The binding~$B$ of a Legendrian open book is automatically isotropic because
its tangent space is contained in the tangent space of the closure of all pages.
Similarly, the fibers of $\theta$ and the boundary of $N$ meet transversely in $N$, and saying
that they are both Legendrian implies that the induced foliation on $N$ is singular
on $B$ and $\p N$.
\end{remark}

A \BLOB is an example of a \defin{maximally foliated} submanifold of $(V,\xi)$,
meaning that the singular distribution defined by intersecting its tangent
spaces with~$\xi$ is integrable, thus forming an oriented singular foliation,
and it has the largest dimension for which this is possible
(see \cite[Section~1]{NiederkruegerPlastikstufe} for further discussion).
A \BLOB in a $3$-dimensional contact manifold is the ``flat version'' of the 
overtwisted disk, the one
where the characteristic foliation is singular along the boundary.
This is a slight difference compared with the definition of plastikstufes in
\cite{NiederkruegerPlastikstufe}, where the boundary was a regular leaf of the
induced foliation, hence analogous to the ``cambered version'' of the
overtwisted disk.
This is a minor technical detail; each version can be deformed into the other
one.

\begin{definitionNumbered}\label{defn:PSOT}
  A contact manifold that admits a \BLOB is called \defin{$PS$-overtwisted}.
\end{definitionNumbered}

Note that the definition of the \BLOB is topologically \emph{much}
less restrictive than the initial definition of the plastikstufe.  For
example, a $3$-manifold admits a relative open book if and only if its
boundary is a nonempty union of tori. On the
other hand, a plastikstufe in dimension~$5$ is always diffeomorphic to
a solid torus~$\SS^1\times \DD^2$.

In this paper we will discuss one setting where we can find 
\BLOB{s} and are unable to find plastikstufes: in
Proposition~\ref{prop:bLob_in_blown_down_domain}, we show that \BLOB{s}
always exist in certain subdomains that are naturally associated to 
Liouville domains with
disconnected boundary, a special case of which produces the Lutz-type
twist due to Mori \cite{Mori_Lutz} (cf.~\S\ref{subsec:LutzMori}).

\begin{remark}
\label{remark:GPS}
Some \BLOB{s} also naturally arise in relation to the results of
\cite{NiederkruegerPresas_neighborhoods}, where it is shown that
sufficiently large neighborhoods of overtwisted submanifolds in
higher dimensional contact manifolds give a filling obstruction.
In \cite{NiederkruegerPresas_neighborhoods} this required a rather
technical argument involving holomorphic disks with an
immersed boundary condition, but it can be simplified and strengthened
by showing (using arguments similar to
those of Proposition~\ref{prop:bLob_in_blown_down_domain})
that such neighborhoods always contain a \BLOB.
\end{remark}

Of course, finding a \BLOB would be useless without the following theorem.

\begin{theorem}\label{thm:plastikstufe_and_weak_fillability}
  If a closed contact manifold is $PS$-overtwisted, then it does not have any
  (semipositive) weak symplectic filling $(W,\omega)$ for which $\omega$
  restricted to the \BLOB is exact.
\end{theorem}

\begin{remark}
\label{remark:small}
  The condition that the restriction of the symplectic form $\omega$
  should be exact is trivially satisfied in dimension~$5$ for the
  plastikstufes defined in \cite{NiederkruegerPlastikstufe}, which were
  all diffeomorphic to $\SS^1 \times \DD^2$.
  In general however this condition could fail, and we believe that this
  could provide a hint as to varying degrees of filling obstructions
  or overtwistedness.
  Though it is unknown whether there is a \emph{unique} natural
  notion of overtwistedness beyond dimension~$3$, or whether the different
  definitions known thus far are inequivalent,
  it would be interesting to speculate that a manifold can only
  be overtwisted in some ``universal'' sense if the \BLOB (or a
  similar object) can be \emph{embedded into a ball} within the contact
  manifold.  In this way the cohomological condition is satisfied automatically,
  thus defining an obstruction to weak fillings due to the above theorem.
  We will refer to any \BLOB that lies inside a ball in the contact manifold
  as a \defin{small \BLOB{}}.
\end{remark}

If $\dim V \ge 3$, then any contact structure~$\xi$ on~$V$ can be
modified either by
\cite{PresasExamplesPlastikstufes, vanKoertPSOvertwistedEverywhere} or by
\cite{EtnyreGeneralizedLutzTwists}---in the latter case without changing the
homotopy class of almost contact structures---to produce one
that is $PS$-overtwisted.
In both cases, the change produces a small plastikstufe, hence 
Theorem~\ref{thm:plastikstufe_and_weak_fillability} and the preceding section
imply Theorem~\ref{thm:non_weakly_fillable} stated in the introduction.

In the proof of Theorem~\ref{thm:plastikstufe_and_weak_fillability} below,
the general strategy is the same as in
\cite{NiederkruegerPlastikstufe, NiederkruegerPresas_neighborhoods}, but there
are differences coming from two sources: the need to handle weak rather
than strong fillings, and \BLOB{s} rather than plastikstufes. 
Working with weak fillings complicates the question of energy bounds 
because the integral of $\omega$ on a
holomorphic curve no longer has a direct relation to the integral of~$d\alpha$.
This is where the homological condition comes in.  Further, it is no 
longer obvious that we can choose our almost complex structure to be
both adapted to a contact form near the binding and boundary of the \BLOB 
and tamed by~$\omega$.  As far as the differences between the plastikstufe
and the \BLOB are concerned, the first is the singularity along the
boundary, which makes energy control easier but makes it harder to ensure that
holomorphic curves cannot escape through the boundary. 
This difference can be handled similarly to
the analogous work in \cite{NiederkruegerPresas_neighborhoods},
which dealt with the case where the fibration of the \BLOB becomes trivial
at the boundary.  The general case additionally requires the somewhat
technical Lemmas~\ref{lemma:boundaryBLOB_contact_form}
and~\ref{lemma:complexStructureBoundaryBLOB} below (though since we will
not need this level of generality for our main results, the reader
may skip these if desired).
The second difference is of course that pages are more complicated and the
interior monodromy can be anything, but this plays no role in the
proof; what matters is the existence of a fibration over $\SS^1$.	

\begin{proof}[Proof of  Theorem~\ref{thm:plastikstufe_and_weak_fillability}]
Let  $N$ be a \BLOB in $(V,\xi)$ with induced Legendrian open book
$(B,\theta)$.
Suppose that $(W,\omega)$ is a weak filling of $V$ for which  
$\restricted{\omega}{TN}$ is exact.
We choose a contact form $\alpha$ for $\xi$ and attach to $(W,\omega)$ the
corresponding conical end from Corollary~\ref{cor:magnetic_completion}.
Since the restriction of $\omega$ to a neighborhood of the \BLOB is exact, we
can choose a closed $2$-form $\Omega$ on $V$ that is cohomologous to
$\restricted{\omega}{TV}$ and vanishes on a neighborhood of~$N$.
In a second step, we can deform the symplectic structure on the
conical end to
\begin{equation*}
  \bigl([t_1,\infty)\times V,\, \Omega + d(t \alpha)\bigr)
\end{equation*}
for large $t_1$ as described in Lemma~\ref{thm:cohomologous_deformation_positive_end}.

Identify the contact manifold $(V, \xi)$ with a level set $\{T\}\times
V$ in the conical end for sufficiently large $T > t_1$, and choose an
almost complex structure $J$ close to $\{T\}\times V$ that makes
$(\{T\}\times V, \xi)$ pseudoconvex and is tamed by $d(t\alpha)$.
We require this $J$ to be of the explicit form given in
\cite{NiederkruegerPlastikstufe} in a neighborhood of the binding
$\{T\}\times B$, which means the following.  We can
identify a neighborhood of $\{T\} \times B$ symplectically with
an open set in $\CC^2 \times T^*B$, with symplectic structure
$\omega_0 \oplus d\lcan$, such that the part of the \BLOB intersecting
this neighborhood lies in $\CC^2 \times B$.  
The desired almost complex structure is then the
product of the standard structure~$i$ on the first factor with a tamed almost
complex structure on the cotangent bundle.
This choice simplifies the behavior of local holomorphic disks
significantly: indeed, any disk lying entirely
in this neighborhood and having boundary on the $\BLOB$ projects to 
disks in $\CC^2$ and $T^*B$, and the latter has boundary in the
zero-section and must therefore be constant for energy reasons.
In this way one can easily understand small
disks close to the binding of the $\BLOB$, and in particular one obtains the
existence of a Bishop family of holomorphic disks close to~$B$, as well as
the important fact that any holomorphic disk intersecting this model
neighborhood must be part of the Bishop family.  We refer to
\cite{NiederkruegerPlastikstufe} for the full details.

Similarly, $J$ should agree on a neighborhood of
$\{T\}\times \p N$ with an almost complex structure that we will describe
in Lemma~\ref{lemma:complexStructureBoundaryBLOB} below.  As explained in
Lemma~\ref{lemma:Jalpha}, we can ensure by increasing $T$ that the
chosen $J$ will not only be tamed by $d(t \alpha)$ but also by 
$\Omega + d(t \alpha)$ close to $\{T\}\times V$.  
Denote the symplectic manifold obtained by attaching $[0,T]\times V$ to $W$ by
$\widehat W$.  
We use contractibility of the space of tamed almost complex structure to extend
$J$ to the interior of the weak filling $\widehat W$.

As in \cite{NiederkruegerPlastikstufe}, we now study the connected
moduli space of $J$-holomorphic disks
\begin{equation*}
  u\colon (\DD^2, \p \DD^2)  \to \bigl(\widehat W, \{T\}\times 
  (N \setminus B)\bigr)
\end{equation*}
emerging from a so-called \emph{Bishop family} of disks in a neighborhood
of some point on~$B$.  The boundaries of these disks necessarily
intersect each page of the Legendrian open book exactly once.

We first establish the energy bound required for Gromov compactness.
Any holomorphic disk $u$ in the moduli space under consideration 
can be capped with a disk~$D$ lying in the $\BLOB$ so that $u$ together
with~$D$ bounds a $3$-ball~$B^3$.  Using Stokes' theorem,
\begin{equation*}
  0 = \int_{B^3} d\omega = \int_{u} \omega + \int_{D} \omega,
\end{equation*}
it then follows
that the energy of the holomorphic disk is equal to minus the symplectic
area of $D \subset N$.  But since the restriction of
$\omega$ to the \BLOB coincides with $T\,d\alpha$ in our construction, 
this quantity can
be determined by integrating $T \alpha$ over the common boundary of
the two disks $u$ and $D$:
\begin{equation*}
  E_\omega(u) = \int_u \omega = - \int_D T\,d\alpha = T\, \int_{\p u} \alpha \;.
\end{equation*}
Since the foliation on the \BLOB is given by $\xi \cap TN
= \ker d\theta$, there is a continuous function
$f\colon N\to \RR$ that is everywhere nonnegative and vanishes only on
$B\cup \p N$ such that $\restricted{\alpha}{TN} = f\,d\theta$.
The energy of $u$ is thus bounded by
\begin{equation*}
  E_{\omega} (u) = T\, \int_{\p u} \alpha \le
  2\pi T\, \max_{p\in N} f(p) \;.
\end{equation*}

This leads to the same contradiction to Gromov compactness as in the
proof for strong fillings \cite{NiederkruegerPlastikstufe}, because by
Lemma~\ref{lemma:complexStructureBoundaryBLOB} below, the boundaries of the
holomorphic disks are trapped between $B$ and $\p N$, and the topology
of the Legendrian open book prevents bubbling of disks.
\end{proof}

\begin{lemma}\label{lemma:boundaryBLOB_contact_form}
Suppose $N$ is a manifold with boundary carrying a relative open book
$(B,\theta)$ which embeds as a \BLOB into two contact manifolds
  $(V_1,\xi_1)$ and $(V_2,\xi_2)$.
  Then there are neighborhoods $U_1\subset V_1$ and $U_2\subset V_2$ of $\p
  N$ and a contactomorphism $\Phi\colon (U_1,\xi_1) \to (U_2, \xi_2)$
  such that $\Phi(N\cap U_1) = N\cap U_2$.
\end{lemma}
\begin{proof}
  Denote the two embeddings by $\iota_j\colon N \hookrightarrow V_j$ for
  $j=1,2$.
  The first step will be to prove the existence of contact forms
  $\alpha_1$ and $\alpha_2$ for $\xi_1$ and $\xi_2$ with
  $\iota_1^*\alpha_1 = \iota_2^*\alpha_2$ near $\p N$.
  Start with any pair of contact forms~$\alpha_1$ and $\alpha_2$.
	By the definition of a \BLOB, there are functions $h_1$ and $h_2$ which vanish
	exactly along $\p N$ such that $\iota_j^*\alpha_j = h_j\, d\theta$.
	We will prove shortly that $h_1$ and $h_2$ are both transverse to zero along
	$\p N$. The implicit function theorem then guarantees the existence of a
	positive function $f$ on~$V$ with $h_1 = f h_2$, allowing us to
	replace $\alpha_2$ by
	$f\alpha_2$.  The key point is that
	$\iota_j^*d\alpha_j = dh_j \wedge d\theta$, so 
	$dh_j$ cannot vanish anywhere along $\p N$, otherwise $TN$ would be an
	isotropic subspace of dimension $n + 1$ inside the symplectic vector space
	$(\xi_j, d\alpha_j)$ of dimension $2n$. 
	
  We now turn to the construction of the desired contactomorphism.
	We fix near $\p N$ a vector field $X_r$ tangent to $\ker d\theta$ and a vector
	field $X_\theta$ tangent to $\p N$ such that $d\theta(X_\theta) = 1$. Then
	$d\alpha_j(X_r, X_\theta) = \iota_j^*d\alpha_j(X_r, X_\theta) = dh(X_r)$ is
	positive.  We denote by $\fF$ the foliation on $\p N$ induced by the pages,
	meaning $T\fF = T\p N \cap \ker d\theta$. 
  Its tangent space is $d\alpha_j$-orthogonal to the symplectic subspace
  $\Span(X_r, X_\theta)$, so we can construct for each $j = 1,2$ a
  complex structure $J_j$ on $\xi_j$ which is compatible with
  $d\alpha_j$, such that $X_\theta = J_jX_r$ and the
  $d\alpha_j$-symplectic complement of
  $\Span(X_r,X_\theta)$ in $\xi_j$ is $T\fF \oplus J_j T\fF$.
  Denoting the Reeb vector field of $\alpha_j$ by $R_j$, we obtain the
  decomposition
  \begin{equation*}
    \restricted{TV_j}{\p N} = 
    \Span(X_r, X_\theta) \oplus T\fF \oplus J_j T\fF \oplus \Span(R_j) \;.
  \end{equation*}
  The first two summands span $\restricted{TN}{\p N}$, and each $\nu_j
  := J_j T\fF \oplus \Span(R_j)$ can be identified with the normal
  bundle of $N$.
  Let $\tau_j$ be the restriction to $\nu_j$ of the exponential map
  for some auxilliary Riemannian metric.
  Each $\tau_j$ allows us to identify a tubular neighborhood of $N$ with a
  neighborhood of the zero section in~$\nu_j$.
  The bundles $\nu_1$ and $\nu_2$ are related by the bundle map $\Phi
  := \Psi \oplus \Phi_R$, where $\Phi_R$ sends $t\cdot R_1$ to $t\cdot
  R_2$ and $\Psi = \phi_2^{-1} \circ \phi_1$, with $\phi_j
  \colon J_j T\fF \to T^* \fF$ denoting the interior product with
  $d\alpha_j$.
  Thus $\tau_2 \circ \Phi \circ \tau_1^{-1}$ combines with the identity
  on $N$ to give a diffeomorphism between tubular neighborhoods of $N$
  in $V_1$ and $V_2$ near $\p N$.
  This map pulls $\alpha_2$ back to $\alpha_1$ and $d\alpha_2$ to
  $d\alpha_1$ for every $p\in \p N$, so that the linear interpolation
  between both forms is a contact form, and we may apply the Moser
  trick.
  Denoting by $\beta_t$ with $t\in [0,1]$ the interpolation between the
  pulled back contact forms, the Moser vector field~$Y_t$ is the
  unique solution to the two equations
  \begin{equation*}
		\beta_t(Y_t) = 0 \quad \text{ and } \quad 
		\restricted{\bigl(\iota_{Y_t}	d\beta_t\bigr)}{\ker\beta_t} = 
		-\restricted{\dot\beta_t}{\ker \beta_t}. 
  \end{equation*}
  From this we see that $Y_t$ vanishes along $\p N$, so that the
  isotopy~$\flow_t$ is well defined on a small neighborhood~$U$ of 
	$\p N$ and fixes $\p N$ pointwise.
	We now observe that $\restricted{Y_t}{N}$ lies in $\ker d\theta$, 
	so that the isotopy preserves $N$.
	Indeed, if $Y_t$ had any component in the complement of the Lagrangian
subspace	$\ker	d\theta$, it would pair via $d\beta_t$ with a vector in 
$\ker d\theta$ and thus be different from $- \dot \beta_t$, which vanishes 
on $TN$.
\end{proof}

We can now construct a suitable almost complex structure on a model which will
be universal according to the preceding lemma.

\begin{lemma}\label{lemma:complexStructureBoundaryBLOB}
  Assume $(W,\omega)$ has a conical end, and identify $(V,\xi)$ with a
  level set $\{T\}\times V$ of this conical end.
  Let $\alpha$ be any contact form for $\xi$.
  If $N$ is a \BLOB in $V$, then we can choose an almost complex
  structure~$J_0$ in a neighborhood $U_W\subset W$ of the boundary~$\p
  N$ with the following properties:
  \begin{itemize}
  \item $J_0$ is compatible with the symplectization form $d(t\alpha)$
    and it restricts to $\xi$.
  \item If $J$ is any almost complex structure on $W$ that makes
    $(V,\xi)$ pseudoconvex and for which $\restricted{J}{U_W} = J_0$,
    then every compact $J$-holomorphic curve
    \begin{equation*}
      u\colon (\Sigma, \p \Sigma) \to (W,N)
    \end{equation*}
    that intersects $U_W$ and whose boundary lies in the \BLOB must
    be constant.
  \end{itemize}
\end{lemma}
\begin{proof}
	The first step is to construct a model neighborhood for $\p N$ which is a
	bundle with exact symplectic fibers and holomorphic projection map.
  Let $\fF_0$ be a fiber of the map $\restricted{\theta}{\p N} \colon
  \p N \to \SS^1$. Then $\fF_0$ is the intersection of $\p N$
  with a page of the \BLOB, and $\p N$ is the mapping torus of some
	diffeomorphism $\psi\colon \fF_0\to \fF_0$.
  We consider the $T^*\fF_0$-fibration
  \begin{equation*}
    \begin{split}
      \pi\colon \CC \times \RR \times \bigl(\RR \times
      T^*\fF_0\bigr)/\sim
      &\quad \to \quad \CC \times T^* \SS^1  \\
      (z,r;s;q,p) &\quad \mapsto \quad (z; s,r)
    \end{split} \;,
  \end{equation*}
  where we use the equivalence relation $(z,r;s;q,p) \sim
  \bigl(z,r;s+1; \psi(q),(D\psi^{-1})^*p\bigr)$ on the total space.
	Since $(\psi, (D\psi^{-1})^*)$ is symplectic, we get a symplectic structure
  $d\lcan$ on the vertical bundle $\ker D\pi$. Let $J_\fF$ be a compatible
	complex structure on this bundle.
  Note that the directions $\p_r$, and $\p_s$ are well defined, so
  that we can extend $J_\fF$ to an almost complex structure $J
  = i \oplus i \oplus J_\fF$ on the total space, where $i\,\p_r =
  \p_s$, and $i\,\p_s = - \p_r$.
  By construction, $\pi$ is holomorphic with respect to
  $J$ upstairs and $i\oplus i$ on $\CC \times T^* \SS^1$.
  The next step consists in finding a $J$-plurisubharmonic function
  on a neighborhood of $\{1\} \times \{0\} \times \bigl(\RR \times
  \fF_0\bigr)/\sim$, where $\fF_0$ denotes the $0$-section in
  $T^*\fF_0$.
  Define a function~$h$ on $\CC \times \RR \times \bigl(\RR \times
  T^*\fF_0\bigr)/\sim$ by using a metric on the \emph{vector} bundle
  $\CC \times \RR \times \bigl(\RR \times T^*\fF_0\bigr)/\sim$ over
  $\CC \times \RR \times \bigl(\RR \times \fF_0\bigr)/\sim$, and
  defining $h(v) = \norm{v}^2 /2$ for every vector~$v$ in this bundle.
  In a bundle chart, we obtain
  \begin{equation*}
    h(z,r;s;q,p) =
    \frac{1}{2}\, \sum_{i,j} g_{i,j}(z,r;s;q)\, p_i p_j \;,
  \end{equation*}
  and it follows that $dd^ch = d (dh\circ J)$ simplifies on the
  $0$-section of this bundle to
  \begin{equation*}
    dd^ch =  \sum_{i,j}  g_{i,j}\, dp_i\wedge (d p_j\circ J) \;.
  \end{equation*}
  We claim now that the function
  \begin{equation*}
    \begin{split}
      F\colon \CC \times \RR \times \bigl(\RR \times
      T^*\fF_0\bigr)/\sim & \to [0,\infty), \\
      (z,r;s;q,p) &\mapsto \abs{z}^2 + r^2 + h(z,r;s;q,p)
    \end{split}
  \end{equation*}
  is $J$-plurisubharmonic in a neighborhood of $\{1\} \times \RR
  \times \bigl(\RR \times \fF_0\bigr)/\sim$.
  Here one just needs to check 
  that $- dd^cF$ simplifies near $\{1\} \times \RR
  \times \bigl(\RR \times \fF_0\bigr)/\sim$ to
  \begin{equation*}
    - dd^cF =  4\, dx\wedge dy  + 2\, dr\wedge ds - dd^ch \;,
  \end{equation*}
  where $x = \RealPart z$ and $y = \ImaginaryPart z$.
  This $2$-form is positive on complex lines.
  We find a neighborhood of $\{x = 1\}$ in the level set $F^{-1}(1)$,
  where the restriction of the $1$-form~$\alpha := - dF\circ J$
  defines a contact structure.
  Furthermore, the submanifold~$N' \subset F^{-1}(1)$ given by the
  embedding
  \begin{equation*}
    \begin{split}
      [0, \epsilon)\times \bigl(\RR\times \fF_0\bigr)/\sim &\quad
      \hookrightarrow \quad \CC \times \RR \times \bigl(\RR \times
      T^*\fF_0\bigr)/\sim \\
      (r;s;q) &\quad \mapsto \quad \bigl(\sqrt{1 - r^2} , r;
      s;q,0\bigr)
    \end{split}
  \end{equation*}
  has $N'\cap \{r=0\}$ as boundary and inherits a singular Legendrian
  foliation given by the form~$r\,ds$.
  This foliation is diffeomorphic to the one on the \BLOB~$N$ in the
  collar neighborhood of~$\p N$, so that by
  Lemma~\ref{lemma:boundaryBLOB_contact_form} above, there is a small
  relatively open set $U_W\subset F^{-1}\bigl((0,1]\bigr)$  in the model
  containing $\p N'$
  such that $\p_+ U := U_W\cap F^{-1}(1)$ with contact form $\alpha$
  is contactomorphic to a neighborhood $U_V$ of $\p N$ in~$V$.
  Note that for $\delta > 0$ sufficiently small, the level set $\{x =
  1 - \delta\}$ is a compact hypersurface with boundary in $\p_+ U$,
  and we will set $\p_- U := \{x= 1 - \delta\} \cap U_W$, writing
  from now on $U_W$ for the compact set $U_W\cap \{x \ge 1 -
  \delta\}$.
  Extending this contactomorphism, we can embed $U_W$ into the
  symplectic manifold~$W$ such that $\p_+ U$ lies in $\{T\}\times V$,
  and $N' \cap U_W$ is mapped to $N \cap U_V$.
  Choose the almost complex structure~$J$ on $U_W$ constructed above,
  and extend it to one that makes the contact manifold $(V,\xi)$
  pseudoconvex.

\begin{figure}[htbp]
  \centering
  \includegraphics[height=3.5cm,keepaspectratio]{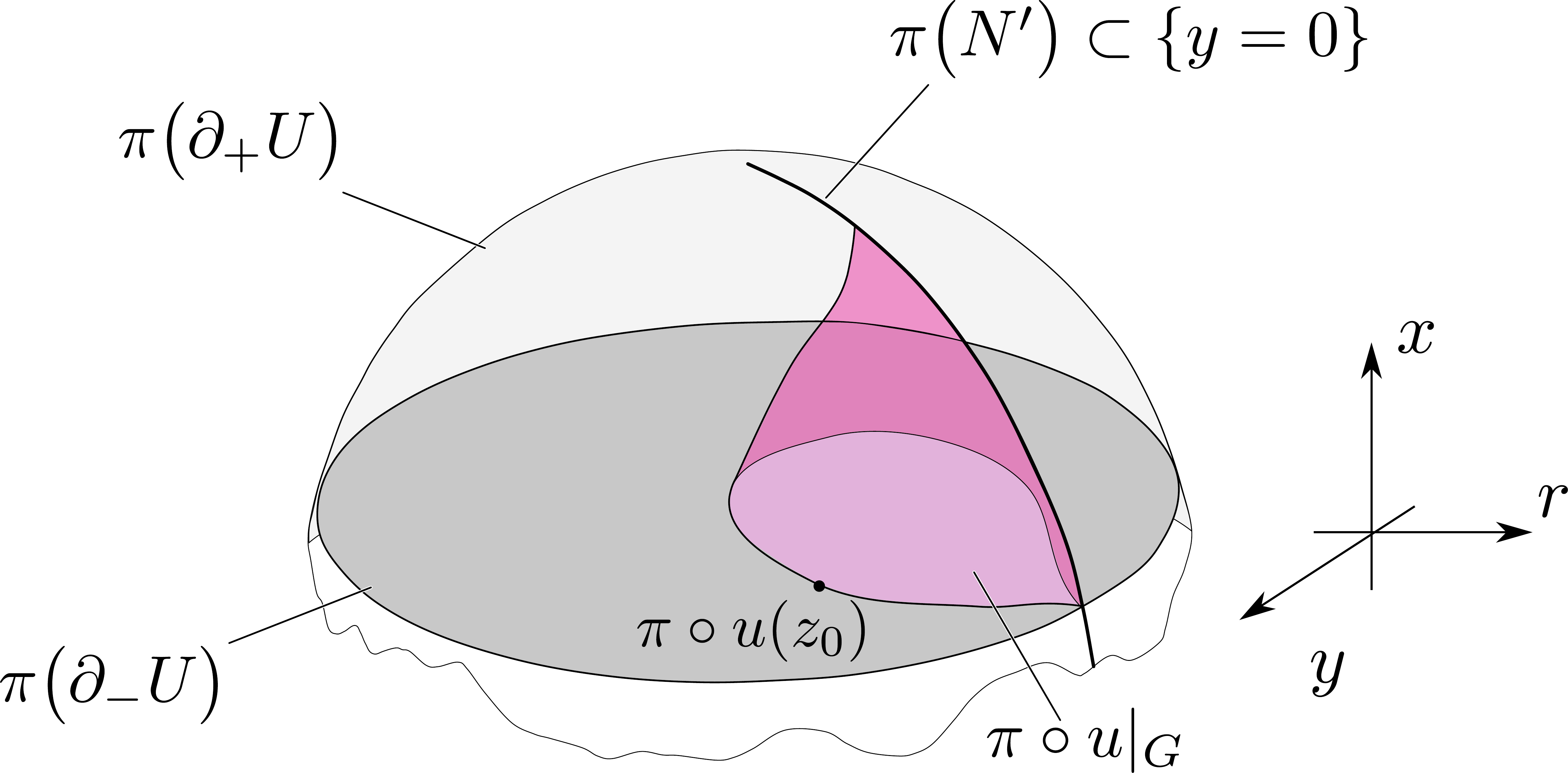}
  \caption{The neighborhood of the boundary~$\p N$ can be thought of as a
    $T^*\fF_0$-bundle.
We obtain a simple model by projecting this neighborhood and the
holomorphic curve~$u$ to the base space.
    The holomorphic curve~$\pi\circ u$ has to be cut off and will have
    two types of boundary: the original one that sits in the \BLOB and
    the boundary where the curve has been cut off.
    Along the cut-off boundary the $x$-value is minimal, and there
    will be a point where the $y$-value also becomes extremal, but
    this contradicts the boundary point lemma showing that the $x$- and
    $y$-values of $u$ have to be constant.\label{fig:bdry_BLOB}}
\end{figure}

  Now let $u\colon (\Sigma,\p \Sigma) \to (W, N)$ be any
  $J$-holomorphic curve that intersects the neighborhood~$U_W$.
  Our aim is to show that $u$ must be constant.
  Define $G := u^{-1}(U_W)$ and write $\restricted{u}{G}$ for the
  restriction of~$u$.
  Perturbing $\delta$ slightly, we can assume that $u^{-1}\bigl(\p_-
  U\bigr) \subset G$ is a properly embedded submanifold so that $G$
  has piecewise smooth boundary.
  Project the curve $\restricted{u}{G}$ via
  \begin{equation*}
    \pi\colon \CC \times \RR \times \bigl(\RR \times
    T^*\fF_0\bigr)/\sim \quad \to \quad \CC \times T^* \SS^1 \;,
  \end{equation*}
  and note that $\restricted{\pi\circ u}{G}$ is a holomorphic map with
  respect to the standard structure (see Fig.~\ref{fig:bdry_BLOB}).
  The boundary $\pi\circ u(\p G)$ lies in the union of
  \begin{align*}
    \pi\bigl(\p_+ U\cap N' \bigr) &= \left\{(z; s,r) \bigm|\,
    \RealPart z \ge 1 - \delta, \, \ImaginaryPart z = 0, \, r = -
    \sqrt{1 - \abs{z}^2} \right\} \intertext{and} \pi(\p_- U) &=
    \bigl\{(z; s, r) \bigm|\, \RealPart z = 1 - \delta, \, \abs{z}^2 +
    r^2 \le 1 \bigr\} \;.
  \end{align*}
  Since both coordinate functions $x = \RealPart z$ and $y =
  \ImaginaryPart z$ are harmonic, it follows that the maxima and
  minima are both attained on $\p G$, so that if we assume $y$ is
  not everywhere equal to $0$, then $u$ must intersect $\pi(\p_- U)$, and
  in particular the minimum of $x$ will be $1-\delta$.
  Let $z_0\in \p G$ be a point for which $u(z_0)$ has both minimal
  $x$-coordinate and extremal $y$-coordinate.
  At this point, it follows that the derivative of
  $\restricted{\pi\circ u}{G}$ in the $\p G$-direction has vanishing
  $x$ and $y$-coordinates.
  Using the Cauchy-Riemann equation at the point $z_0$, we then see
  that the derivatives also vanish in the radial direction, thus
  contradicting the boundary point lemma, making both $x$ and $y$
  constant on~$G$.
  It follows now that $u$ is completely contained in~$U_W \cap \{z =
  x_0\}$, and from this we immediately recover that the
  $r$-coordinate of $\restricted{u}{\p \Sigma}$ is equal to~$-
  \sqrt{1-x_0^2}$.
  The $r$-coordinate is also harmonic, and it follows that $\pi\circ
  u$ must have constant $r$-coordinate everywhere, since both its
  maximum and its minimum are equal, and the Cauchy-Riemann equation
  then implies that the $s$-coordinate is also constant.
  This finishes the proof, because it follows that the
  projection~$\pi\circ u$ is constant, so that $u$ is completely
  contained in a fiber of $\pi$ that is symplectomorphic to $T^*\fF_0$
  with exact symplectic form $d\lcan$, but since $J_\fF$ was
  compatible with $d\lcan$, and since the boundary of $u$ lies in
  the $0$-section of $T^*\fF_0$, it follows that $u$ has no
  $d\lcan$-energy, and hence must be constant.
\end{proof}

\section{Giroux domains}
\label{section:giroux_domains}

While the filling obstructions we've discussed so far (namely
\BLOB{s} and negative stabilizations) were previously understood in
less general forms, in this section we shall introduce a subtler class of
filling obstructions that generalizes Giroux torsion in dimension
three and is completely new in higher dimensions.
The fundamental objects in this discussion are called Giroux domains
and ideal Liouville domains.  As was sketched in the introduction, an
ideal Liouville domain is a natural compactification of a complete
Liouville manifold, and its product with $\SS^1$ naturally inherits a
contact structure, producing what we call a Giroux domain.
The definitions and elementary properties of these objects, including a
blow-down operation along boundary components, are due to Giroux but
cannot yet be found anywhere in the literature, so we will discuss
them in some detail in \S\ref{subsec:idealLiouville} and the beginning
of \S\ref{subsec:GirouxDomains}.
Before that, in \S\ref{subsec:round}, we introduce for later
convenience a slightly more general context for the blow-down
operation.
Some explicit examples of blown down Giroux domains have already
appeared in the work of Mori \cite{Mori_Lutz}, who showed that his
examples always contain a plastikstufe.
The notion of the \BLOB allows us to generalize this result using a
purely topological description that we will explain in~\S\ref{subsec:blowdown}.
The last subsection, culminating with the statement of
Theorem~\ref{thm:torsion}, defines a filling obstruction in terms of
Giroux domains which refines Theorem~\ref{thm:1torsion} from the
introduction and sets the stage for our higher dimensional
generalization of Giroux torsion in~\S\ref{sec:torsion}.

\subsection{Round hypersurfaces}
\label{subsec:round}

We say that an oriented hypersurface $H$ in a contact manifold $(V, \xi)$ is
a \defin{$\xi$-round hypersurface modeled on} some closed contact manifold 
$(M, \xi_M)$ if it is transverse to $\xi$ and admits an orientation
preserving identification with 
$\SS^1 \times M$ such that $\xi \cap TH = T\SS^1 \oplus \xi_M$.
In this definition, the word ``round'' is used as in ``round handle''. 
In general, the orientation of a round hypersurface may be chosen at
will, and we shall assume in particular that whenever $H$ is a component
of $\p V$, its orientation 
is the \emph{opposite} of the natural boundary orientation;
see Remark~\ref{remark:IDontKnowHowYouWereInverted} below.
Observe that in dimension three, a $\xi$-round hypersurface is simply a 
pre-Lagrangian torus with closed characteristic leaves.

\begin{lemma}
\label{lemma:round_neighborhood}
Any $\xi$-round hypersurface $H = \SS^1 \times M$ in the interior (or
boundary) of $(V, \xi)$ has a neighborhood 
$(-\varepsilon,\varepsilon) \times H$ (or $[0,\varepsilon) \times H$
respectively) on which
$\xi$ can be defined by the contact form $\alpha_M + s\, dt$ where $s$ is the
coordinate on the interval, $t$ the coordinate in $\SS^1$
and $\alpha_M$ is a contact form for $\xi_M$.
\end{lemma}
\begin{proof}
Fix any tubular neighborhood (or collar neighborhood) of $H$ with
coordinate~$t$.
The $1$-form described defines a contact form near $H$ which induces the same
hyperplane field as $\xi$ on $H$, hence they are isotopic relative to $H$.
Pulling back the neighborhood under this isotopy gives the desired
neighborhood.
\end{proof}

Suppose $H$ is a $\xi$-round boundary component of $(V, \xi)$, with
orientation opposite the boundary orientation, and consider the
collar neighborhood from the preceding lemma. 
We now explain how to 
modify $(V,\xi)$ by \defin{blowing down} $H$ to $M$. 
Let $D$ be the disk of radius $\sqrt\varepsilon$ in $\RR^2$.
The map $\Psi \colon (r e^{i\theta},m) \mapsto (r^2,\theta,m)$ is a 
diffeomorphism from 
$\big(D \setminus \{0\}\big) \times M$ to $(0, \varepsilon) \times \SS^1
\times M$ which pulls back 
$\alpha_M + s\, dt$ to the contact form $\alpha_M + r^2 d\theta$. 
Thus we can glue $D \times M$ to $V \setminus H$ 
to get a new contact manifold in which~$H$ has been replaced by~$M$.
This process is equivalent to performing a \emph{contact cut} of $V$ with
respect to the (local) $\theta$-action, as described in
\cite{Lerman_ContactCuts}.

\begin{remark}
\label{remark:IDontKnowHowYouWereInverted}
Topologically, the blow down operation glues $\DD^2 \times M$ to~$V$ 
via the natural identification of 
$\p(\DD^2 \times M)$ with $\SS^1 \times M = H \subset V$.  This is why it is
appropriate to assign to~$H$ the reverse of its natural boundary orientation
with respect to~$V$.
\end{remark}

\subsection{Ideal Liouville domains}
\label{subsec:idealLiouville}

The following notion is of central importance for the new filling
obstructions that we will introduce.

\begin{definition}[Giroux]
  Let $\Sigma$ be a compact manifold with boundary, $\omega$ a
  symplectic form on the interior $\IntSig$ of $\Sigma$ and $\xi$ a contact
  structure on $\p\Sigma$.
  The triple $(\Sigma, \omega, \xi)$ is an \defin{ideal Liouville domain} if
  there exists an auxiliary $1$-form $\beta$ on $\IntSig$ such that:
  \begin{itemize}
  \item $d\beta = \omega$ on $\IntSig$,
  \item For any smooth function $f\colon \Sigma \to [0,\infty)$ with
	regular level set $\p\Sigma = f^{-1}(0)$, the $1$-form $f\beta$
	extends smoothly to $\p\Sigma$ such that its restriction to
	$\p\Sigma$ is a contact form for~$\xi$.
  \end{itemize}
  In this situation, $\beta$ is called a \defin{Liouville form} for $(\Sigma,
  \omega, \xi)$.
\end{definition}

\begin{remark}
In the above definition, the space of possible auxiliary Liouville forms $\beta$
is contractible.  Indeed, we first observe that if the second condition
is satisfied for any given function $f_1$ as specified in the definition, 
then it is also satisfied for any other function $f_2$ with the required 
properties, as we then have $f_2 = g f_1$ for some smooth function
$g \colon \Sigma \to (0,\infty)$.  Thus we can fix a suitable function~$f$
and see that the set of admissible primitives $\beta$ on $\IntSig$ is convex.
An interesting variation on the above definition is obtained by also
regarding~$\xi$ as auxiliary data: this still leaves a
contractible space of auxiliary choices, but it is slightly less
convenient for our purposes. 
\end{remark}

\begin{remark}
\label{remark:fbeta}
Note that for $\beta$ and $f$ as in the above definition, there is no
requirement that $d(f\beta)$ should be symplectic, and in general it is not.
It is true however that one can always find 
(using Lemma~\ref{lemma:idealCollar} below) suitable functions~$f$ for
which $f\beta$ also defines a Liouville form on $\IntSig$, and Liouville
forms of this type arise naturally in certain examples,
cf.~Example~\ref{ex:torsionDomain} and Remark~\ref{remark:convexLiouville}.
\end{remark}

One can check that a Liouville form $\beta$ for an ideal Liouville domain
$\Sigma$ defines on the interior of $\Sigma$ the structure of a \emph{complete}
Liouville manifold.  This means that the flow of the vector field $X$
which is $\omega$-dual to $\beta$ is complete, and in particular the interior 
of $\Sigma$ has infinite volume with respect to $\omega$. This follows from 
Lemma~\ref{lemma:idealCollar} below, which describes precisely what happens 
near~$\p \Sigma$.  
For our purposes, one may regard the statement of this lemma as
part of the definition of an ideal Liouville domain, but keeping in mind that it
is already implied by the definition above.

\begin{lemma}[Giroux]
\label{lemma:idealCollar}
Suppose $(\Sigma, \omega, \xi)$ is an ideal Liouville domain with auxiliary
Liouville form~$\beta$, and let $X$ denote its $\omega$-dual vector
field, i.e.~the unique vector field on $\IntSig$ that satisfies
$\iota_X\omega = \beta$.

Choose any smooth function $f\colon \Sigma \to [0,\infty)$ with regular level
set $\p\Sigma = f^{-1}(0)$.  Then the vector
field $X_f := \frac{1}{f} X$ on $\IntSig$ extends smoothly over $\p\Sigma$
so that it points transversely outward.  Moreover, a collar neighborhood
of~$\p\Sigma$ can be identified with $(0, 1] \times \p\Sigma$ with 
coordinate $s \in (0, 1]$ such that 
$\beta = \frac{1}{1 - s}\, \alpha$ on $(0, 1) \times \p\Sigma$, where
$\alpha$ is a contact form for~$\xi$.
\end{lemma}
\begin{proof}
By definition, the $1$-form $\gamma := f\beta$ extends smoothly to $\Sigma$
and restricts on the boundary~$\p\Sigma$ to a contact form for $\xi$.
The smooth $2n$-form $\mu := f\,d\gamma^n - n\, df\wedge
\gamma \wedge d\gamma^{n-1}$ on the domain~$\Sigma$ simplifies on
the interior~$\IntSig$ to
\begin{equation*}
  \mu = f^{n+1} \,\omega^n \;,
\end{equation*}
and is hence a volume form on $\IntSig$.  It is also nondegenerate
along $\p\Sigma$, since $f\,d\gamma^n$ vanishes and $\gamma \wedge
d\gamma^{n - 1}$ is a volume form on $T(\p\Sigma) = \ker df$.  It
follows that there is a unique vector field~$X_f$ on $\Sigma$
satisfying the equation
\begin{equation*}
  \iota_{X_f} \mu = n\, \gamma \wedge d\gamma^{n - 1} \;.
\end{equation*}
Using $\iota_X \omega^n = n\, \beta \wedge d\beta^{n - 1}$ on the
interior~$\IntSig$, one can check that $\restricted{X_f}{\IntSig} =
\frac{1}{f}\,X$, and since the first term of $\mu$ vanishes at~$\p\Sigma$
and $f$ decreases in the outward direction, it follows that $X_f$ points
transversely outward through~$\p\Sigma$.

We now construct the collar neighborhood.
The basic idea is to follow the flow of $X_f$ starting from $\p\Sigma$,
but for a particular choice of the function $f\colon \Sigma \to [0,\infty)$
with regular level set $f^{-1}(0) = \p\Sigma$.  Starting from an
arbitrary function~$f$ of this type, any other such function~$h$ can be
written as $h = gf$ for some positive function~$g$ on~$\Sigma$. 
We then seek $h$ such that the 
vector field~$X_h = \frac{1}{h}\, X$ satisfies
\begin{equation*}
  \lie{X_h} (h\beta) =  0 \;.
\end{equation*}
This condition is equivalent to $dh(X_h) = -1$, which leads to the
ordinary differential equation $dg(X_f) = - \frac{1+df(X_f)}{f}\, g$.
The function~$f$ vanishes along $\p\Sigma$, and by the construction
above, we see that $df(X_f) = -1$ on $\p\Sigma$, thus the
differential equation is well behaved at $\p\Sigma$ and can be
solved with initial condition $\restricted{g}{\p\Sigma} \equiv 1$.

We denote by $\alpha$ the contact form induced by $h\beta$ on
$\p\Sigma$.  Since $\iota_{X_h} h\beta = 0$ and $\lie{X_h} h\beta
= 0$, the flow $\flow_t^{X_h}$ of $X_h$ for negative~$t$ pulls $h\beta$
back to $\alpha$.  Further, from $dh(X_h) = -1$ we obtain
$h\circ \flow_t^{X_h} = -t$, so
\begin{equation*}
  (\flow_t^{X_h})^* (h\beta) =
  -t\cdot (\flow_t^{X_h})^*\beta = \alpha \;.
\end{equation*}
Reparameterizing the time variable, we finally obtain the map
$\Phi(s,p) := \flow_{s-1}^{X_h}(p)$ which gives the desired collar
neighborhood with $(1 - s)\, \Phi^*\beta = \alpha$.
\end{proof}

\subsection{Giroux domains}
\label{subsec:GirouxDomains}

Given an ideal Liouville domain $(\Sigma, \omega, \xi)$ and a Liouville form
$\beta$, one can endow $\Sigma \times \RR$ with the contact structure
$\ker (f\, dt + f\beta)$ for any smooth function $f\colon \Sigma \to [0,\infty)$
with regular level set $f^{-1}(0) = \p\Sigma$.  Over the interior
of $\Sigma$, $\ker (f\,dt + f\,\beta) = \ker (dt + \beta)$, so one recovers the
standard notion of the \emph{contactization} of the Liouville manifold 
defined by~$\beta$.  On the boundary we have
$f\,dt = 0$, so the contact hyperplanes are $\xi \oplus T\RR$. 
Any two contact structures obtained in this way from
different Liouville forms are isotopic relative to the boundary. 
Since the contact forms constructed on $\Sigma \times \RR$ are $\RR$-invariant,
one can just as well replace $\RR$ by~$\SS^1$.
We will refer to $\Sigma \times \SS^1$ with the contact structure defined
in this way as the \defin{Giroux domain} associated to $(\Sigma,
\omega, \xi)$; see Example~\ref{ex:torsionDomain} 
below for our main motivation.  Observe that the
boundary is a $\xi$-round hypersurface modeled on $(\p\Sigma, \xi)$.

\begin{remark}
The above is a special case of a more general construction, also due to
Giroux, known as the \emph{suspension} of a
symplectomorphism $\varphi$ with compact support in $\IntSig$.  The result
of this construction also has $\xi$-round boundary, and blowing it down 
gives the contact manifold associated to the abstract open book 
$(\Sigma, \varphi)$.  Observe that unlike Giroux's original construction of the
contact structure associated to an open book (see e.g. in
\cite[Section~7.3]{GeigesBook}), this construction does not require any tweaking
near the binding.

In a different direction, one can generalize the construction of Giroux 
domains to allow for nontrivial circle bundles over $\Sigma$ using ideas from
\cite{Geiges_circle_bundles}.
\end{remark}

\begin{example}\label{ex:torsionDomain}
We consider
\begin{equation*}
  \Sigma = \SS^1 \times [0,\pi] , \quad
  \omega = \frac 1{\sin^2 s}\, d\theta \wedge ds
\end{equation*}
where $s$ is the coordinate in $[0,\pi]$ and $\theta$ the
coordinate in $\SS^1$, carrying the trivial contact structure 
$\ker \pm d\theta$. 
One can take as a Liouville form 
$\beta = \cot s\; d\theta$. Setting $f(\theta,s) = \sin s$, we get the
contact form 
$f(\theta,s)\cdot (\beta + dt) = \cos s\; d\theta + \sin s\; dt$ on 
$\Sigma \times \SS^1$.
Thus the Giroux domain associated to this ideal Liouville domain is
a Giroux $\pi$-torsion domain.
\end{example}

\subsection{Blowing down}
\label{subsec:blowdown}

Let $M$ be a union of connected components of the boundary of a Giroux domain
$\Sigma \times \SS^1$. These components are $\xi$-round
hypersurfaces and can thus be blown down as described in
\S\ref{subsec:round}.  We shall denote the resulting manifold by
$(\Sigma \times \SS^1) \blowdown M$.  It inherits a natural contact structure
for which each of the blown down boundary components becomes a codimension two
contact submanifold.

\begin{example}
\label{ex:torsionLutz}
Continuing the annulus example, a Giroux $\pi$-torsion domain with one boundary
component blown down is a so-called \emph{Lutz tube}, i.e.~the solid torus
that results from performing a Lutz twist along a transverse knot. 
With both boundary components blown down,
it is the standard contact structure on $\SS^2 \times \SS^1$.
\end{example}

In the above example, when one boundary component is blown down but not the
other, the resulting domain contains an overtwisted disk. We now generalize this
to higher dimensions.

\begin{proposition}\label{prop:bLob_in_blown_down_domain}
	Suppose $(V,\xi)$ is a contact manifold containing a subdomain~$G$ with
	nonempty boundary, obtained from a Giroux domain by blowing down at least one
	boundary component. Then $(V,\xi)$ contains a small \BLOB
	(cf.~Remark~\ref{remark:small}).
\end{proposition}

The \BLOB in the above proposition will come from a Lagrangian submanifold in
an ideal Liouville domain $(\Sigma,\omega,\xi)$. We first need a technical 
definition
describing how these submanifolds will be allowed to approach the boundary. We
say that a submanifold~$L$ properly embedded inside $\Sigma$ and transverse
to the boundary is a \defin{Lagrangian with cylindrical end} if:
\begin{itemize}
	\item $\mathring L$ is Lagrangian in $\IntSig$,
	\item $\p L$ is Legendrian in $\p \Sigma$,
	\item There is a Liouville form $\beta$ whose $\omega$-dual vector 
	field is tangent to $L$ near $\p \Sigma$. 
		More precisely, there is a collar $(0, 1] \times \partial \Sigma$ 
		as in Lemma~\ref{lemma:idealCollar} which intersects $L$ along
		$(0, 1] \times \p L$.
\end{itemize}
We will say that the Liouville form in this definition 
is \defin{adapted} to~$L$.

\begin{lemma}
\label{lemma:lag_to_blob}
Let $(\Sigma, \omega, \xi)$ be an ideal Liouville domain. If $L$ is a
Lagrangian with cylindrical end in $\Sigma$, then $\widehat L := L
\times \SS^1$ inside the contactization $\Sigma \times \SS^1$ is
isotopic to a maximally foliated submanifold whose singular set is its
boundary and whose foliation is otherwise defined via a fibration
\begin{equation*}
  \vartheta\colon \widehat L \to \SS^1,\quad (l,t) \mapsto F(l) + t \;,
\end{equation*}
for some smooth function $F\colon L \to \SS^1$ that is constant
on a neighborhood of~$\p L$.
\end{lemma}
\begin{proof}
We first assume that there is a Liouville form
$\beta$ adapted to $L$ which induces a rational cohomology class on $L$.
This implies there is a real number $\hbar > 0$ such that
$\hbar^{-1}$ times the cohomology class of the restriction of $\beta$ to $L$ is
integral: $\hbar^{-1}\,[i^*\beta] \in H^1(L; \ZZ)$.
First note that $\hbar u\,dt + u\beta$ defines a contact structure isotopic to
$\ker (u\, dt + u\beta)$ relative to the boundary of the Giroux domain 
$G := \Sigma \times \SS^1$.
Furthermore, the vector field constructed in the standard proof of Gray's
theorem vanishes along $\p\Sigma \times \SS^1$, so this isotopy is actually tangent
to the identity along the boundary. We shall now prove the lemma using this 
contact form (and no further isotopy of $L \times \SS^1$).

In the interior of $G$, the contact structure is defined by $\hbar\,dt + \beta$,
which restricts to $\eta = \hbar\,dt + i^*\beta$ on $\widehat L$. Since $\eta$ is
closed, $\widehat L$ is foliated. Moreover, $\hbar\, dt$ never vanishes in 
$\widehat L$, so there is no singularity there. Along the boundary, the contact
structure is defined by $u\beta$, whose restriction to $\widehat L$ vanishes,
thus the singularities are as claimed.

We now define the fibration $\vartheta$ using Tischler's 
construction (cf.~\cite{Tischler}).
Let $(l_0, t_0)$ be any base point in the interior of $\widehat L$. We define
$\vartheta(l, t) = \frac {1}{\hbar}\,\int_\gamma \eta$, where $\gamma$ is any 
path from $(l_0, t_0)$ to $(l, t)$. Since $\eta$ is closed, Stokes' theorem 
guarantees that this is well defined modulo the integral of
$\eta$ along loops based at $(l_0, t_0)$.  If $(\gamma_L, \gamma_t)$ is such a
loop, then the integral over it is 
$\bigl\langle [\beta], [\gamma_L] \bigr\rangle + 
\hbar\, \bigl\langle [dt], [\gamma_t] \bigr\rangle$,
which belongs to $\hbar\, \ZZ + \hbar \ZZ = \hbar\, \ZZ$, thus $\vartheta$ has
a well-defined value in $\SS^1 = \RR/\ZZ$.
Observe that $\vartheta(l,t) = \vartheta(l,0) + t$, and two
points $(l_1,t_1)$ and $(l_2,t_2)$ lie in the same connected component
of a fiber of $\vartheta$ if and only if they lie on the same leaf of
the Legendrian foliation. On a suitable collar neighborhood of the boundary,
the $1$-form $\eta$ simplifies to $\hbar\, dt$, so the behavior
of~$\vartheta$ is also as claimed.

We now explain how to enforce the rationality assumption by perturbation of the
Liouville structure.
Suppose $\beta_0$ is any Liouville form adapted to $L$, in which case
$\restricted{\beta_0}{TL}$ is a closed $1$-form that vanishes on a
collar neighborhood of~$\p L$.
For every $\epsilon > 0$, we will find a closed $1$-form~$\lambda_L$ on~$L$
with compact support in $\mathring L$ and
$\norm{\lambda_L} < \epsilon$ (in the $\mathcal{C}^0$-norm with respect to a 
fixed auxiliary metric on $L$) such that $i^*\beta_0 + \lambda_L$
represents a rational cohomology class on~$L$.
Since the restriction of $\beta_0$ to $L$ vanishes near $\p L$, its cohomology
class belongs to the kernel $K$ of the map $H^1_{\dR}(L) \to
H^1_{\dR}(\p L)$ induced by inclusion.
Let $\alpha_1, \dotsc, \alpha_p$ be a set of closed $1$-forms representing
a basis of the image in 
$K$ of $H^1(L; \ZZ)$. By the definition of $K$, we can assume that all these 
$1$-forms vanish near the boundary of $L$. The restriction of $\beta_0$ to $L$
can be written as $\sum c_i \alpha_i + df$ for some real coefficients $c_i$ 
and some function $f$. Since $\QQ$ is dense in $\RR$, one can find arbitrarily
small numbers $\varepsilon_i$ such that $c_i + \varepsilon_i$ is rational for
all $i$ and then set $\lambda_L = \sum \varepsilon_i \alpha_i$.

We extend $\lambda_L$ to a tubular neighborhood~$U$ of $L$ in $\Sigma$ by
pulling it back to the normal bundle, and multiply it by a fixed
cutoff function $\rho\colon U \to [0,1]$ that has compact support on
$U$ and equals $1$ on~$L$.
In this way we obtain a $1$-form $\beta_0'$ given by $\beta_0 +
\rho\lambda_L$ on $U$ that extends smoothly to $\beta_0$ on
$\Sigma\setminus U$, and whose restriction to $L$ yields the desired
closed $1$-form with compact support in $\mathring L$ that represents
a rational cohomology class.
We can choose $\epsilon$ above arbitrarily small, hence we can
assume that all forms in the segment between $d\beta_0$ and
$d\beta_0'$ are symplectic.
The corresponding contact structures are then isotopic relative to
$\Sigma \setminus U$ and $\p \Sigma$.
\end{proof}

\begin{proof}[Proof of Proposition \ref{prop:bLob_in_blown_down_domain}]
  Let $(\Sigma,\omega, \xi_\p)$ denote the ideal Liouville domain used to construct
  $G$.  We will construct a Lagrangian $L\subset \Sigma$ with cylindrical end
	and blow down the foliated submanifold of Lemma~\ref{lemma:lag_to_blob} to
	find the desired \BLOB.
	If $\dim \Sigma = 2$, it suffices to take for~$L$ an embedded path
	between two distinct boundary components of~$\Sigma$, where one corresponds
	to a blown down boundary component of~$G$ and the other does not.
	More generally, choose two disjoint boundary parallel Lagrangian disks
	$L_\mathrm{bd}$ and $L_\mathrm{p}$ with cylindrical ends in $\Sigma$ such that
	$\p L_\mathrm{bd}$ is a Legendrian sphere in one of the blown down boundary
	components of $\p\Sigma$, and $\p L_\mathrm{p}$ is a Legendrian sphere in another
	boundary component that is not blown down.
	By a symplectic isotopy supported in a tube connecting them, we can deform
	$L_\mathrm{p}$ away from $\p L_\mathrm{p}$ so that it intersects
	$L_\mathrm{bd}$ transversely.

	One can remove transverse self-intersection points between two
	Lagrangians $L$ and $L'$ using \cite{PolterovichLagrangianSurgery}.  This
	construction works by removing for each intersection two small balls
	from $L$ and $L'$ containing this point, and gluing in a tube
	diffeomorphic to $[-\epsilon,\epsilon]\times \SS^{n-1}$ joining the
	boundaries of the two balls.  In fact, the construction is explicit:
	choose a Darboux chart around the intersection point such that
	$L$ and $L'$ are represented by the $n$-planes
	$\{(x_1,\dotsc,x_n,0,\dotsc,0)\}$ and $\{(0,\dotsc,0,
	y_1,\dotsc,y_n)\}$ respectively.  Remove a disk of radius $\epsilon$ 
	around $0$ in both planes and glue in the tube
	\begin{equation*}
		(-\epsilon,\epsilon) \times \SS^{n-1} \hookrightarrow \RR^{2n},\,
		(t;x_1,\dotsc,x_n) \mapsto \bigl(\rho_1(t)\cdot (x_1, \dotsc,
		x_n); \rho_2(t)\cdot (x_1, \dotsc, x_n)\bigr)
	\end{equation*}
	for a smooth function $\rho_1 \colon (-\epsilon,\epsilon) \to [0,1]$
	that is $0$ for values between $-\epsilon$ and $-\epsilon/2$, has
	positive derivative for $t > -\epsilon/2$ and is the identity
	close to $+\epsilon$.
        Define $\rho_2(t) := \rho_1(-t)$.
  This defines a Lagrangian manifold that glues
	well to $L\setminus \epsilon\cdot \DD^n$ for $t$ close to $\epsilon$ and to
	$L'\setminus \epsilon\cdot\DD^n$ for $t$ close to~$-\epsilon$.
	The symplectic isotopy and the surgery process both took place
        far away from the boundary, so we obtain by this
        construction a Lagrangian that still has cylindrical ends.
	Lemma~\ref{lemma:lag_to_blob} then produces a foliated submanifold which
	becomes a \BLOB in the blown down Giroux domain.
  This \BLOB also embeds into a ball,
	because $L$ is obtained from two Lagrangian disks parallel to the boundary and
	a thin tube that lies in the neighborhood of an embedded path, so that $L$
	lies in a ball of the form $[0,1]\times \DD^{2n-1} \subset \Sigma$.
  Moving to the contactization and blowing down the
  corresponding boundary components then gives a neighborhood
  diffeomorphic to a ball $\DD^2 \times \DD^{2n-1}$ that contains the
  \BLOB.
\end{proof}

\subsection{Convex hypersurfaces and gluing}
Recall that a hypersurface $\Sigma$ in a contact manifold $(V, \xi)$ is
said to be \defin{$\xi$-convex} (cf.~\cite{Giroux_91}) if there is 
a contact vector field $X$ transverse to
$\Sigma$. In this situation, the \defin{dividing set} associated to $\Sigma$ and
$X$ is the hypersurface $\Gamma$ in $\Sigma$ where $X$ is tangent to $\xi$.
Observe that since the vector field $X$ corresponds to a ``contact
Hamiltonian'' function and can thus be cut off
away from $\Sigma$, its flow identifies a neighborhood of $\Sigma$ with 
$\Sigma \times \RR$, with $\xi$ defined by $\gamma + u\, dt$ where $t$ denotes the
coordinate on $\RR$, $\gamma$ is a $1$-form on $\Sigma$ and $u$ a function on
$\Sigma$. 
It follows from the computations in \cite{Giroux_91} that the $1$-form $\gamma$
induces a contact structure $\xi_\Gamma$ on the dividing set $\Gamma$ and, if
$\Sigma_0$ denotes the closure of a connected component of $\Sigma \setminus
\Gamma$, $\bigl(\Sigma_0, d(\gamma/u), \xi_\Gamma\bigr)$ is an ideal Liouville
domain whose contactization is $\Sigma_0 \times \RR$ equipped with the
restriction of~$\xi$.

If $\Sigma$ is closed then $\Gamma$ cannot be empty, otherwise $\Sigma$ would be
a closed exact symplectic manifold, contradicting Stokes' theorem.  So in this
case, $\Sigma \setminus \Gamma$ has at least one component on which $u$ is
positive and one where it is negative.  One can then see $\Sigma \times \RR$ as
several contactizations of ideal Liouville domains glued together. Going in the
opposite direction, we can take advantage of the fact that boundary 
components of Giroux domains are $\xi$-round hypersurfaces and use 
Lemma~\ref{lemma:round_neighborhood} to glue together any two
Giroux domains along boundary components modeled on
isomorphic contact manifolds.

\begin{remark}
\label{remark:convexLiouville}
One can check that the ideal Liouville domain $(\Sigma_0,d(\gamma/u),
\xi_\Gamma)$ defined above
depends only on the contact structure and contact vector field
near~$\Sigma$, not on the choice of contact form.  For an arbitrary
choice of contact form, one cannot expect $d\gamma$ itself to be symplectic
everywhere on $\Sigma \setminus\Gamma$, but analogously to
Remark~\ref{remark:fbeta}, one can always choose a contact form for which
this is true.  The surface $\Sigma \times \{\text{const}\} \subset
\Sigma \times \SS^1$ in Example~\ref{ex:torsionDomain} provides a
popular example.
\end{remark}

\subsection{Obstructions to fillability}

We now want to state a non-fillability result.
As preparation, note that any embedding of the interior of a Giroux domain
$I_\Sigma := \IntSig \times \SS^1$ into a contact manifold
$(V, \xi)$ determines a distinguished subspace
$H_1(\Sigma; \RR) \otimes H_1(\SS^1; \RR) \subset H_2(V ; \RR)$.
We call its annihilator in $H^2_{\dR}(V)$ the space of cohomology classes
\defin{obstructed} by $I_\Sigma$, and we denote it by $\oO(I_\Sigma)$.
Classes in $\oO(I_\Sigma)$ are exactly those whose restriction to $I_\Sigma$ can
be represented by closed $2$-forms pulled back from the interior of $\Sigma$.
If $N \subset (V,\xi)$ is any subdomain resulting from gluing together a
collection of Giroux domains
$I_{\Sigma_1},\dotsc,I_{\Sigma_k}$ and blowing down some of their boundary
components, then we define its obstructed subspace $\oO(N) \subset H^2_{\dR}(V)$
to be $\oO(I_{\Sigma_1}) \cap \dotsm \cap \oO(I_{\Sigma_k})$.  We will say that
such a domain is \defin{fully obstructing} if $\oO(N) = H^2_{\dR}(V)$.

\begin{example}
\label{ex:Lutz}
If $\Sigma$ is homeomorphic to $[-1, 1] \times M$ for some closed manifold $M$,
and $N$ is the result of blowing down one boundary component of the
Giroux domain $I_\Sigma$, then \emph{any} embedding of~$N$ is fully obstructing.
Indeed, any class in $H_1(\Sigma;\RR) \otimes H_1(\SS^1;\RR)$ 
can be represented by a cycle in
the $M \times \DD^2$ part of the blown down Giroux domain and, of course,
$H_1(\SS^1;\RR)$ becomes trivial in $H_1(\DD^2;\RR)$.  For instance,
a Lutz tube (see Example~\ref{ex:torsionLutz})
in a contact $3$-manifold is always fully obstructing, and the same is
true for the higher dimensional generalization that we will discuss
in~\S\ref{sec:torsion}.
\end{example}

\begin{theorem}
\label{thm:torsion}
Suppose $(V,\xi)$ is a closed contact manifold containing a
subdomain~$N$ with nonempty boundary, which is obtained by gluing and
blowing down Giroux domains.
\begin{itemize}
\item[(a)] If $N$ has at least one blown down boundary component then
it contains a small \BLOB, hence $(V,\xi)$ does not have any (semipositive) 
weak filling.
\item[(b)] If $N$ contains two Giroux domains $\Sigma^+\times \SS^1$
  and $\Sigma^-\times \SS^1$ glued together such that $\Sigma^-$ has a
  boundary component not touching $\Sigma^+$, then $(V,\xi)$ has no
  (semipositive) weak filling $(W,\omega)$ with
  $[\omega_V] \in \oO(\Sigma^+ \times \SS^1)$.
\end{itemize}
In particular $(V,\xi)$ has no (semipositive) strong filling in either case.
\end{theorem}

The first statement in this theorem follows immediately from
Proposition~\ref{prop:bLob_in_blown_down_domain} and
Theorem~\ref{thm:plastikstufe_and_weak_fillability}.
We will prove the second in \S\ref{sec:applications}, essentially by using the
symplectic cobordism construction of the next section to reduce it to
the first statement, though some care must be taken because the filling 
obtained by attaching our cobordism to a given
semipositive filling need not always be
semipositive.  We will also give in \S\ref{sec:applications}
an alternative argument for both parts of Theorem~\ref{thm:torsion}
using $J$-holomorphic spheres: this requires slightly stricter homological
assumptions than stated above, but has the advantage of not requiring
semipositivity at all, due to the polyfold machinery recently developed
in \cite{HoferWZ_GW}.

Without delving into the details, we should mention that we also expect
the above filling obstruction to be detected algebraically in
Symplectic Field Theory via the notion of \emph{algebraic torsion}
defined in \cite{LatschevWendl}.  Recall that a contact manifold is said
to be \emph{algebraically overtwisted} if it has algebraic $0$-torsion
(this is equivalent to having vanishing contact homology), but there are
also infinitely many ``higher order'' filling obstructions known as
algebraic \emph{$k$-torsion} for integers $k \ge 1$.  It turns out that
one can always choose the data on a Giroux domain $\Sigma \times \SS^1$
so that gradient flow lines of a Morse function on~$\Sigma$ give rise to
holomorphic curves in the symplectization of $\Sigma \times \SS^1$, and these
can be counted in SFT.  The expected result is as follows:

\begin{conjecture}
\label{conj:algebraic}
Suppose $(V,\xi)$ contains a subdomain~$N$ as in Theorem~\ref{thm:torsion},
choose any $c \in \oO(N)$ and consider SFT with coefficients
in $\RR[H_2(V;\RR) / \ker c]$.  Then $(V,\xi)$ has algebraic $1$-torsion,
and it is also algebraically overtwisted
if~$N$ contains any blown down boundary components.
\end{conjecture}

\section{Surgery along Giroux domains}
\label{section:surgery}

\subsection{A handle attachment theorem}

In this section, we explain a surgery procedure which removes the interior of a
Giroux domain from a contact manifold and blows down the resulting boundary.
This surgery corresponds to a symplectic cobordism that can be glued on
top of any weak filling satisfying suitable cohomological conditions, 
leading to a proof of Theorem~\ref{thm:torsion}.

Suppose 
$(V,\xi)$ is a $(2n-1)$-dimensional
contact manifold without boundary, containing a
Giroux domain $G \subset V$, possibly with some boundary components
blown down.  Removing the interior of~$G$, the boundary of
$\overline{V \setminus G}$ is then a $\xi$-round hypersurface
\begin{equation*}
  \p(\overline{V \setminus G}) = M \times \SS^1 \;,
\end{equation*}
where $(M,\xi_M)$ is a (possibly disconnected) closed contact manifold.
We can thus blow it down as described in \S\ref{subsec:round}, 
producing a new manifold
\begin{equation*}
  V' := (\overline{V \setminus G}) \blowdown M
\end{equation*}
without boundary, which inherits a natural contact structure~$\xi'$.

Topologically, the surgery taking $(V,\xi)$ to $(V',\xi')$ can be
understood as a certain handle attachment.
We now give a point-set description of this handle attachment which
is sufficient to state the theorem below, and postpone the smooth
description to the next subsection.
Assume that $G$ is obtained from the ideal Liouville domain
$\bigl(\Sigma, \omega, \xi_\Sigma\bigr)$ with boundary $\p \Sigma =
M_\mathrm{p} \sqcup M_\mathrm{bd}$ by blowing down the Giroux domain
$\Sigma \times \SS^1$ at $M_\mathrm{bd} \times \SS^1$ but preserving
$M_\mathrm{p} \times \SS^1$ as in Fig.~\ref{fig:BlowDownSketch}
(here \textsl{bd} stands for ``blown down'', and \textsl{p} 
for ``preserved'').
Then topologically,
\begin{equation*}
  G = \bigl(M_\mathrm{bd} \times \DD^2\bigr)
  \cup_{M_\mathrm{bd} \times \SS^1} (\Sigma \times \SS^1) \;.
\end{equation*}
Note that $M_\textrm{bd}$ can now be regarded as a codimension~$2$
contact submanifold of~$G$, namely by identifying it with
$M_\textrm{bd} \times \{0\}$.
Next, remove a small open collar neighborhood of $M_\textrm{bd}$ from
$\Sigma$ and denote the resulting submanifold by~$\Sigma_h$.
We can regard $\Sigma_h\times \SS^1$ as a subdomain of $G$, and
consider the manifold with boundary and corners defined by
\begin{equation*}
  \bigl([0,1] \times V\bigr) \cup_{\{1\} \times (\Sigma_h \times \SS^1)}
  (\Sigma_h \times \DD^2) \;.
\end{equation*}
After smoothing the corners, this becomes a smooth oriented
cobordism~$\Wcob$ with boundary 
(see Fig.~\ref{fig:GirouxDomainPresurgery}),
\begin{equation*}
  \p \Wcob = - V \sqcup V' \sqcup (M_\mathrm{bd} \times \SS^2) \;.
\end{equation*}

\begin{figure}[htbp]
  \centering \subfigure[The domain $G$ is obtained from the product
  manifold $\Sigma \times \SS^1$ by blowing down the boundary
  components $M_\mathrm{bd}\times \SS^1$ to
  $M_\mathrm{bd}$.\label{fig:BlowDownSketch}]
  {\includegraphics[width=0.3\textwidth,keepaspectratio]{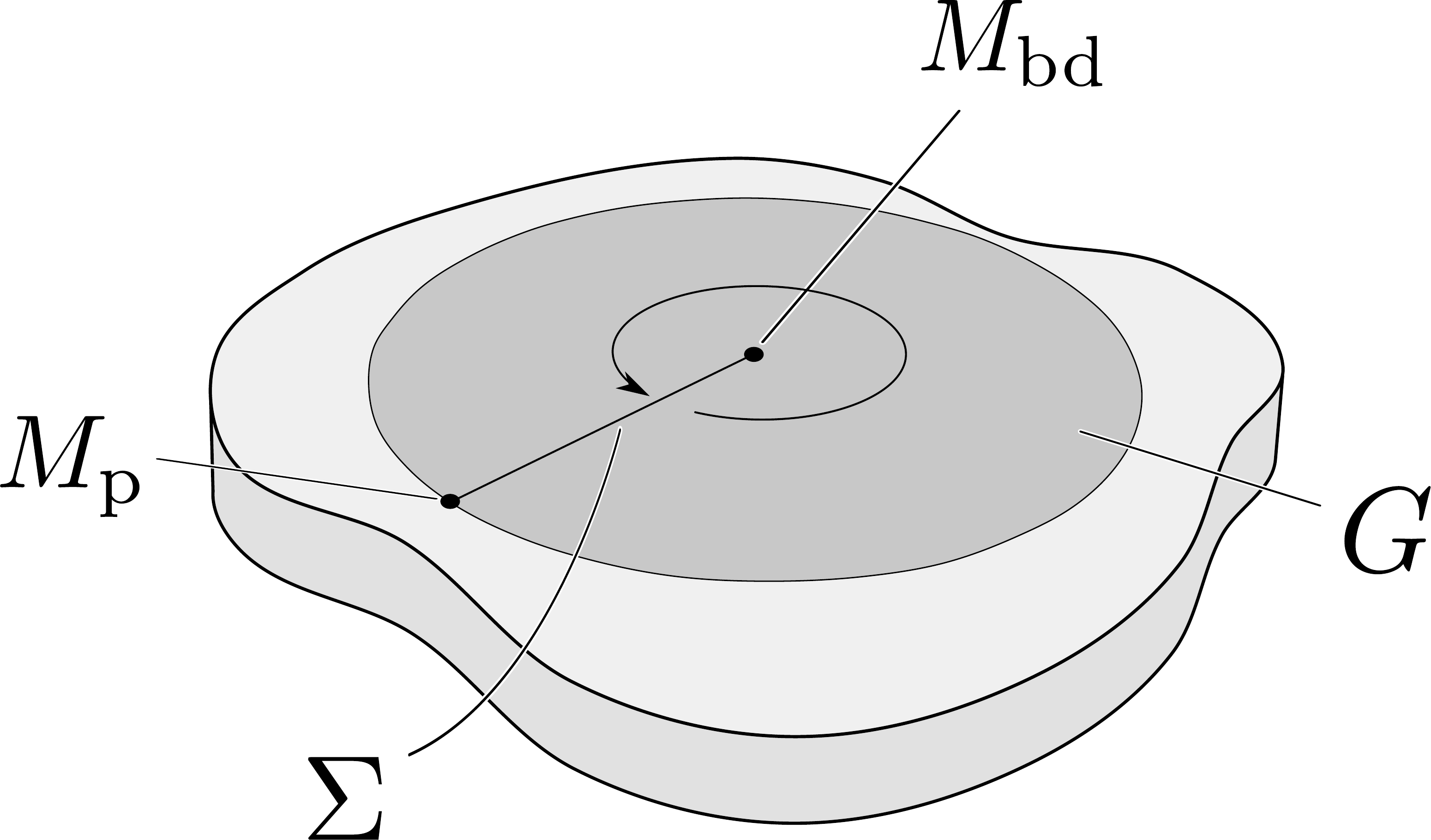}}
  \qquad
  \subfigure[The cobordism is obtained by gluing
  $\Sigma_h\times \DD^2$ onto $G$, and rounding its corners.
  Note that after the handle attachment the boundary of the surgered
  manifold consists of the contact manifold $V'$ plus components
  diffeomorphic to $M_\mathrm{bd} \times \SS^2$ corresponding to the
  blown down boundary of~$G$.\label{fig:GirouxDomainPresurgery}]
  {\includegraphics[width=0.6\textwidth,keepaspectratio]{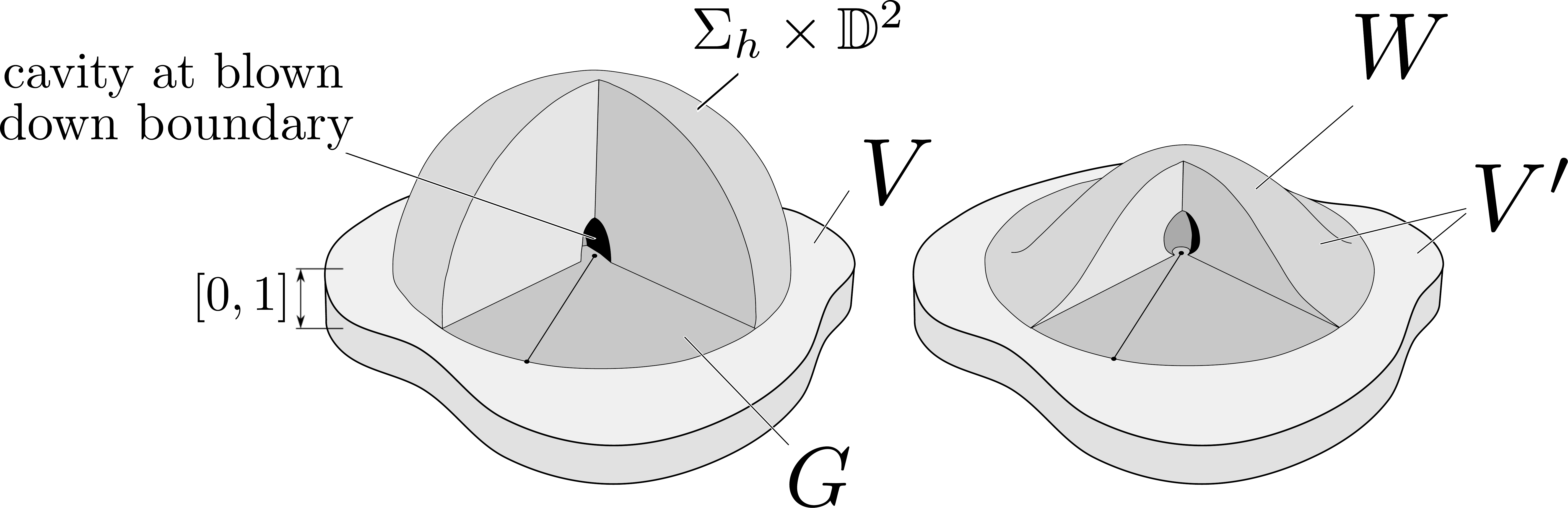}}
  \caption{}
\end{figure}

We can now state the main theorem of this section.
\begin{theorem}
\label{thm:cobordism}
Suppose $\Wcob$ denotes the $2n$-dimensional smooth cobordism described above, 
and $\Omega$ is a closed $2$-form on $V$ such that:
\begin{itemize}
\item
$\Omega$ weakly dominates $\xi$
\item
the cohomology class of $\Omega$ belongs to the obstructed subspace 
$\oO(G)$, i.e.~for every $1$-cycle $Z$ in~$\Sigma$,
\begin{equation*}
  \int_{Z \times \SS^1} \Omega = 0 \;.
\end{equation*}
\end{itemize}
Then $\Wcob$ admits a symplectic structure~$\omega$ with the
following properties:
\begin{enumerate}
\item $\restricted{\omega}{TV} = \Omega$.
\item The co-core $\Sigma_h \times \{0\} \subset \Sigma_h \times \DD^2
  \subset \Wcob$ is a symplectic submanifold weakly filling $(\p
  \Sigma_h \times \{0\}, \xi_\Sigma)$.
\item $(V',\xi')$ is a weakly filled boundary component of
  $(\Wcob, \omega)$ that is contactomorphic to the blown down
  manifold $(\overline{V \setminus G}) \blowdown M_\mathrm{p}$.
\item A neighborhood of $M_\mathrm{bd} \times \SS^2 \subset \p \Wcob$ in
$(\Wcob, \omega)$ can be identified symplectically with 
\begin{equation*}
  \bigl((-\delta,0] \times M_\mathrm{bd} \times \SS^2,\,
  \omega_0 \oplus \omega_{\SS^2}\bigr)
\end{equation*}
for some $\delta > 0$, where 
$\omega_{\SS^2}$ is an area form on~$\SS^2$ and $\omega_0$ is a
symplectic form on $(-\delta,0] \times M_\mathrm{bd}$ for which the boundary
$(M_\mathrm{bd},\xi_\Sigma)$ is weakly filled.  Moreover, the intersection of
the co-core $\Sigma \times \{0\}$ with this neighborhood has the form
$(-\delta,0] \times M_\mathrm{bd} \times \{\operatorname{const}\}$.
\end{enumerate}
\end{theorem}

\begin{remark}
Recall that due to Lemma~\ref{lemma:model_tubular_neighborhood}, a pair
of weak symplectic cobordisms can be smoothly glued together along
a positive/negative pair of contactomorphic boundary components whenever
the symplectic forms restricted to these boundary components match.
Thus the symplectic cobordism of the above theorem can be glued on top of
any weak filling $(W, \omega)$ of $(V,\xi)$ for which 
$[\restricted{\omega}{TV}] \in \oO(G)$.
\end{remark}

\subsection{Construction of the symplectic cobordism}

In this section we will give the proof of Theorem~\ref{thm:cobordism}.
The proof will consist of the following five steps:
\begin{enumerate}
\item Find a standardized model with a special contact form~$\lambda$
  for tubular neighborhoods of $\partial G$ and the blown down
  components~$M_\textrm{bd}$.
\item Construct a symplectic form on our proto-cobordism~$[0, 1]
  \times V$ that is well adjusted to both $\Omega$ and $\lambda$.
\item Carve out the interior of $\{1\}\times \Sigma \times \SS^1$ from
  $[0, 1]\times V$.  This creates a notch with corners along its
  edges, and we will then smoothly glue the handle $\Sigma \times \DD^2$
  into the cavity, creating a smooth manifold.
\item Study the symplectic form induced from the proto-cobordism on the glued
	part of the handle and extend it to the whole handle.
\item Check that the new boundary of the cobordism has the desired
properties.
\end{enumerate}

\step{Neighborhoods and contact form for $G$} \ \\
For simplicity, we first pretend that $G$ is a Giroux domain 
$\Sigma \times \SS^1$ without blown down boundary components.
Consider a collar neighborhood $(0, 1 ] \times \partial \Sigma$ associated to
some Liouville form $\beta$ by Lemma~\ref{lemma:idealCollar}
and denote by $\alpha$ the corresponding contact form on $\partial \Sigma$.
Let $s$ be the coordinate in $(0, 1]$.  We denote by $u$ 
a smooth function $\Sigma \to [0,1]$ which has the boundary
$\p\Sigma = u^{-1}(0)$ as a regular level set, equals $1 - s$ in the region 
$s \geq 3/4$ and $1$ in the region $s \leq 1/4$ and outside the collar, and 
satisfies $u'\le 0$ everywhere on the collar (see Fig.~\ref{fig:graphs_u_f}).
We set $\gamma = u\beta$.
The contact form on $G$ associated to $\beta$ and $u$ is then 
$\lambda := u\beta + u\,d\theta = \gamma + u \,d\theta$, where $\theta$
denotes the coordinate on~$\SS^1$. In the collar one can
set $f := u/(1 - s)$ so that $\lambda = f \alpha + u\, d\theta$. Note that the
contact condition in $(0, 1] \times \p \Sigma \times \SS^1$ is equivalent to
\begin{equation}
\label{eqn:contactNew}
f\,(f'u -u'f) > 0 \;,
\end{equation}
so appealing to Lemma~\ref{lemma:round_neighborhood}, we can slightly
extend our collar neighborhood embedded in $(V,\xi)$ to one of the form
$(0, 1 + \varepsilon] \times \p\Sigma \times \SS^1$, with $\lambda$
written as above and $u$ extended as $1 - s$ when $s > 1$.

\begin{figure}[htbp]
  \centering
  \includegraphics[height=3cm,keepaspectratio]{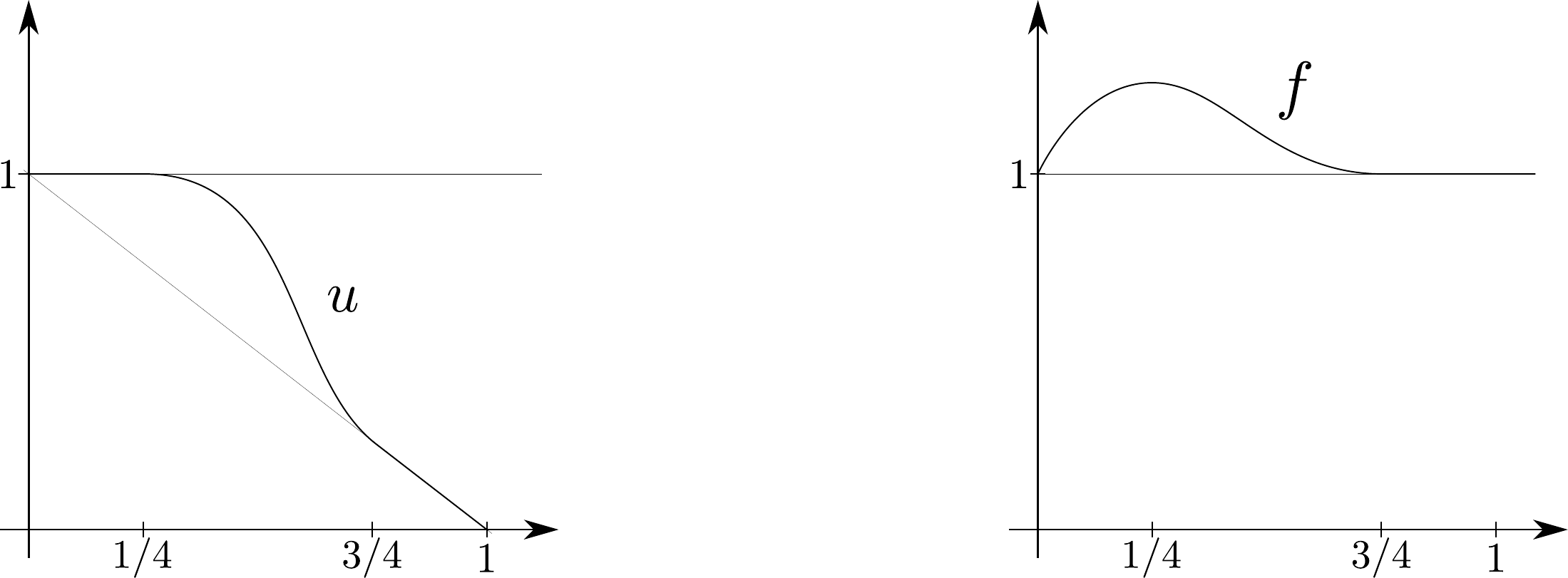}
\caption{The functions $u$ and $f$.}\label{fig:graphs_u_f}
\end{figure}

In the more general case where $G$ is a Giroux domain with some boundary
components blown down, the function $u$ becomes $r^2$ in 
$M_\mathrm{bd} \times  \DD^2$, so $\lambda$ is still a smooth contact form
(but of course there is no extended collar).

\step{The symplectic form on $[0, 1] \times V$} \ \\
The assumption that $\Omega$ weakly dominates $\xi$ implies that the $2$-form
$\omega = d(t\lambda) + \Omega$ is symplectic on $(-\delta, 1] \times V$
for some small positive constant $\delta$.
The cohomological assumption $[\Omega] \in \oO(G)$ implies that $\Omega$ is
cohomologous to some $2$-form 
$\omega_0$ such that $\restricted{\omega_0}{G}$ is the pull
back of a $2$-form on $\Sigma$. In addition, since the collar neighborhood
$(0,1] \times \p\Sigma$ retracts to $\p\Sigma$, we can assume that
$\iota_{\partial_s} \omega_0 = 0$ when $s \geq 1/4$.

\begin{lemma}
\label{lemma:collarDeformationNew}
We can modify the form $\omega$ defined above
to a new symplectic form on $(-\delta, 1]\times V$,
keeping the assumption that $\omega$ restrict to $\Omega$ on $\{0\} \times V$
and $\xi$ be weakly dominated by~$\omega$ on each slice 
$\{ t\} \times V$, but asking in addition that $\omega$ restrict to
$C\, d(t\lambda) + \omega_0$ on $[1/2, 1] \times V$
for some large constant~$C > 0$.
\end{lemma}
\begin{proof}
Using Lemma~\ref{thm:cohomologous_deformation_positive_end}, we find
a symplectic form $\omega'$ on $(-\delta, \infty) \times V$ such that
each $\{ t \} \times V$ is still weakly filled
and $\omega'$ restricts to $d(t\lambda) + \omega_0$ for $t$ greater than some
large constant $C/2$.
The scaling diffeomorphism $(t,v) \mapsto (t/C, v)$ pulls back $\omega'$
to the desired symplectic form.
\end{proof}

\step{Handle attachment} \ \\
We now give a smooth description of the handle attachment which
is compatible with the smooth description of the blow-down process for
$\xi$-round hypersurfaces.
For this, we will first create a small basin in the top of
$[0,1]\times V$ to which we can glue in the handle.

\begin{figure}[htbp]
  \centering \subfigure[{The precise construction of the handle
    attachment sketched in Fig.~\ref{fig:GirouxDomainPresurgery} works
    by creating a trench on the top side of the cobordism $[0,1]
    \times V$ to which we can glue in the handle.
  In the picture above we need to remove the hatched area, which
  corresponds to the Giroux domain~$\Sigma \times \SS^1$ minus a small
  neighborhood of the blown down boundary.}\label{fig:hole1}]
  {\includegraphics[width=0.45\textwidth,keepaspectratio]{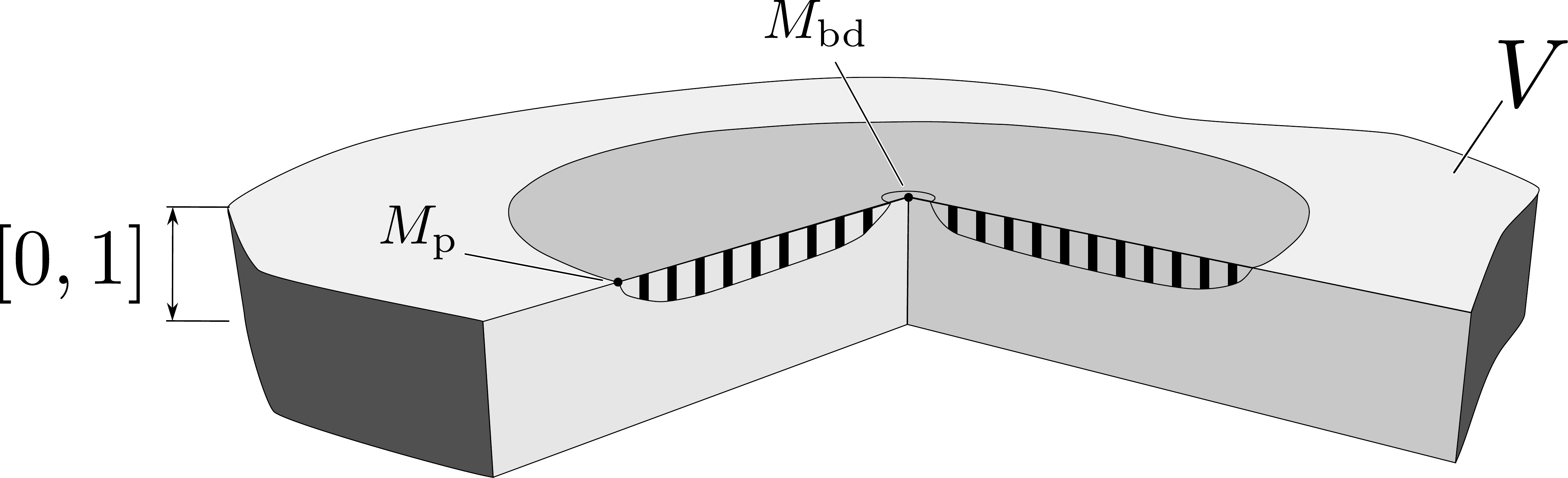}}
  \qquad
	\subfigure[{The vector field $X$ is tangent to the top face and transverse to
	the hypersurface $\hH$, which is $\Sigma_h$ pushed inside $[0, 1] \times
	\Sigma$ relative to its boundary. Everything above $\hH$ has been discarded to
	make room for the handle.}\label{fig:hole2}]{\includegraphics[width=0.45\textwidth,keepaspectratio]{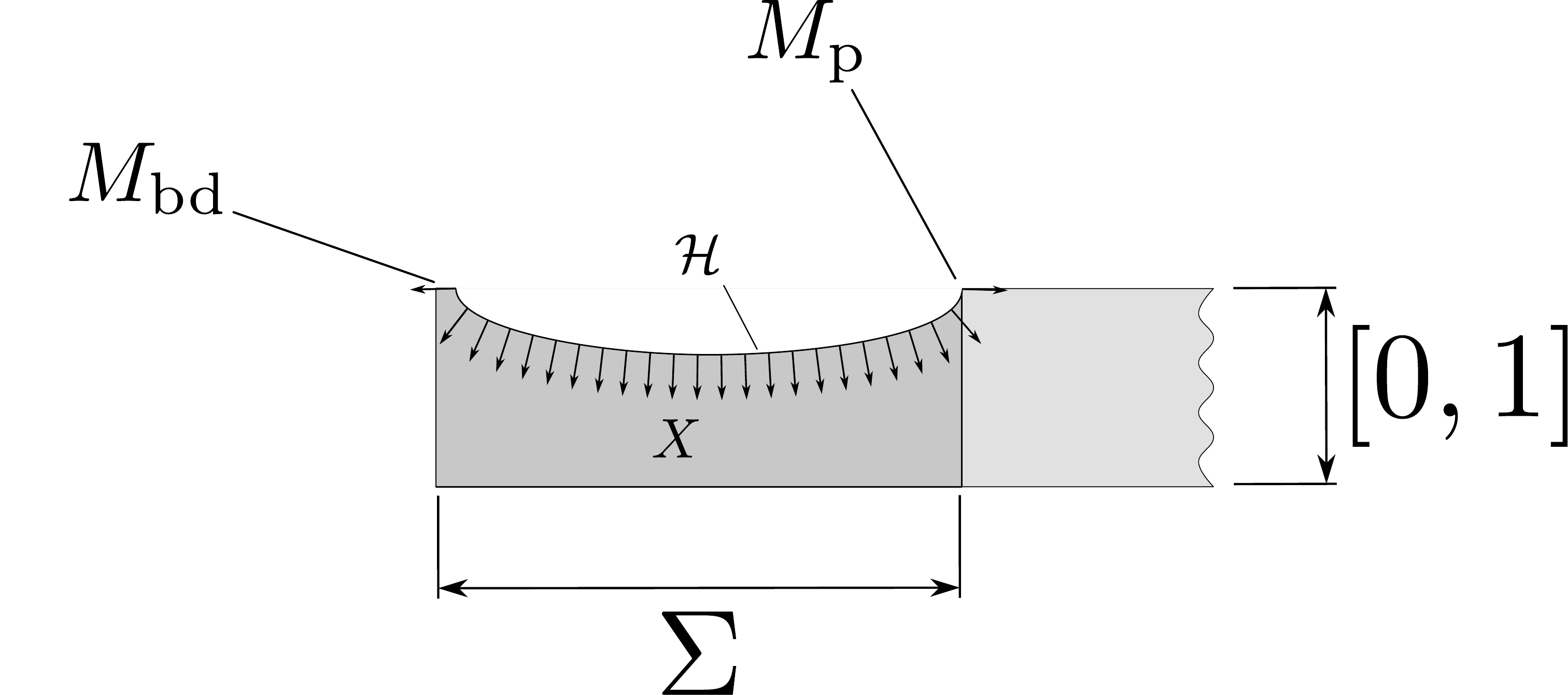}}
  \caption{}\label{fig:hole}
\end{figure}

Recall that $M_\mathrm{bd}$ denotes the blown down boundary components of the
Liouville domain~$\Sigma$, and $M_\mathrm{p}$ denotes the other components.
Let $h$ be a smooth function from $\Sigma$ to $(1/2, \infty)$ such that
\begin{itemize}
\item $h$ restricts on the special collar of Step~1 to a function only
  depending on $s$ with nonnegative derivative $h'(s)$,
\item $h$ is constant on $(0,1/4)\times \p\Sigma$ and outside the
  boundary collar,
\item For $s \ge 1 - \epsilon$, 
$h(s) = s$ near $M_\mathrm{p}$, and $h(s) = s + \epsilon$ near $M_\mathrm{bd}$.
\end{itemize}
We denote by $\Sigma_h$ the subset of $\Sigma$ on which
$h$ is less than or equal to~$1$, and by $\hH \subset [0, 1] \times
V$ the product of $\SS^1$ with the graph of $h$ over $\Sigma_h$,
see Fig.~\ref{fig:hole}.
We discard the region $\{ t \geq h \}$ from $[0,1]\times V$ to get an
open manifold, to which we will glue the
``handle'' $\Sigma_h \times D$.
Here $D$ denotes the disk around the origin in $\RR^2$ with radius
$\sqrt\varepsilon$.
In the following, we will find a symplectic vector field $X$ in a neighborhood of
the hypersurface $\hH$ in $[0,1]\times V$ that is transverse to $\hH$,
never points in the positive $t$-direction, and is tangent
to $\{1\}\times V$ near the boundary of~$\hH$.
Shrinking $\epsilon$ if needed, we may assume that the flow of $X$ starting from
$\hH$ embeds $\hH \times [0, \varepsilon]$ into $[1/2, 1] \times V$.
We denote by $\flow^X_\tau$ the flow of $X$ at time $\tau$.
The manifold $W'$ is obtained by attaching $\Sigma_h \times D$ to
$\bigl([0,1]\times V\bigr) \setminus \{ t \geq h \}$ using the gluing
map $\Psi$ from $\Sigma_h \times D^*$ (where $D^* = D \setminus
\{0\}$) to $[0,1]\times V$ defined by
\begin{equation*}
  \Psi(\sigma, re^{i\theta}) = \flow^X_{r^2}\bigl(h(\sigma),
  \sigma, \theta\bigr) \;.
\end{equation*}
Note that as a point-set operation, the handle attachement reduces to
the operation of adding the co-core $\Sigma_h \times \{0\}$ to the open manifold
$\bigl([0,1]\times V\bigr) \setminus \{ t \geq h \}$.

The vector field $X$ that we will use below coincides with $\partial_s$ near
$ \{ 1 \} \times \p G$.
This implies that the attachment using $\Psi$ restricts precisely to
the gluing map used to blow down the $\xi$-round hypersurface 
$M_\mathrm{p} \times \SS^1$.

As a gluing vector field $X$, we choose the $\omega$-dual of $-C\, d\theta$, where $C$ is the constant appearing in~$\omega$.
Since this $1$-form is closed, $X$ is a symplectic vector field.

\begin{lemma}
The vector field $X$ is transverse to the hypersurface $\hH$ and coincides with
$\partial_s$ near $\{1\}\times \p G$ and $\{1\}\times M_\mathrm{bd}$. 
\end{lemma}

\begin{proof}
Away from the special collar neighborhoods considered in Step~1, 
$\lambda = d\theta + \beta$, and this gives $dt(X) = -1$.
Elsewhere, on the collars 
$[0, 1] \times \bigl( [1/4, 1] \times \p \Sigma \times \SS^1\bigr)$, 
we use the ansatz $X = X^t\, \p_t + X^s\, \p_s$.
Computing the interior product $\iota_X \omega$ using 
$\omega = C\, d(t(u\, d\theta + f\alpha)) + \omega_0$ and 
$\iota_{\partial s}\omega_0 = 0$, we find that $X$
is indeed $\omega$-dual to $-C\, d\theta$ provided
\begin{equation*}
  \begin{split}
    u X^t + t u' X^s &= -1 \;,\\
    f X^t + t f' X^s &= 0 \;.
  \end{split}
\end{equation*}
This system is everywhere nonsingular due to the contact
condition~\eqref{eqn:contactNew}.
For $s \ge 3/4$ and $t = 1$, we have $X = \p_s$ as promised.
For $s < 3/4$, the conditions $f(s) > 0$ and $f'(s) > 0$ imply $X^t <
0$ and $X^s > 0$, hence $X$ is transverse to $\hH$.
\end{proof}

\step{Symplectic form on the handle}
\begin{lemma}
	\label{lemma:Psiomega}
The gluing map $\Psi$ from $\Sigma_h \times D^*$ to $[0, 1] \times V$ pulls back
$\omega$ to
\begin{equation*}
\Psi^*\omega = 2C\omega_D + C\, d(hu) \wedge d\theta + \Omega_0
\end{equation*}
where $\omega_D := -r\, dr \wedge d\theta$ and
$\Omega_0$ is a symplectic form on $\Sigma_h$ which weakly fills 
$(\partial \Sigma_h, \ker \gamma)$.
\end{lemma}

\begin{proof}
Let $j_\hH$ denote the embedding  $\Sigma_h\times \SS^1 \to
\hH \subset [0,1]\times V, (\sigma, \theta) \mapsto \bigl(h(\sigma),
\sigma, \theta\bigr)$.
Then we can decompose $\Psi$ as $\Psi = \Phi \circ P$, where $P$ is the
map from $\Sigma_h \times D^*$ to $\Sigma_h \times (0, \varepsilon] \times
\SS^1$ defined by $P(\sigma, re^{i\theta}) = (\sigma, r^2, \theta)$ and
\begin{equation*}
  \Phi(\sigma, \tau, \theta) := \flow^X_\tau\bigl(h(\sigma), \sigma,
  \theta\bigr) = \flow^X_\tau\bigl(j_\hH(\sigma, \theta)\bigr) \;.
\end{equation*}
Using the fact that the flow of $\bigl(\flow^X_\tau\bigr)_*\partial_\tau = X$
preserves $\omega$ and $\iota_X \omega = -C\, d\theta$, we obtain
for the pull back
\begin{equation*}
  \Phi^*\omega = -C\, d\tau \wedge d\theta + j_\hH^*\omega \;,
\end{equation*}
and since the symplectic form $\omega$ is given in the range of
$j_\hH$ by $C\, d(t\lambda) + \omega_0$ with $\lambda = u\, d\theta +
\gamma$, we can compute
\begin{equation*}
  j_\hH^*\omega = C\, d(h\lambda) + \omega_0 = C\, d(hu) \wedge d\theta
  + \Omega_0 \;,
\end{equation*}
where we have set $\Omega_0 = C\, d(h\gamma) + \omega_0$ (which is a $2$-form
on $\Sigma_h$).

Now since $P^*d\tau = 2 r\,dr$, the only thing left to prove is that $\Omega_0$ is a
symplectic form which weakly fills $(\partial \Sigma_h, \ker \gamma)$.
Since $\omega_D$ is the only term in $\Psi^*\omega$ that contains a
$dr$-factor, and thus it follows that 
$(\Psi^*\omega)^n = 2nC\omega_D \wedge \Omega_0^{n-1} \ne 0$, we deduce that
$\Omega_0$ is symplectic.

The $2$-form $\Omega_0$ restricts on $\p\Sigma_h$ to $C\, d\gamma +
\omega_0$.
Recall that the weakly dominating condition on $\{1\} \times V$ means that
for any constant $\nu \ge 0$,
$\lambda \wedge \bigl( \omega + \nu\, d\lambda \bigr)^{n - 1} > 0$.
Restricting to $\{1\} \times G$, where $\lambda = u\, d\theta + \gamma$ and 
$\omega = C\, d\lambda + \omega_0$, this becomes:
\begin{equation*}
(u\, d\theta + \gamma) \wedge \left[(C + \nu)\, du \wedge d\theta 
+ \bigl( C\, d\gamma + \omega_0 + \nu\, d\gamma \bigr)\right]^{n - 1} > 0 \; ,
\end{equation*}
which we expand along $\{1\} \times \partial \Sigma_h \times \SS^1$ as
\begin{equation*}
(n - 1)(C + \nu)\,\gamma \wedge du \wedge d\theta 
\wedge \big( C\, d\gamma + \omega_0 + \nu\, d\gamma \big)^{n - 2} > 0 \;.
\end{equation*}
In particular, this proves that $\gamma \wedge (\Omega_0 + \nu\,
d\gamma)^{n - 2}$ never vanishes. In order to check that it has the correct
sign, it suffices to consider the case $\nu = 0$.
\end{proof}

To finish the construction of the symplectic cobordism, we want to
define a symplectic structure on $\Sigma_h \times D$ that agrees in a
neighborhood of the boundary $\Sigma_h \times \partial D$ with
$\Psi^*\omega$, and that has a split form near $\Sigma_h \times
\{0\}$.  Let $\rho_1$ and $\rho_2$ be functions from $[0, \sqrt\varepsilon]$ to
$\RR$ (constraints will be added later). We set:
\begin{align*}
\widetilde\omega &:=  2C\rho_1 \omega_D + C\, d(\rho_2 hu) \wedge d\theta +
\Omega_0\\
&= g\omega_D + C\rho_2\, d(hu) \wedge d\theta + \Omega_0
\quad \text{ with $g := \left(2\rho_1 - \frac{hu\rho_2'}r\right)C$.}
\end{align*}

We choose $\rho_1(r) = \rho_2(r) = 1$ for $r$ close to $\sqrt\varepsilon$, so 
that $\widetilde\omega$ extends $\Psi^*\omega$.  Near 0, we choose $\rho_1$ to
be a large positive constant and $\rho_2$ to vanish so that $\widetilde\omega$
makes sense near the center of $D$.
One can compute $\widetilde\omega^n = ng\omega_D \wedge \Omega_0^{n - 1}$.
Since $\Omega_0$ is symplectic on $\Sigma_h$, we see that $\widetilde\omega$
is symplectic as soon as $g$ is positive. This condition is arranged
by choosing $\rho_1$ sufficiently large away from $r = \sqrt\varepsilon$.

\step{Properties of the new boundary of $W$} \ \\
We now consider in turns the two types of new boundary components resulting from
the above construction:
$V'$ and $M_\mathrm{bd} \times \SS^2$.
Since $hu$ is constant on $\p \Sigma_h$, the restriction of
$\widetilde\omega$ to $M_\textrm{p} \times D$ is $g\omega_D +
\Omega_0$.
As we already noted, the gluing map $\Psi$ extends the one used to define
the blow-down, and the contact form on $V'$ is $\lambda = \gamma - r^2 d\theta$. 
Thus in order to check the weak filling condition, we only need compute, for
any constant $\nu \ge 0$,
\begin{equation*}
  \lambda \wedge (\widetilde\omega + \nu\, d\lambda)^{n - 1} 
  = (n - 1)(g + 2\nu)\, \omega_D \wedge 
  \gamma \wedge (\Omega_0 + \nu\, d\gamma)^{n - 2} .
\end{equation*}
This is indeed a positive volume form for any nonnegative $\nu$ because 
$(\Sigma_h, \Omega_0)$ is a weak filling of $(M_\textrm{p}, \ker \gamma)$
according to Lemma~\ref{lemma:Psiomega}.

The situation along $M_\textrm{bd} \times \SS^2$ is very similar. 
There $\widetilde\omega$ induces $\omega_{\SS^2} + \Omega_0$ for some
area form $\omega_{\SS^2}$. The distribution we consider is $\ker \gamma$.
We compute:
\begin{equation*}
\gamma \wedge (\widetilde\omega + \nu\, d\gamma)^{n - 1} 
= (n - 1)\, \omega_{\SS^2} \wedge 
\gamma \wedge (\Omega_0 + \nu\, d\gamma)^{n - 2}
\end{equation*}
so the restriction of $\widetilde\omega$ is symplectic on $\ker \gamma$
because of Lemma~\ref{lemma:Psiomega}.
Lemma~\ref{lemma:model_tubular_neighborhood} then gives us a neighborhood of
$M_\textrm{bd} \times \SS^2$ in $(W', \widetilde\omega)$
that can be identified symplectically with
\begin{equation*}
  \bigl((-\epsilon,0] \times  M_\textrm{bd} \times \SS^2,
  \left( d(t\gamma) + \Omega_0 \right) \oplus 
  \omega_{\SS^2}\bigr)
\end{equation*}
for $\epsilon > 0$ sufficiently small.  Observe also that $\widetilde\omega$
already takes this split form in a neighborhood of the co-core
$\Sigma_h \times \{0\} \subset W'$, so we can apply the extension property of
Lemma~\ref{lemma:model_tubular_neighborhood} to get a collar whose intersection
with the co-core is precisely $(-\epsilon,0] \times M_\textrm{bd} \times \{0\}$
with $0 \in \DD \subset \SS^2$.

\section{Giroux domains and non-fillability}
\label{sec:applications}

We now use the cobordism of the preceding section to prove
Theorem~\ref{thm:torsion} on filling obstructions.  We will present two
slightly different approaches in \S\ref{sec:BLOB_in_handle} and
\S\ref{sec:hol_Spheres_in_handle} respectively: the first uses holomorphic
disks and the \BLOB, thus relying on a version of
Theorem~\ref{thm:plastikstufe_and_weak_fillability}.  The second 
approach uses holomorphic spheres and proves a slightly weaker result, as it 
requires stricter homological assumptions on the Giroux domains---though it 
should be mentioned that these assumptions are satisfied in all the 
interesting examples we know thus far, namely for the higher dimensional
notions of Lutz twists and Giroux torsion defined in~\S\ref{sec:torsion}.
The use of spheres instead of disks simplifies the proof in that it makes
the Freholm and compactness properties easier, while at the same time
allowing the use of the recently completed polyfold technology
\cite{HoferWZ_GW} to handle transversality issues.  For this reason the
second proof does not require semipositivity.

\subsection{Proof of Theorem~\ref{thm:torsion} via the \BLOB}
\label{sec:BLOB_in_handle}

Part~(a) of the theorem follows immediately from the fact that
if $(V,\xi)$ contains a Giroux domain~$N$ that has some boundary
components that are blown down and others that are not, then by
Proposition~\ref{prop:bLob_in_blown_down_domain} it
contains a small \BLOB, so Theorem~\ref{thm:plastikstufe_and_weak_fillability}
implies that $(V,\xi)$ does not admit any semipositive weak filling.
To prove part~(b), suppose $N$ has the form
\begin{equation*}
  N = (\Sigma^+ \times \SS^1) \cup_{Y \times \SS^1} 
  (\Sigma^- \times \SS^1) \;,
\end{equation*}
where $\Sigma^\pm$ are ideal Liouville domains with boundary $\p
\Sigma^\pm = \p_\mathrm{glue} \Sigma^\pm \sqcup \p_\mathrm{free}
\Sigma^\pm$, $Y := \p_\mathrm{glue} \Sigma^+ = \p_\mathrm{glue}
\Sigma^-$ carries the induced contact form~$\alpha$ and
$\p_\mathrm{free} \Sigma^-$ is not empty.
Arguing by contradiction, assume that $(V,\xi)$ is weakly filled by a
semipositive symplectic filling $(W_0,\omega)$ with
$[\restricted{\omega}{TV}] \in \oO(\Sigma^+)$.
This establishes the cohomological condition needed by
Theorem~\ref{thm:cobordism} on $\Sigma^+ \times \SS^1$, so applying
the theorem, we can enlarge $(W_0,\omega)$ by attaching $\Sigma^+
\times \DD^2$, producing a compact symplectic manifold $(W_1,\omega)$
whose boundary $(V',\xi')$ supports a contact structure that is weakly
filled.
Since the boundary~$V'$ of the new symplectic manifold~$(W_1,\omega)$
is contactomorphic to 
$(\overline{V \setminus (\Sigma^+ \times \SS^1)}) \blowdown Y$, we
find in $(V',\xi')$ a domain isomorphic to $(\Sigma^-\times \SS^1)
\blowdown Y$ that contains a small \BLOB.
Unfortunately this does not directly obstruct the existence of the
weak filling $(W_1,\omega)$, because even though $W_0$ was
semipositive, $W_1$ might not be.
We will follow the proof of
Theorem~\ref{thm:plastikstufe_and_weak_fillability}, with the
difference that we need to reconsider compactness to make sure that
bubbling is still a ``codimension~$2$ phenomenon''.
Choose an almost complex structure $J$ on $(W_1,\omega)$ with the
following properties:
\begin{itemize}
\item[(i)] $J$ is tamed by~$\omega$ and makes $(V', \xi')$ strictly
  $J$-convex,
\item[(ii)] $J$ is adapted to the \BLOB in the standard way, i.e.~it
is chosen close to the boundary of the \BLOB according to
  Lemma~\ref{lemma:complexStructureBoundaryBLOB} and in a neighborhood
  of the binding according to \cite{NiederkruegerPlastikstufe}
  (cf.~the proof of Theorem~\ref{thm:plastikstufe_and_weak_fillability}),
\item[(iii)] for some small radius $r > 0$, $J = J_{\Sigma^+} \oplus
  i$ on $\Sigma^+ \times \DD^2_r \subset W_1$, where $J_{\Sigma +}$ is
  a tamed almost complex structure on~$\Sigma^+$ for which
  $\p\Sigma^+$ is $J_{\Sigma^+}$-convex.
\end{itemize}
The third condition uses the fact from Theorem~\ref{thm:cobordism}
that the co-core $\cocore := \Sigma^+ \times \{0\}$ of the handle is a
symplectic (and now also $J$-holomorphic) hypersurface weakly filling
its boundary.
The binding of the \BLOB lies in the boundary of the
co-core~$\cocore_+$, and the normal form described in
\cite{NiederkruegerPlastikstufe} is compatible with the
splitting~$\Sigma^+ \times \DD^2_r$ so that (ii) and (iii) can be
simultaneously achieved.
By choosing $J_{\Sigma^+}$ generic, we can also assume that every
somewhere injective $J_{\Sigma^+}$-holomorphic curve in $\Sigma^+$
is Fredholm regular and thus has nonnegative index.  Note that
any closed $J$-holomorphic curve in $\Sigma^+ \times \DD^2_r$
is necessarily contained in $\Sigma^+ \times \{z\}$ for some
$z \in \DD^2_r$, and the index of this curve differs from its
index as a $J_{\Sigma^+}$-holomorphic curve in $\Sigma^+$ by
the Euler characteristic of its domain.  This implies that every
somewhere injective $J$-holomorphic sphere contained in
$\Sigma^+ \times \DD^2_r$ has index at least~$2$.
Likewise, by a generic perturbation of~$J$ outside of this
neighborhood we may assume all somewhere injective curves that are 
\emph{not} contained entirely in $\Sigma^+ \times \DD^2_r$ also
have nonnegative index.

Now let $\mM$ be the connected moduli space of holomorphic disks attached to 
the \BLOB that contains the standard Bishop family.
We can cap off every holomorphic disk~$u \in \mM$ by attaching a
smooth disk that lies in the \BLOB, producing a trivial homology class
in $H_2(W_1)$.
The cap and the co-core intersect exactly once, and it follows that $u$ also
must intersect the co-core~$\cocore_+$ exactly once, because $u$ and
$\cocore_+$ are both $J$-complex.
To finish the proof, we have to study the compactness of $\mM$ and argue
that $\overline{\mM} \setminus \mM$ consists of strata of codimension
at least~$2$.
A nodal disk~$u_\infty$ lying in $\overline{\mM} \setminus \mM$ has
exactly one disk component $u_0$, which is injective at the boundary,
and one component $u_+$ that
intersects the co-core once; either $u_+ = u_0$ or $u_+$ is a holomorphic
sphere.
Every other nonconstant connected component~$v$ is a holomorphic
sphere whose homology class has vanishing intersection with the relative class
$[\cocore_+]$. So either $v$ does not intersect the $J$-complex
submanifold~$\cocore_+$ at all or $v$ is completely contained in $\cocore_+$.
In either case, $v$ is homotopic to a sphere lying in $W_0$:
indeed, if $v$ does not intersect the co-core, we can move it out of the
handle by pushing it radially from $\Sigma^+ \times (\DD^2 \setminus
\{0\})$ into the boundary $\Sigma^+ \times \SS^1 \subset W_0$, and if
$v\subset \cocore_+ = \Sigma_+\times \{0\}$, then we can simply shift it to
$\Sigma_+\times \{1\} \subset W_0$.
Using the fact that $u_0$ and $u_+$ are both somewhere injective, together
with the semipositivity and genericity assumptions, we deduce that
every connected component
of~$u_\infty$ has nonnegative index, thus $\overline{\mM} \setminus
\mM$ has codimension at least two in~$\overline\mM$.
The rest of the proof is the same as for
Theorem~\ref{thm:plastikstufe_and_weak_fillability}.

\subsection{An alternative argument using holomorphic spheres}
\label{sec:hol_Spheres_in_handle}

In this section we will prove the following variation on
Theorem~\ref{thm:torsion}, which does not involve the word
``semipositive'' at all.

\begin{theorem}
\label{thm:torsion2}
Suppose $(V,\xi)$ is a closed contact manifold containing a
subdomain~$N$ with nonempty boundary, which is obtained by gluing and
blowing down Giroux domains.  If either $N$ has at least one blown down
boundary component or it includes at least two Giroux domains glued
together, then $(V,\xi)$ does not admit any weak filling $(W,\omega)$
with $[\omega_V] \in \oO(N)$.
In particular $(V,\xi)$ is not strongly fillable.
\end{theorem}
\begin{proof}
We consider first the case where $N$ has no blown down boundary components
but consists of at least two Giroux domains glued together: without loss
of generality, we may then assume~$N$ has the form
\begin{equation*}
  N = (\Sigma^+ \times \SS^1) \cup_{Y \times \SS^1} 
  (\Sigma^- \times \SS^1) \;,
\end{equation*}
where $\Sigma^\pm$ are ideal Liouville domains with boundary $\p
\Sigma^\pm = \p_\mathrm{glue} \Sigma^\pm \sqcup \p_\mathrm{free}
\Sigma^\pm$, $Y := \p_\mathrm{glue} \Sigma^+ = \p_\mathrm{glue}
\Sigma^-$ carries the induced contact form~$\alpha$ and
$\p_\mathrm{free} \Sigma^+$ is not empty.  Arguing by contradiction,
assume also that $(V,\xi)$ has a weak filling
$(W_0,\omega)$ with $[\restricted{\omega}{TV}] \in
\oO(N)$.  This establishes the cohomological condition needed by
Theorem~\ref{thm:cobordism} on both $\Sigma^+ \times \SS^1$ and
$\Sigma^- \times \SS^1$, so applying the theorem, we can enlarge
$(W_0,\omega)$ by attaching $\Sigma^+ \times \DD^2$ and $\Sigma^-
\times \DD^2$ in succession, producing a compact symplectic manifold
$(W_1,\omega)$ whose boundary is a disjoint union of pieces
\begin{equation*}
  \p W_1 = \p_\mathrm{sph} W_1 \sqcup \p_\mathrm{ct} W_1 \;,
\end{equation*}
where $\p_\mathrm{ct} W_1 \neq \emptyset$ supports a contact structure
that is weakly dominated and $\p_\mathrm{sph} W_1 \cong Y \times \SS^2$
with symplectic fibers $\{*\} \times \SS^2$ (here \textsl{ct} stands for 
``contact'' and \textsl{sph} for ``sphere'').  Moreover, a neighborhood
of $\p_\mathrm{sph} W_1$ in $(W_1,\omega)$ can be identified
symplectically with the collar
\begin{equation*}
  \bigl((-\delta,0]\times Y \times \SS^2, \omega_Y \oplus \omega_{\SS^2}\bigr) \;,
\end{equation*}
where $\omega_{\SS^2}$ is an area form on~$\SS^2$ and $\omega_Y$ is a
symplectic form on $(-\delta,0]\times Y$ 
with weakly filled boundary $(Y,\ker\alpha)$.

We can choose an almost complex structure $J$ on $(W_1,\omega)$ with the
following properties:
\begin{enumerate}
\renewcommand{\labelenumi}{(\roman{enumi})}
\item $J$ is tamed by~$\omega$ and makes $\p_\mathrm{ct} W_1$ strictly $J$-convex,
\item $J = J_Y \oplus j$ on the collar 
$(-\delta,0] \times Y \times \SS^2$,
where $j$ is an $\omega_{\SS^2}$-compatible almost complex structure
on~$\SS^2$ and $J_Y$ is an almost complex structure on $(-\delta,0] \times Y$
which is tamed by $\omega_Y$ and makes
$\{0\}\times Y$ strictly $J_Y$-convex.
\end{enumerate}
By choosing $J$ generic outside the collar neighborhood
$(-\delta,0] \times Y \times \SS^2$, we may assume
all somewhere injective $J$-holomorphic curves that
aren't contained entirely in that region are Fredholm regular.

For each $(t,y) \in (-\delta,0] \times Y$, there is now an embedded
pseudoholomorphic sphere
\begin{equation*}
  u_{(t,y)} \colon (\SS^2,j) \to \bigl( (-\delta,0] \times Y \times \SS^2,J \bigr),\,
  z \mapsto  (t,y,z) \;.
\end{equation*}
These curves are all Fredholm regular: indeed, a neighborhood of
$u_{(t,y)}$ in the moduli space of unparametrized $J$-holomorphic spheres
in $W_1$ can be identified with a neighborhood of zero in the kernel of
the linearized Cauchy-Riemann operator on its normal bundle, but the latter is
a trivial bundle with the standard Cauchy-Riemann operator.  Hence the
operator splits into a direct sum of standard Cauchy-Riemann operators
on trivial line bundles over~$\SS^2$, all of which have index~$2$ and
are surjective by the Riemann-Roch theorem (cf.~\cite{McDuffSalamonJHolo}).
It follows also that the curves $u_{(t,y)}$ have index~$2n-2$.
Denote the co-cores of the two handles by
$$
\cocore_\pm := \Sigma^\pm \times \{0\} \subset \Sigma^\pm \times \DD^2
\subset W_1.
$$
The curves~$u_{(t,y)}$ have exactly one transverse intersection with
each of the two co-cores, so we have homological intersection numbers:
\begin{equation}
\label{eqn:intersection}
[u_{(t,y)}] \bullet [\cocore_+] = [u_{(t,y)}] \bullet [\cocore_-] = 1 \;.
\end{equation}
We claim that \emph{every} somewhere injective $J$-holomorphic sphere
in~$W_1$ which intersects $\{0\}\times Y \times \SS^2$ is equivalent
to $u_{(0,y)}$ for some $y \in Y$.
Indeed, if
$u \colon \SS^2 \to W_1$ is any such sphere, we define the open subset
$\uU = u^{-1}( (-\delta,0] \times Y \times \SS^2)$, and observe that
$\restricted{u}{\uU}$ can be identified with a pair of maps
$u_{\SS^2} \colon \uU \to \SS^2$ and $u_Y \colon \uU \to  (-\delta,0] \times Y$
which are $j$-holomorphic and $J_Y$-holomorphic respectively.  But then
$u_Y$ touches the boundary of $(-\delta,0]\times Y$ tangentially,
which is impossible due to pseudoconvexity unless $u_Y$ is constant,
so we conclude that $\uU = \SS^2$ and $u_{\SS^2} \colon \SS^2 \to \SS^2$ is
a degree~$1$ holomorphic map, hence the identity up to parameterization.

Let $\mM$ denote the connected component of the moduli space of 
unparametrized $J$-holomorphic spheres containing the curves $u_{(t,y)}$,
and define $\mM_1$ to be the same space of curves but with one marked point,
along with the natural compactifications $\overline{\mM}$ and
$\overline{\mM}_1$, consisting of nodal $J$-holomorphic spheres.  
Note that curves in $\overline{\mM}$ can never touch $\p_\mathrm{ct} W_1$ due to
$J$-convexity.
By~\eqref{eqn:intersection}, every curve
in~$\mM$ intersects each of $\cocore_+$ and $\cocore_-$ algebraically once,
thus all curves in~$\mM$ are somewhere injective, and
the only nodal curves in $\overline{\mM}$ intersecting 
$\{0\}\times Y \times \SS^2$ are $u_{(0,y)}$ for $y \in Y$.
Now our genericity assumptions for~$J$, along with the Fredholm regularity of
the special curves $u_{(t,y)}$, imply that $\mM$ is a smooth 
$(2n-2)$-dimensional manifold with boundary, where the boundary is a single
connected component consisting of the curves $u_{(0,y)}$.
After perhaps shrinking $\delta > 0$, we claim in fact that every
curve in~$\mM$ intersecting $(-\delta,0] \times Y \times \SS^2$
is one of the special curves $u_{(t,y)}$: were this not the case, we would
find sequences of negative numbers $t_k \to 0$ and holomorphic spheres
$u_k \in \mM$ which are not equivalent to any $u_{(t,y)}$ but
intersect $\{t_k\} \times Y\times \SS^2$, and a subsequence then converges
to a nodal curve in $\overline{\mM}$ intersecting 
$\{0\} \times Y \times \SS^2$, but the latter must be of the form
$u_{(0,y)}$.  We then obtain a contradiction from the implicit function theorem,
as the $(2n-2)$-dimensional moduli space of curves close to $u_{(0,y)}$
consists only of curves of the form $u_{(t',y')}$ for
some $(t',y') \in (-\delta,0] \times Y $.

Although $\mM$ and $\mM_1$ are smooth as explained above, their 
compactifications $\overline{\mM}$ and $\overline{\mM}_1$ need not be, 
due to the presence of nodal curves with
multiply covered components for which transversality fails.  This is
exactly the kind of problem that the polyfold machinery of 
Hofer-Wysocki-Zehnder \cite{HoferWZ_GW} is designed to solve: we perturb
the nonlinear Cauchy-Riemann equation via a generic multisection of the
appropriate polyfold bundle so that the
compact space $\overline{\mM}'$ of solutions to this perturbed 
equation is, in general, an oriented, weighted branched orbifold with 
boundary and corners.  In the case at hand, the perturbation can be chosen to have
support in a neighborhood of the nonsmooth part of $\overline{\mM}$, thus
we may assume in particular that elements of~$\overline{\mM}'$ approaching
a neighborhood of the boundary are still actually $J$-holomorphic curves,
so the uniqueness statements above continue to hold.

To conclude the proof, choose a smoothly embedded path $\ell \subset W_1$ with
one boundary point in $\p_\mathrm{ct} W_1$ and the other in 
$\p_\mathrm{sph} W_1$, meeting both transversely, and define the compact
space
\begin{equation*}
  \overline{\mM}_\ell' = \ev^{-1}(\ell)
\end{equation*}
where $\ev \colon \overline{\mM}_1' \to W_1$ denotes the natural evaluation
map.  For generic choices, $\overline{\mM}_\ell'$ is then a smooth, compact,
oriented, weighted branched $1$-dimensional manifold with boundary,
the latter being $\ev^{-1}(\p \ell)$.  By
pseudoconvexity however, no curve in $\overline{\mM}'$ intersects
$\p_\mathrm{ct} W_1$, hence $\p \overline{\mM}_\ell' = 
\ev^{-1}(\p_\mathrm{sph} W_1)$, but this consists
of only one curve, namely the unique $u_{(0,y)}$ with $y \in \p\ell$.
Since there is no compact, oriented, weighted branched $1$-manifold with 
connected boundary (see e.g.~\cite[Lemma~5.11]{Salamon_lecture_notes}), 
this gives the desired contradiction.

The proof is essentially the same but slightly simpler
if $N \subset (V,\xi)$ has any blown down boundary components.
If suffices then to consider the case where $N$ is a single blown down
Giroux domain $(\Sigma \times \SS^1) \blowdown M_\mathrm{bd}$.
Attaching $\Sigma \times \DD^2$ via Theorem~\ref{thm:cobordism}, we again
obtain a symplectic manifold $(W_1,\omega)$ with
$\p W_1 = \p_\mathrm{ct} W_1 \sqcup \p_\mathrm{sph} W_1$, where 
$\p_\mathrm{ct} W_1 \ne \emptyset$ 
is weakly filled
and $\p_\mathrm{sph} W_1 \cong M_\mathrm{bd} \times \SS^2$ is foliated by
symplectic spheres that give rise to $J$-holomorphic spheres intersecting the
co-core $\Sigma \times \{0\}$ exactly once.  The rest of the argument is the
same.
\end{proof}

\begin{remark}
\label{remark:nopolyfolds}
If the original filling is assumed semipositive, then the above proof can
also be modified to take advantage of the symplectic co-core in the
same way as \S\ref{sec:BLOB_in_handle} and thus avoid the need for polyfolds.
\end{remark}

\section{Construction of Liouville domains with disconnected boundary}
\label{sec:Liouville}

\subsection{Contact products and Liouville pairs}

In this section we shall construct Liouville pairs on closed
manifolds of every odd dimension; more precisely, we shall prove
Theorem~\ref{thm:LiouvilleExists} from the introduction and thus lay the
groundwork for our Giroux torsion construction in~\S\ref{sec:torsion}.  
Recall that the goal is to find positive/negative pairs of contact forms
$(\alpha_+,\alpha_-)$ on oriented
odd-dimensional manifolds~$M$ with the property that, if $s \in \RR$ denotes
the coordinate on the first factor of $\RR \times M$,
\begin{equation*}
  \beta := e^{-s} \alpha_- + e^{s} \alpha_+
\end{equation*}
defines a positively oriented Liouville form on $\RR\times M$.  

The first example of a Liouville pair is $\pm d\theta$ on $\SS^1$. 
One can construct higher dimensional examples using contact 
products.  The contact product of 
$(M_1, \xi_1)$ and $(M_2, \xi_2)$ is defined as the product
of their symplectizations $S\xi_1 \times S\xi_2$ divided by the diagonal
$\RR$-action (cf.~\cite{Giroux_Bourbaki}).
This describes a contact manifold but, since the Liouville pair
condition is really about contact forms and not only contact structures, we want
a more specific construction.
Suppose we have contact forms $\alpha_1$ and $\alpha_2$. Those give
identifications between $S\xi_i$ and $\RR \times M_i$ with fiber coordinates
$t_i$ on $\RR$. On the product, one has the Liouville form 
$\lambda = e^{t_1}\alpha_1 + e^{t_2}\alpha_2$ and its dual vector field 
$X = \partial_{t_1} + \partial_{t_2}$. 
We shall 
say that a manifold $V$ with a contact form $\lambda$ is a \defin{linear
model} for the contact product of $(M_1, \alpha_1)$ and $(M_2, \alpha_2)$
if it is realized as a hypersurface in $S\xi_1 \times S\xi_2$ 
transverse to $X$ and defined by a linear equation on $t_1$ and $t_2$.
Concretely, this means 
$V = M_1 \times \RR \times M_2$ is embedded into the product
$(\RR \times M_1) \times (\RR \times M_2)$ 
by $\varphi(m_1, t, m_2) = (\mu t, m_1, \nu t, m_2)$ for some constants $\mu$
and $\nu$.  This gives a hypersurface positively transverse to $X$ provided
$\nu > \mu$.  The contact form induced by $\lambda$ on $V$ is then
$e^{\mu t}\alpha_1 + e^{\nu t} \alpha_2$.

\begin{proposition}
\label{prop:Liouville_product}
If $M_1$ is $\RR$ or $\SS^1$ endowed with the Liouville pair 
$\alpha_\pm = \pm d\theta$ and $(M_2, \alpha_2)$ is any manifold with a contact
form, then any linear model for the contact product inherits a Liouville pair
$\pm e^{\mu t} d\theta + e^{\nu t} \alpha$.
\end{proposition}

\begin{proof}
We set $a = e^s + e^{-s}$, $b = e^s - e^{-s}$ and $e_\rho = e^{\rho t}$ for any
real number $\rho$ so that our candidate Liouville form on 
$\RR \times M_1 \times \RR \times M_2$ is $B = ae_\nu \alpha + b e_\mu d\theta$.
One computes
\begin{equation*}
dB = e_\nu(a\nu \,dt + b\,ds) \wedge \alpha + a e_\nu\,d\alpha
+e_\mu(b\mu\,dt + a\,ds)\wedge d\theta
\end{equation*}
and then, denoting by $2q + 1$ the dimension of $M_2$,
\begin{equation*}
dB^{q + 2} = f(\nu a^2 - \mu b^2) 
ds \wedge d\theta \wedge dt \wedge \alpha \wedge d\alpha^q ,
\end{equation*}
where $f = (q + 1)(q + 2) a^q e_{\mu + (q + 1)\nu}$ 
and $\nu a^2 - \mu b^2$ is positive because $\nu > \mu$ and $a^2 - b^2 = 4$.
\end{proof}

\begin{remark}
One can ask whether the above proposition extends to other Liouville pairs. It
seems that not all linear models will be suitable for this. What we can prove,
but will not use in this paper, is that if $\alpha_\pm$ is a Liouville pair on
some manifold $M_1$ (of any dimension) then $\alpha_\pm + e^t \alpha$ is a
Liouville pair on $M_1 \times \RR \times M_2$. 
\end{remark}

Of course, the disadvantage of the contact product construction is that 
the resulting manifold is never compact, and there seems to be
no general way of finding compact quotients
of contact products.  We shall therefore specialize further by seeking
examples among Lie groups which can be seen as
symplectizations of some subgroups that have co-compact lattices.
(The idea to consider left-invariant contact forms on Lie groups
is borrowed from Geiges \cite{Geiges_disconnected}.)

Before that, let us describe a corollary of the following 
algebraic constructions that has the advantage of seeming somewhat
more concrete.  We will not use this concrete description explicitly,
so we leave it as an exercise to check that it can indeed be related to our
abstract treatement below.
Taking any integer $n \ge 0$, if we assign to $\RR^{n} \times \RR^{n+1}$ the
coordinates
$(t_1,\dotsc,t_n,\theta_0,\dotsc,\theta_n)$ 
then one can show using the above proposition or by an explicit calculation that
for a suitable choice of orientation on $\RR^n \times \RR^{n+1}$,
\begin{equation}
\label{eqn:LiouvillePair}
\alpha_\pm := \pm e^{t_1 + \dotsm + t_n} \, d\theta_0 +
e^{-t_1}\, d\theta_1 + \dotsm + e^{-t_n} \, d\theta_n \;.
\end{equation}
is a Liouville pair. We would like to prove the existence of 
compact quotients of 
$\RR^n \times \RR^{n+1}$ to which $\alpha_+$
and~$\alpha_-$ both descend.  Observe that both are invariant under the group
action of $\RR^n$ on $\RR^n \times \RR^{n+1}$ defined by
\begin{multline}
\label{eqn:groupAction}
(\tau_1,\dotsc,\tau_n) \cdot (t_1,\dotsc,t_n,\theta_0,\theta_1,\dotsc,
\theta_n) := \\
(t_1 + \tau_1,\dotsm,t_n + \tau_n, e^{-(\tau_1 + \dotsm + \tau_n)} \theta_0,
e^{\tau_1} \theta_1,\dotsc,e^{\tau_n} \theta_n) \;.
\end{multline}

What we will prove in the next few sections implies the following:

\begin{lemma}
\label{lemma:concrete}
There exist lattices $\Lambda \subset \RR^n$
and $\Lambda' \subset \RR^{n+1}$ such that the group action of $\Lambda$
on $\RR^n \times \RR^{n+1}$ defined by \eqref{eqn:groupAction} preserves
$\RR^n \times \Lambda'$.
\end{lemma}

It follows that the action of $\Lambda$ on $\RR^n \times \RR^{n+1}$
descends to a smooth group action on $\RR^n \times (\RR^{n+1} / \Lambda')$,
so dividing by this action we obtain a bundle with fiber
$\RR^{n+1} / \Lambda' \cong \TT^{n+1}$ and base $\RR^n / \Lambda \cong \TT^n$,
to which the Liouville pair $(\alpha_+,\alpha_-)$ descends.  In this way
one obtains the following result, which suffices to prove the existence
of Liouville domains with disconnected boundary in all dimensions:

\begin{proposition}
\label{prop:torusBundle}
For every integer $n \ge 0$,
the Liouville pair defined by \eqref{eqn:LiouvillePair} on
$\RR^n \times \RR^{n+1}$ descends to a compact quotient which is
a $\TT^{n+1}$-bundle over $\TT^n$.
\end{proposition}

Lemma~\ref{lemma:concrete} is trivial when $n=0$, and elementary when $n=1$:
for the latter case, one can choose $\Lambda \subset \RR$ to be 
generated by any real number $\tau \ne 0$ such that $e^\tau$ is an 
eigenvalue of some matrix $A \in \SL(2,\ZZ)$.  Then~$A$ may be viewed as
the matrix of the linear transformation $\RR^2 \to \RR^2 :
(\theta_0,\theta_1) \mapsto (e^{-\tau} \theta_0, e^\tau \theta_1)$ in some
other basis where it has integer coefficients. This transformation 
therefore preserves the lattice generated by that basis.
This produces a Liouville pair on every $\TT^2$-bundle over $\SS^1$ with 
hyperbolic monodromy---these examples have appeared in the previous
work of Geiges \cite{Geiges_disconnected4} and Mitsumatsu
\cite{Mitsumatsu_Anosov}.
A hint of the general arithmetic strategy we will use below appears in this
discussion, as the condition that $e^\tau$ should be an eigenvalue of some
matrix in $\SL(2,\ZZ)$ implies that $e^\tau$ belongs to a quadratic extension of
the field~$\QQ$.

\subsection{Some Lie groups as symplectizations}

Denote by $\Aff^+(\RR)$ the group of orientation preserving affine
transformations of the real line.  Similarly, $\widetilde\Aff(\CC)$ will
denote the universal cover of the group $\Aff(\CC)$ of affine
transformations of the complex plane, which can be identified with 
the semi-direct product $\CC \ltimes \CC$ by associating to any 
$(a,b) \in \CC \times \CC$ the
transformation $z \mapsto e^a z + b$.  Observe that the same trick
identifies $\Aff^+(\RR)$ with $\RR \ltimes \RR$.

Let $\gR$ denote the Lie algebra of the affine group $\Aff^+(\RR)$.
The identification $\Aff^+(\RR) = \RR \ltimes \RR$ defines global coordinates
$(t,\theta)$ on $\Aff^+(\RR)$ and hence a basis
$(\T,\X)$ of left-invariant vector fields that match $(\p_t,\p_\theta)$ at the 
identity; they satisfy $[\T, \X] = \X$.  
Writing the dual Lie algebra as $\gR^*$, its dual 
basis is the pair of left-invariant $1$-forms
\begin{equation*}
  \T^* = dt, \qquad \X^* = e^{-t}\, d\theta.
\end{equation*}
Thus we can associate to $\Aff^+(\RR)$ the left-invariant Liouville forms
$\Theta^*$ or $-\Theta^*$ and view it as the symplectization of 
$(\RR, d\theta)$ or $(\RR, -d\theta)$ respectively 
with fiber coordinate $-t = -\ln \circ \det$.  Note that $(\RR,\pm d\theta)$
has a canonical contact type embedding into $(\Aff^+(\RR),\pm \Theta^*)$,
namely as the unimodular subgroup $\{ \det = 1 \}$.

We denote by $\gC$ the Lie algebra of 
$\widetilde{\Aff}(\CC) = \CC \ltimes \CC$.  Using coordinates
$(u + iv,x + iy)$ on $\CC \ltimes \CC$, the basis
$(\R,\Ss,\U,\V)$ of $\gC$ defined to match $(\p_u,\p_v,\p_x,\p_y)$ at the
identity satisfies the relations
\begin{equation*}
\begin{array}{ll}
\left[\R, \U\right] = \U, & [\Ss, \U] = \V, \\
\left[\R, \V\right] = \V, & [\Ss, \V] = -\U,
\end{array}
\end{equation*}
with all other brackets vanishing. These relations give the following exterior
derivatives for the dual basis of left-invariant $1$-forms:
\begin{equation}\label{eqn:exteriorC}
\begin{aligned}
d\R^* = 0, \qquad & d\U^* = \U^* \wedge \R^* + \Ss^* \wedge \V^*, \\
d\Ss^* = 0,\qquad & d\V^* = \U^* \wedge \Ss^* + \V^* \wedge \R^*.
\end{aligned}
\end{equation}
Although this will not be used, we note for concreteness that in the
coordinates defined above,
\begin{equation*}
\begin{array}{ll}
\R^* = du, \quad & \U^* = e^{-u} (\cos v)\, dx + e^{-u} (\sin v)\, dy, \\
\Ss^* = dv, \quad & \V^* = - e^{-u} (\sin v)\, dx + e^{-u} (\cos v)\, dy.
\end{array}
\end{equation*}
Now we can define a left-invariant Liouville form as, for instance, 
\[
\beta = \U^* \quad \text{so that} \quad 
d\beta^2 = -2\;\R^* \wedge \Ss^* \wedge \U^*\wedge \V^*.
\]
The corresponding Liouville vector field is $\beta^\# = -\R$, which is
transverse to the unimodular subgroup $\{\abs{\det}^2 = 1\} = i\RR \ltimes \CC$,
whose Lie algebra is the kernel of~$\R^*$. So we have on $\widetilde\Aff(\CC)$
a left-invariant symplectization structure with fiber coordinate
$- \ln \circ |\det|$, where $\det$ denotes the determinant of the projection in
$\Aff(\CC)$.

We now combine any number of copies of the two preceding Lie groups as
$\GrsOrig := \Aff^+(\RR)^r \times \widetilde\Aff(\CC)^s$, and then consider
the subgroup
\begin{equation*}
\GrsunOrig = \Big\{ 
(\varphi_1,\dotsc,\varphi_r,\psi_1, \dotsc, \psi_s) \in \GrsOrig \;\Big|\; 
\prod_i \det\varphi_i \prod_j \abs{\det\psi_j}^2 = 1\Big\} \;,
\end{equation*}
where $\det\psi_j$ should be understood again as the determinant of the
projection of $\psi_j$ to~$\Aff(\CC)$. 
The discussion above shows that this group can be seen as a linear model for a
contact product. When $r$ is positive, we can single out one of the
$\Aff^+(\RR)$ factors 
and apply Proposition \ref{prop:Liouville_product} to obtain:

\begin{corollary}
	\label{cor:Liouville_Lie_group}
For any positive $r$, the Lie group $\GrsunOrig$ admits a left-invariant
Liouville pair. \qed
\end{corollary}

The goal of the next two sections is to prove the existence of 
co-compact lattices in this group in
order to find closed manifolds with Liouville pairs.
As preparation, it will be useful observe that both 
$\GrsOrig$ and $\GrsunOrig$ can be
viewed as semi-direct products: setting
\begin{equation}\label{eqn:hLie}
\begin{split}
\h^{r,s} &:= \RR^r \times \CC^s,\\
\h^{r,s}_1 &:= \Big\{ (t_1,\dotsc,t_r,w_1,\dotsc,w_s) \in \h^{r,s} \ \Big|\ 
\sum_i t_i + 2\sum_j \RealPart w_j  = 0 \Big\}
\end{split}
\end{equation}
and defining the action of each on $\RR^r \times \CC^s$ by
\begin{equation*}
  (t_1,\dotsc,t_r,w_1,\dotsc,w_s) \cdot (\theta_1,\dotsc,\theta_r,z_1,\dotsc,z_s)
  := (e^{t_1} \theta_1,\dotsc, e^{t_r} \theta_r,
  e^{w_1} z_1, \dotsc, e^{w_s} z_s) \;,
\end{equation*}
we have natural isomorphisms
$\GrsOrig = \h^{r,s} \ltimes (\RR^r \times \CC^s)$ and
$\GrsunOrig = \h^{r,s}_1 \ltimes (\RR^r \times \CC^s)$.

\begin{remark}
For most of the following discussion, the reader is free to assume $s=0$,
since this suffices to prove Lemma~\ref{lemma:concrete} and thus the
existence of closed manifolds admitting Liouville pairs in all dimensions.
The case $s > 0$ is only really needed to 
relate our results to those of Geiges in
\S\ref{subsec:Geiges}.
\end{remark}

\subsection{Some number theory}
\label{subsec:number}

In this section we will need some standard notions and results
from algebraic number theory, e.g.~Dirichlet's Unit Theorem; a good
reference for this material is \cite{Marcus_numberFields}.

Throughout this section and the next, $\k$ will denote a number field, 
i.e.~a finite degree extension of $\QQ$, and $n$ will 
denote its degree $[\k : \QQ]$. 
Such a field is always isomorphic to $\QQ[X]/(f)$ for some irreducible
polynomial $f \in \QQ[X]$ of degree~$n$ (with simple roots).
We will denote by $r$ the number of real roots and by $s$ the number of
complex conjugate pairs of nonreal roots, thus $n = r + 2s$.
Each root $\alpha$ gives an embedding of $\k$ into $\CC$, sending 
(the equivalence class of) $X$ to $\alpha$. 
These embeddings will be denoted by $\rho_1, \dotsc, \rho_r$ and 
$\sigma_1, \dotsc, \sigma_s, \bar\sigma_1, \dotsc, \bar \sigma_s$.
This method actually gives all embeddings of $\k$ into $\CC$, and
we can collect them to define an injective map
\begin{equation*}
  j \colon \k \to \RR^r \times \CC^s : 
  x  \mapsto (\rho_1(x), \dotsc,  \rho_r(x), \sigma_1(x), \dotsc, \sigma_s(x)) \;.
\end{equation*}
The norm of an element of $\k$ is defined as
$N(x) = \prod_i \rho_i(x) \prod_j \abs{\sigma_j(x)}^2$, and the fact that
$f$ is irreducible implies that $N(x)$ vanishes only when $x=0$.
The ring of integers $\Ok$ of $\k$ is by definition 
the set of all elements in $\k$  which
are roots of monic polynomials with coefficients in~$\ZZ$.  These all have 
integer-valued norms, and
an important observation is that the map $j$ defined above sends $\Ok$ to a
lattice in $\RR^r \times \CC^s$.

Invertible elements in the ring $\Ok$ are called \emph{units} of $\k$, and 
they form a (multiplicative) group denoted by $\Okx$.
They all have norm $\pm 1$ since $N(xy) = N(x)N(y)$.
We denote by $\Okxp$ the subgroup of \emph{positive} units: 
$\Okxp = \{ x \in \Okx \ |\  \rho_i(x) > 0 \text{ for all } i\}$.
Among units are the roots of unity, whose (finite) set is denoted by $\Uk$.
We also set $\Ukp = \Uk \cap \Okxp$.  Dirichlet's Unit Theorem implies
that $\Okx$ is a finitely generated abelian group with torsion $\Uk$ and 
rank $r + s - 1$.  Since $x^2 \in \Okxp$ whenever $x \in \Okx$, it follows
that $\Okxp$ is similarly the product of the finite cyclic group $\Ukp$
with a free abelian group of rank $r + s - 1$.
The map~$j$ restricts to an injective group homomorphism of $\Okxp$ into
the multiplicative group $H^{r,s} := (\RRpx)^r \times (\CCx)^s$, and since
$N(\Okxp) = \{1\}$, its image lies in the subgroup
\begin{equation*}
H_1^{r,s} = 
\Big\{ (\rho_1, \dots, \rho_r, \sigma_1, \dots, \sigma_s) \in H^{r,s} \ \Big|\ 
\prod_i \rho_i \prod_j \abs{\sigma_j}^2 = 1 \Big\} \;.
\end{equation*}
The precise formulation of Dirichlet's theorem is that $j(\Okxp)$ is a lattice
in $H_1^{r,s}$.

\subsubsection*{Examples}
We now discuss three examples of increasing complexity to see all the objects
discussed above appearing.
In the next subsection we will see the contact manifolds associated to these
fields and the Liouville pair construction where applicable.

The very first example of a number field is $\QQ$ itself. In this case $n = 1$,
$f = X - 1$, $r = 1$, $s = 0$ and $j$ is the inclusion of $\QQ$ in $\RR$.
The ring of integers is $\Ok = \ZZ$, with $\Okx = \{ \pm 1 \}$ and 
$\Okxp = \{ 1 \}$.

As a less trivial number field, we consider $\k = \QQ[i]$.
We have $n = 2$ and can set $f = X^2 + 1$, so the roots are $\pm i$,
hence $r = 0$ and $s = 1$.
Choosing $i$ as a member of the complex conjugate pair $\pm i$, we have
$j \colon \k \to \CC : a + bX \mapsto a + ib$.  The norm of $a + bX$ is
$a^2 + b^2$.  The integer ring $\Ok$ is $\ZZ + \ZZ X$, and its
image under $j$ is the lattice $\ZZ + i \ZZ$ in $\CC$.
The group of units is $\Okx = \{\pm 1, \pm X\} = \Uk$. All units are
automatically positive since there is no real embedding.
We have $H^{0,1} = \CC^*$ and $H_1^{0,1} = \SS^1$, in which 
$j(\Okxp)$ is indeed a lattice.

As a last example, we consider $\k = \QQ[\sqrt 2]$. 
Here $n = 2$ and $f = X^2 -2$ with roots $\pm\sqrt 2$, so $r = 2$ and $s = 0$.
The $j$ map is defined by $a + bX \mapsto (a + b\sqrt 2, a -b \sqrt 2)$.
The norm of $a + b X$ is $a^2 - 2b^2$.
The integer ring $\Ok$ is $\ZZ + \ZZ X$, and its
image under~$j$ is the lattice 
\begin{equation*}
\{ (a + b\sqrt 2, a - b\sqrt 2) \ |\  a,b \in \ZZ \} 
= \ZZ(1, 1) + \ZZ(\sqrt 2, -\sqrt 2) \subset \RR^2 \;.
\end{equation*}
The group of units is $\Okx = \{\pm (1 + X)^k\ |\ k \in \ZZ\}$,
and $\Uk = \{ \pm 1\}$. 
Restricting to positive elements, we have
$\Okxp = \{(3 + 2X)^k \ |\ k \in \ZZ\}$
and $\Ukp = \{ 1\}$. 
The image of $\Okxp$ in $H^{2,0} = (\RR^*_+)^2$ is 
$j(\Okxp) = \left\{((3 + 2\sqrt 2)^k, (3 - 2\sqrt 2)^k) \ |\ k \in \ZZ\right\}$,
which is indeed a lattice in $H_1^{2,0} = \{(y,1/y) \in (\RRpx)^2 \ |\  y > 0\}$.

\subsection{A manifold associated to a number field}
\label{subsec:cocompact}

The next result provides the desired co-compact lattices in the
Lie groups~$\GrsunOrig$.

\begin{proposition}\label{prop:moreLattices}
Suppose $\k = \QQ[X] / (f)$ is a number field of degree $n = r + 2s \ge 1$,
where $f \in \QQ[X]$ is an irreducible polynomial with $r$ real
and $2s$ complex roots (all simple).  Then one can associate to $\k$
a lattice $\Gk \subset \GrsunOrig$ such that the quotient
\begin{equation*}
  M_\k := \GrsunOrig / \Gk
\end{equation*}
is a $\TT^n$-bundle over $\TT^{n-1}$.
\end{proposition}

To prove this, we continue with the same notation as in the previous
section and observe that the Lie algebras of
$H^{r,s}$ and $H^{r,s}_1$ are precisely $\h^{r,s}$ and $\h^{r,s}_1$
respectively, defined in \eqref{eqn:hLie} above.
Since $H^{r,s}$ is abelian, the exponential map 
\begin{equation*}
  \exp \colon \h^{r,s} \to H^{r,s} : (t_1,\dotsc,t_r,w_1,\dotsc,w_s) \mapsto
  (e^{t_1},\dotsc,e^{t_r},e^{w_1},\dotsc,e^{w_s})
\end{equation*}
is a surjective group homomorphism, as is its restriction
to $\h^{r,s}_1 \to H^{r,s}_1$, and
its kernel is the free abelian group 
$\{0\} \times 2\pi i \ZZ^s \subset \h^{r,s}_1$.
Thus the preimage of
$j(\Okxp)$ in $\h_1^{r,s}$ is a rank $r + s - 1 + s = n - 1$ lattice,
which we denote by~$\Gammak$.

The group $H^{r,s}$ acts on $\RR^r \times \CC^s$ via coordinate-wise
multiplication, so pulling back this action via the exponential map
defines an action of $\Gammak$ on $\RR^r \times \CC^s$, which preserves
the lattice $j(\Ok)$ since multiplication by elements of $\Okx$
preserves~$\Ok$.  The inclusions
$\Gammak \hookrightarrow \h_1^{r,s}$ and
$j(\Ok) \hookrightarrow \RR^r \times \CC^s$ can therefore be combined to an
inclusion of the semi-direct product
\begin{equation*}
  \Gk := \Gammak \ltimes j(\Ok) \hookrightarrow \h^{r,s}_1 \ltimes
  (\RR^r \times \CC^s) = \GrsunOrig \;,
\end{equation*}
forming a lattice.  Proposition~\ref{prop:moreLattices} now follows from
the observation that the projection $\GrsunOrig = \h^{r,s}_1 \ltimes
(\RR^r \times \CC^s) \to \h^{r,s}_1$ descends to a well-defined projection
\begin{equation*}
  \GrsunOrig / \Gk \to \h^{r,s}_1 / \Gammak \;,
\end{equation*}
forming a bundle with fiber $(\RR^r \times \CC^s) / j(\Ok) \cong \TT^n$ 
and base $\h^{r,s}_1 / \Gammak \cong \TT^{n-1}$.

Note that the only choices we made in the construction of $M_\k$ were the
ordering of the embeddings of $\k$ into $\RR$ and $\CC$, and which complex
embedding we pick out of each complex conjugate pair.
The manifold $M_\k$ does not depend on these choices up to diffeomorphism.
Moreover, each choice of orientations for the factors of $\Aff^+(\RR)$ 
in $\GrsOrig$ determines a contact structure on~$M_\k$ uniquely up to isotopy.
Indeed, aside from the orientation of $\Aff^+(\RR)$, the only other choices 
involved were the Liouville forms
on the relevant Lie groups, but one can check that all possible left-invariant
Liouville forms defining the same orientation
are isotopic---they form a connected open subset of the dual Lie
algebra.  If we fix a single orientation of $\Aff^+(\RR)$ from the beginning,
we then obtain the
canonical contact structure promised in Theorem~\ref{thm:LiouvilleExists} from
the introduction.

We remark that $G_\k$ is not the only possible lattice in $\GrsunOrig$. 
In the totally real case 
($s = 0$) in particular, one can replace $\Ok$ by any additive subgroup $M$ of $\k$
which is a free abelian group of rank $n$, and $\Okxp$ by any of its finite
index subgroups preserving $M$. The contact manifolds obtained in this way are cusp
cross sections of Hilbert modular varieties, see \cite[Chapter~1]{vdGeer}. In
particular, they are Stein fillable and can be embedded as separating strictly
pseudoconvex hypersurfaces in holomorphic manifolds.

For later use in \S\ref{sec:torsion}, we note the following observation.
\begin{lemma}
\label{lemma:centerGk} If $\k$ is a totally real number field which is not
	$\QQ$, then $\pi_1(M_\k)$ has trivial center.
\end{lemma}
\begin{proof}
Since $M_\k = \GrsunOrig / G_\k$ and $\GrsunOrig$ is simply connected, the 
lemma is equivalent to the claim that the group $G_\k$ has trivial center.
In the totally real case (i.e.~$s=0$), we have $G_\k = \Okxp \ltimes \Ok$,
so as a set $\Gk = \Okxp \times \Ok$, and the composition law is
$(u, x)(u', x') = (uu', x + ux')$.
Suppose $(u, x)$ is central in $G_\k$. This implies that for any $(u', x')$,
$x + ux' = x' + u'x$.

We can apply this to $u' = 1$ to deduce that for any integer $x'$, $ux' = x'$.
Since $\Ok$ is integral (recall it embeds in $\RR$), we get $u = 1$ or $x' = 0$.
Since $\Ok$ is not a trivial group (it has rank $\deg(\k)$),
this implies $u = 1$.

Similarily, we can apply the above formula with $x' = 0$ to deduce that
for any unit $u'$, $u'x = x$.  Hence $u' = 1$ or $x = 0$.
Since $\Okxp$ has rank $\deg(\k) - 1$ and we assume $\deg(\k) > 1$, 
we obtain $x = 0$, so $(1, 0)$ is the only central element.
\end{proof}

\subsubsection*{Examples}
Recall that our first example was $\k = \QQ$. The corresponding $\GrsunOrig$
is the group of affine transformations of $\RR$ with determinant one,
i.e.~it is~$\RR$.  The unit group $\Okxp$ is trivial, hence it acts trivially 
on $\RR$, implying $M_\QQ = \RR/j(\Ok) \cong \RR/\ZZ = \SS^1$.
The resulting Liouville pair on $\SS^1$ is $(d\theta,-d\theta)$.

Our second example was $\k = \QQ[i]$, and the corresponding $\GrsunOrig$ is
$\widetilde\Aff(\CC) = i\RR \ltimes \CC$, whose elements $(iv,w)$ correspond
to affine transformations $z \mapsto e^{iv} z + w$.  Note that since $r=0$
in this case, all left-invariant contact forms on $\GrsunOrig$ induce the
same orientation, so there can be no left-invariant Liouville pair, but we can
still extract a co-compact lattice.
We have $\h^{r,s}_1 = i\RR \subset \CC = \h^{r,s}$, and 
$\Gammak \subset i\RR$ is
spanned by $m := i\pi / 2$, where $m$ stands for ``monodromy''.
The action of $m$ on $\CC$ is $z \mapsto e^{i\pi / 2} z$, which
does indeed preserve the lattice $j(\Ok) = \ZZ + i\ZZ$.
We conclude that $M_\k$ is a $\TT^2$-bundle over $\SS^1$ whose monodromy
is a quarter turn.  Observe that $M_\k$ is a finite quotient of $\TT^3$,
which cannot admit any Liouville pair due to
\cite[Example~2.14]{WendlGirouxTorsion}.  

We proceed to the last example $\k = \QQ[\sqrt 2]$. 
The corresponding $\GrsunOrig$ is the unimodular subgroup within
$\Aff^+(\RR)^2$, which is the solvable group of Thurston's geometries. 
In the hyperplane $\h_1^{r,s} = \{ (-t,t) \} \subset \RR^2 = \h^{r,s}$, 
$\Gammak$ is spanned by $m := \big(\ln(3 + 2\sqrt 2), \ln(3 - 2\sqrt 2)\big)$.
The action of $m$ on $\RR^2$ is then
$(y_1, y_2) \mapsto ((3 + 2\sqrt 2)y_1, (3 - 2\sqrt 2)y_2)$, and
one can check by hand that it indeed preserves the lattice $j(\Ok)$.
Recall that a basis of this lattice is $\{(1, 1), (\sqrt 2, - \sqrt 2)\}$.
In this basis, the matrix of $m$ is 
$
A = \begin{pmatrix}
3 & 4 \\
2 & 3
\end{pmatrix},
$
so we see that $M_\k$ is a $\TT^2$-bundle over $\SS^1$ with monodromy $A$,
which is hyperbolic.
The Liouville pair we constructed yields two contact structures
which rotate in opposite directions between the stable and unstable
foliations of the Anosov flow defined by the monodromy 
(cf.~\cite{Mitsumatsu_Anosov}).

\subsection{Geiges pairs and Geiges groups}
\label{subsec:Geiges}

The idea of Geiges in \cite{Geiges_disconnected} was to consider a special
class of Liouville pairs (without the general definition) that satisfy a much
stronger algebraic condition, and to look for examples among left-invariant
contact forms on Lie groups.  The particular groups that Geiges considered
turn out to be a subfamily of the ones that we've studied above.

\begin{definition}
A \defin{Geiges pair} on an oriented 
manifold $M^{2n + 1}$ is a pair of contact forms
$\alpha_+$ and $\alpha_-$ on $M$ such that:
\begin{itemize}
\item 
$\alpha_+ \wedge d\alpha_+^n = -\alpha_- \wedge d\alpha_-^n > 0$, and
\item 
for all $0 \leq k \leq n -1$, 
$\alpha_\pm \wedge d\alpha_\pm^k \wedge	d\alpha_\mp^{n - k} = 0$.
\end{itemize}
\end{definition}
A version of \cite[Proposition~1]{Geiges_disconnected} is then the simple
observation that Geiges pairs are also Liouville pairs.  Note that the
Liouville pairs we constructed in the preceding section are Geiges pairs in
dimensions~$1$ and~$3$, but not in higher dimensions in general.

Geiges constructed in each odd dimension $2n - 1$ a Lie group $G_{2n - 1}$
admitting a left-invariant Geiges pair, and also found co-compact lattices in
these groups in dimensions~$3$ and $5$, thus giving examples of compact
Liouville domains with two boundary components in dimensions~$4$ and~$6$.  We
shall now show that our number theoretic construction can also be used to find
co-compact lattices for all the Geiges groups, implying the existence of Geiges
pairs on some closed manifold in every odd dimension.

\begin{proposition}
	\label{prop:isomorphism}
For any positive integer $n$, there is an isomorphism between $G_{2n - 1}$ and
$\GrsunOrig$ where $r = 1$ if $n$ is odd, $r = 2$ if $n$ is even, and 
$s = (n - r)/2$. 
\end{proposition}

The remainder of this section is devoted to the proof of this isomorphism.
The pairs
constructed by Geiges have a nice form in the basis of the Lie algebra he
considered, but our isomorphism will not preserve this basis in any nice way.
Of course, the point of our description of these groups was that it 
makes the construction of co-compact lattices much easier.

First we recall the definition of the Geiges group $G_{2n - 1}$.
For each positive integer $n$, let $A$ denote the $n \times n$ matrix
\begin{equation*}
A = 
\begin{pmatrix}
0      & -1      & 0      & \dots  & 0      \\
0      &  0      & 1      & \dots  & 0      \\
\vdots & \vdots  & \ddots & \ddots & \vdots \\
0      & 0       & \dots  &   0    & 1      \\
-1     & 0       & \dots  & \dots  & 0      
\end{pmatrix} \;.
\end{equation*}
We define the $(2n - 1)$-dimensional Geiges group $G_{2n - 1}$ as the
semi-direct product $\RR^{n - 1} \ltimes_A \RR^n$, where
$(y_1, \dots, y_{n - 1})$ acts as $\exp(y_1 A + \cdots + y_{n - 1} A^{n - 1})$
on $\RR^n$.
Reversing the sign of the first vector in the canonical basis of $\RR^n$, one
sees that $A$ is similar to the matrix of cyclic permutation of this basis.
In particular, all powers of $A$ appearing in the action have vanishing trace,
because powers between $1$ and $n - 1$ of this permutation have no
nontrivial fixed points, thus no diagonal term can appear.

The matrix $A$ is orthogonal and has characteristic polynomial
$X^n - 1$, so its eigenvalues are the~$n$th roots of unity.
We denote by $R_\alpha$ the $2 \times 2$ rotation matrix with
angle $\alpha$ and set $\theta = 2\pi/n$.
Then $A$ is similar to a block diagonal matrix 
$B = \mathrm{diag}(1, R_\theta, \dots, R_{s\theta})$ if $r = 1$, or 
$B = \mathrm{diag}(1, -1, R_\theta, \dots, R_{s\theta})$ if $r = 2$.
Choose an invertible matrix $P$ such that $A = P^{-1} B P$.
The map $(y, x) \mapsto (y, Px)$ is now an isomorphism 
$\RR^{n-1} \ltimes_A \RR^n \to \RR^{n-1} \ltimes_B \RR^n$, where we define
the latter group using~$B$ instead of~$A$ to construct the
$\RR^{n - 1}$-action analogously.

To simplify the notation, we now assume that $n$ is odd, so $r = 1$
and $s = (n-1) / 2$; the other case is completely analogous.
The matrix $B$ can be seen as acting on $\RR \times \CC^s$, with $R_\theta$
acting as multiplication by $e^{i\theta}$.  The matrix
$\rho(y) = y_1 B + \dotsm + y_{n-1} B^{n-1}$ for $y \in \RR^{n-1}$
thus splits into block form as $\mathrm{diag}(\rho_0(y),\rho_1(y),\dotsc,
\rho_s(y))$ for some linear maps $\rho_0 \colon \RR^{n-1} \to \RR$
and $\rho_i \colon \RR^{n-1} \to \CC$, $i=1,\dotsc,s$.  Using the
identification $\GrsOrig = (\RR \times \CC^s) \ltimes (\RR \times \CC^s)$, 
we can now write down an injective group homomorphism
\begin{equation*}
\begin{split}
\RR^{n-1} \ltimes_B (\RR \times \CC^s) &\to \GrsOrig \\
\big(y, (x, z_1, \dots, z_s)\big) &\mapsto 
\big( (\rho_0(y),\rho_1(y),\dotsc,\rho_s(y)),(x,z_1,\dotsc,z_s) \big).
\end{split}
\end{equation*}
Since $B$ is similar to $A$ and each power of $A$ appearing in the 
definition of the Geiges group has vanishing trace, the same is true 
for~$B$.  After taking the exponential, this translates to the fact that the
above map actually takes values in the subgroup $\GrsunOrig$: indeed,
$\abs{\exp(e^{i\theta})}^2 = \exp(2\RealPart e^{i\theta}) = \exp(\tr R_\theta)$.
We conclude that it is an isomorphism to $\GrsunOrig$ since the dimensions
match.  The desired isomorphism from $G_{2n-1}$ to $\GrsunOrig$ is now
obtained by composing the two isomorphisms we've constructed.

\section{Lutz twists and Giroux torsion in higher dimensions}
\label{sec:torsion}

In this section we examine the $(2n-1)$-dimensional generalizations of
Giroux torsion and Lutz twists that arise from any closed 
$(2n-3)$-dimensional manifold with a Liouville pair. We will begin with general
considerations and then turn to specific examples constructed using the
Liouville manifolds of \S\ref{sec:Liouville} to prove
Theorems~\ref{thm:generalisationTori} and \ref{thm:exist_weak_not_strong}
(in~\S\ref{subsec:weakNotStrong}) and~\ref{thm:stacking}
(in~\S\ref{subsec:tightNotWeak}) from the introduction.

Throughout the following, we choose an integer $n \ge 2$ and assume 
$M$ to be a closed oriented $(2n-3)$-dimensional manifold with a fixed 
Liouville pair $(\alpha_+,\alpha_-)$, writing the resulting 
positive/negative contact structures as $\xi_\pm = \ker\alpha_\pm$.
We will often consider manifolds of the form
$\RR \times \SS^1 \times M$ or $\SS^1 \times \SS^1 \times M$, with the
natural coordinates on the first two factors denoted by $s$ and~$t$ 
respectively.

\subsection{Torsion domains and the Lutz-Mori twist}
\label{subsec:LutzMori}

Given~$M$ with Liouville pair $(\alpha_+,\alpha_-)$, we define a
$1$-form on $\RR\times \SS^1 \times M$ by
\begin{equation}
\label{eqn:GT2}
\lambdaGT = \frac{1 + \cos s}{2}\, \alpha_+ + 
\frac{1 - \cos s}{2}\, \alpha_- + (\sin s) \, dt,
\end{equation}
and denote $\xiGT := \ker\lambdaGT$.

\begin{proposition}\label{prop:GTdomain}
The co-oriented distribution $\xiGT$ defined above is a positive
contact structure on $\RR \times \SS^1 \times M$, 
which can be viewed as an infinite chain of
Giroux domains $[k\pi, (k +1)\pi] \times \SS^1 \times M =
\left( M \times [k\pi, (k+1)\pi] \right) \times \SS^1$ glued together.
\end{proposition}
\begin{proof}
Let $\varphi \colon (0, \pi) \to \RR$ 
denote the orientation reversing diffeomorphism defined by
$\varphi(s) = \ln\frac{1 + \cos s}{\sin s}$.  This induces an
orientation preserving diffeomorphism
from the interior of $\Sigma := M \times [0,\pi]$ to $\RR\times M$, 
so pulling back
$\beta := \frac{1}{2}\left( e^u \alpha_+ + e^{-u} \alpha_-\right)$ gives a
Liouville form which defines on~$\Sigma$ the structure of
an ideal Liouville domain.  Regarding $\p\Sigma$ as the zero-set of the
function $\sin s$ and writing $u = \varphi(s)$, the Giroux domain 
$\Sigma \times \SS^1$, then inherits the contact form
\begin{equation*}
  \lambdaGT = (\sin s) \cdot \left[ dt + \frac{1}{2}\left( e^u\alpha_+ + e^{-u}\alpha_- \right)
  \right] \; ,
\end{equation*}
proving that $\lambdaGT$ is indeed a positive contact form on
$M \times [0, \pi] \times \SS^1 = [0,\pi] \times \SS^1 \times M$.
A similar argument proves the contact condition on
$[\pi,2\pi] \times \SS^1 \times M$, and the rest follows by periodicity.
\end{proof}

For any positive integer $k$, one 
can then define the \defin{Giroux $2k\pi$-torsion domain} modeled on 
$(M, \alpha_+, \alpha_-)$ as 
$([0, 2k\pi] \times \SS^1 \times M, \lambdaGT)$.

The fact that Giroux torsion is a filling obstruction in dimension three now
generalizes to the following immediate consequence of 
Theorem~\ref{thm:torsion}.  Note that for the case $n=2$, the additional
topological condition giving an obstruction to weak fillability is
equivalent to the condition that the embedding
$[0,2\pi] \times \SS^1 \times M \hookrightarrow V$ should separate~$V$.

\begin{corollary}
\label{cor:torsion}
If $(V,\xi)$ is a closed $(2n-1)$-dimensional contact manifold admitting
a contact embedding
$\iota \colon ([0, 2\pi] \times \SS^1 \times M,\xiGT) \hookrightarrow (V,\xi)$,
then $(V,\xi)$ is not strongly fillable.  Moreover, if
$\iota_*([\SS^1] \times C) = 0 \in H_2(V;\RR)$ for every $C \in H_1(M;\RR)$,
then $(V,\xi)$ is also not weakly fillable. \qed
\end{corollary}

The torsion domains $([0, 2k\pi] \times \SS^1 \times M,\xiGT)$ allow us to
define a ``twisting'' operation on contact structures that generalizes
the $3$-dimensional \emph{Lutz modification} along a pre-Lagrangian torus with
closed leaves, see \cite[Section~1.4]{ColinGH}.
Note that for any $k$, both boundary components of 
$[0,2k\pi] \times \SS^1 \times M$ are $\xiGT$-round hypersurfaces modeled on
$(M, \xi_+)$ (see \S\ref{subsec:round}). Now if $(V,\xi)$ is any
$(2n-1)$-dimensional contact manifold containing a $\xi$-round hypersurface $H
\subset V$ modeled on $(M, \xi_+)$, then we can cut~$V$ open along~$H$ and
insert $([0,2k\pi] \times \SS^1 \times M,\xiGT)$ such that the contact
structures glue together smoothly.  
The resulting manifold is diffeomorphic to~$V$, and it determines a new contact
structure~$\xi_k$ on~$V$ uniquely up to isotopy.  We shall say in this case that $\xi_k$ is obtained
from~$\xi$ by a $k$-fold \defin{Lutz-Mori twist along~$H$} (we use the name
Mori to emphasize that A.~Mori \cite{Mori_Lutz} introduced a similar 
modification along a
codimension 2 contact submanifold in dimension 5, see below).

Recall that any positive co-oriented contact structure~$\xi$ on an
oriented $(2n-1)$-dimensional manifold~$V$ induces an 
\emph{almost contact structure} on~$V$, i.e.~a reduction of the structure
group of~$TV$ to $U(n-1)$.  For our purposes, we can regard an
almost contact structure as equivalent to a choice of co-oriented
hyperplane distribution $\xi \subset TV$ together with a symplectic
structure on the bundle $\xi \to V$, and this choice is determined
uniquely up to homotopy when $\xi$ is contact.  The homotopy class of
almost contact structures amounts to a ``classical'' invariant that
one can use to distinguish non-isotopic contact structures.
As we will see, one of the important properties of the Lutz-Mori twist
is that it does not change this invariant, though it can change
the isomorphism class of the contact structure.

\begin{theorem}
\label{thm:LutzMoriTorsion}
Suppose $(V,\xi)$ is any contact manifold containing a closed $\xi$-round 
hypersurface $H$ modeled on $(M, \xi_+)$, where $\xi_+ = \ker\alpha_+$ and 
$(\alpha_+,\alpha_-)$ is a
Liouville pair on~$M$.  Then for any positive integer $k$, one can modify~$\xi$ 
near~$H$ by the $k$-fold Lutz-Mori twist as described above to define
a contact structure $\xi_k$ with the following properties:
\begin{enumerate}
\item $\xi$ and $\xi_k$ are homotopic through a family of 
almost contact structures.
\item $(V,\xi_k)$ is not strongly fillable if~$V$ is closed.
\item If $V$ is closed and the natural map $H_1(M;\RR) \to H_2(V;\RR)$ 
induced by the inclusion $\SS^1 \times M = H \hookrightarrow V$ is trivial, 
then $(V,\xi_k)$ is not weakly fillable.
\end{enumerate}
\end{theorem}

Before proving this theorem, we note the following
characterization of Liouville pairs, the proof of which is a simple computation.
It will be useful for understanding homotopy classes of
almost contact structures as well as Reeb vector fields in the next section.

\begin{lemma}\label{lemma:LPchar}
A pair of $1$-forms $(\alpha_+,\alpha_-)$ on an oriented 
$(2n-1)$-dimensional manifold~$M$ is a Liouville pair if and only if
for every pair of constants $C_+,C_- \ge 0$ not both zero,
$$
(C_+\alpha_+ - C_-\alpha_-) \wedge 
\left( C_+\, d\alpha_+ + C_-\, d\alpha_-\right)^{n-1} > 0\;.
$$
\qed
\end{lemma}

\begin{proof}[Proof of Theorem \ref{thm:LutzMoriTorsion}]
The last two statements are simply Corollary~\ref{cor:torsion}.
To prove the first, we can model the Lutz-Mori twist as follows.
By Lemma~\ref{lemma:round_neighborhood},
a neighborhood of~$H$ in $(V,\xi)$ can be identified with a neighborhood
of $\{0\} \times \SS^1 \times M$ in $(\RR\times \SS^1\times M, \xiGT)$,
i.e.~with $((-\epsilon,\epsilon) \times \SS^1 \times M,\xiGT)$ for
$\epsilon > 0$ sufficiently small.  Then given $k$, choose
a diffeomorphism $\varphi_k \colon (-\epsilon,\epsilon) \to 
(-\epsilon,2\pi k + \epsilon)$ as shown in 
Fig.~\ref{fig:functionsAlmostContactHomotopy}, with
fixed slope~$1$ outside the interval $(\epsilon/3,2\epsilon/3)$,
and define a new contact form on $(-\epsilon,\epsilon) \times \SS^1 \times M$
by
\begin{equation*}
  \lambda_k := \frac{1 + \cos \varphi_k(s)}{2}\, \alpha_+ + 
  \frac{1 - \cos \varphi_k(s)}{2}\, \alpha_- + (\sin \varphi_k(s)) \, dt \;.
\end{equation*}
For convenience, let us also set $\varphi_0(s) = s$, so
$\lambda_0 := \lambdaGT$ and
$\xi_0 := \xiGT$.  Then it will suffice to show that for each integer
$k \ge 0$, the almost contact structure induced by $\xi_k$
on $(-\epsilon,\epsilon) \times \SS^1 \times M$ admits a compactly
supported homotopy through almost
contact structures to a fixed almost contact structure independent of~$k$.

To see this, choose a smooth function
$\psi \colon [0,\epsilon] \to [0,1]$ which vanishes near the boundary and
equals~$1$ precisely on $[\epsilon/3,2\epsilon/3]$ 
(see Fig.~\ref{fig:functionsAlmostContactHomotopy}), 
and for $\tau \in [0,1]$,
define a smooth $1$-parameter family of nowhere zero $1$-forms
and co-oriented hyperplane fields by
\begin{equation*}
  \lambda_{k,\tau} = [1 - \tau\psi(s)]\, \lambda_k + \tau\psi(s)\, ds,
  \qquad
  \xi_{k,\tau} = \ker\lambda_{k,\tau} \;.
\end{equation*}
We have $\lambda_{k,\tau} = \lambda_k$ outside of some compact subset of
$(0,\epsilon) \times \SS^1 \times M$ for all~$\tau$,
while $\lambda_{k,0} \equiv \lambda_k$ and
$\lambda_{k,1}$ is everywhere independent of~$k$.  This shows that
the homotopy type of~$\xi_k$ as a co-oriented hyperplane field is
independent of~$k$.  It remains only to show that the homotopy
$\{ \xi_{k,\tau} \}_{\tau \in [0,1]}$ can be accompanied by a homotopy
$\{ \Omega_{k,\tau} \}_{\tau \in [0,1]}$ of symplectic bundle structures
such that $\Omega_{k,0} = d\lambda_k$ and $\Omega_{k,1}$ has no
$k$-dependence.

\begin{figure}[htbp]
  \centering
  \includegraphics[height=4cm,keepaspectratio]{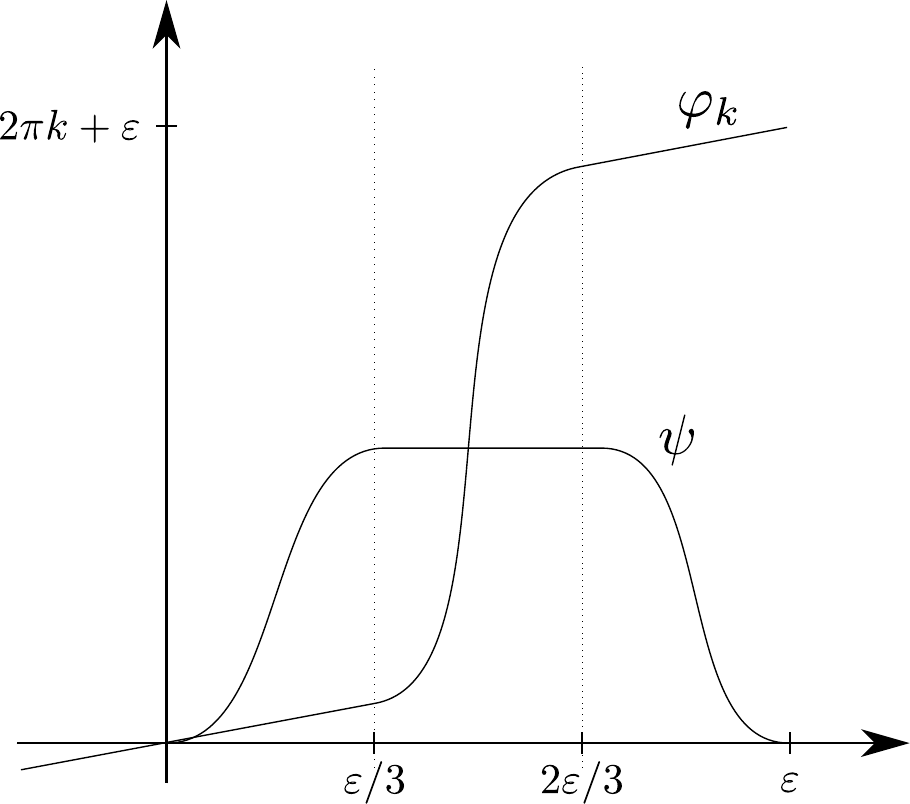}
  \caption{The function $\varphi_k$ is a diffeomorphism from
    $(-\epsilon,\epsilon)$ onto $(-\epsilon, 2\pi k + \epsilon)$ that
    has slope~$1$ outside $[\epsilon/3,2\epsilon/3]$.
    The function~$\psi$ is $1$ on the interval $[\epsilon/3,
    2\epsilon/3]$ and falls off to
    zero.}\label{fig:functionsAlmostContactHomotopy}
\end{figure}

We claim first that $\lambda_{k,\tau}$ is always contact, 
with the exception of $\lambda_{k,1} = ds$ on
$[\epsilon/3,2\epsilon/3] \times \SS^1 \times M$.
Indeed, 
$\lambda_{k,\tau} \wedge \left(d\lambda_{k,\tau}\right)^{n-1}
  = (1 - \tau\psi)^n \lambda_k \wedge (d\lambda_k)^{n-1}$
since the term $\tau\psi(s)\, ds$ vanishes in this product;
firstly it is closed, and secondly the only term where $dt$ appears in
$d\lambda_k$ is a multiple of $ds\wedge dt$.
Since $\lambda_k$ is contact, it follows that   $\lambda_{k,\tau}$ also
is whenever $\tau\psi < 1$.
Thus $d\lambda_{k,\tau}$ defines a suitable family of symplectic
bundle structures for $\tau < 1$, which we next would like to modify so
that it extends to $\tau=1$.  To facilitate this,
observe that whenever $\tau$ and $\psi$ are both close to~$1$,
the projection along the $s$-direction restricts to a fiberwise
isomorphism $\xi_{k,\tau} \to T(\SS^1 \times M)$.  Thus any symplectic
bundle structure $\Omega$ on $\xi_{k,\tau}$ can be identified via this
isomorphism with an $s$-dependent family of nondegenerate
(but not necessarily closed) $2$-forms $\hat{\Omega}(s)$ on 
$\SS^1 \times M$.  For $\Omega = d\lambda_{k,\tau}$ in particular, we find
that $\widehat{d\lambda_{k,\tau}}(s)$ belongs to the
contractible space $\Xi$ of $2$-forms on $\SS^1 \times M$ 
having the form
\begin{equation*}
  \omega := C_+\, d\alpha_+ + C_-\, d\alpha_- + \delta\, \alpha_+ \wedge \alpha_-
  + B\, dt \wedge (C_+ \alpha_+ - C_- \alpha_-)
\end{equation*}
for some constants $C_+,C_- \ge 0$, $B > 0$ and $\delta \in \RR$,
where $C_+$ and $C_-$ are assumed not both zero. In the case
of $\widehat{d\lambda_{k,\tau}}(s)$, one can compute:
\begin{equation*}
  C_\pm = \frac{1 \pm \cos\varphi_k}{2}, \qquad
  \delta = - \frac{1 - \tau\psi}{2\tau\psi}  \,
  \varphi_k'\,\sin\varphi_k
  \quad \text{ and } \quad
  B = \frac{1 - \tau\psi}{\tau\psi} \varphi_k'\;.
\end{equation*}
It turns out that \emph{any} $2$-form in~$\Xi$ is nondegenerate since
\begin{equation*}
  \omega^n = (n-1) B \, dt \wedge (C_+ \alpha_+ - C_- \alpha_-)
	\wedge (C_+\, d\alpha_+ + C_-\, d\alpha_-)^{n-1}
\end{equation*}
is nonzero due to Lemma~\ref{lemma:LPchar}. We can therefore solve the extension problem
to modify $\widehat{d\lambda_{k,\tau}}(s)$ for $\tau$ near~$1$ to a smooth
homotopy of nondegenerate $2$-forms that match 
$\widehat{d\lambda_{k,\tau}}(s)$
outside a neighborhood of $\{\epsilon/3 \le s \le 2\epsilon/3\}$ but also
extend to $\tau=1$ as nondegenerate forms with no dependence on~$k$.
Pulling back through the fiberwise isomorphism 
$\xi_{k,\tau} \to T(\SS^1 \times M)$, this determines a homotopy of almost
contact structures as desired.
\end{proof}

The original Lutz twist in dimension three modifies a contact structure in
the neighborhood of a transverse knot to produce one that is always
overtwisted, and Mori \cite{Mori_Lutz} generalized this to an operation
on contact $5$-manifolds along certain special contact submanifolds
of codimension~$2$.  In our context, Mori's construction generalizes as
follows: suppose $(V,\xi)$
is a $(2n-1)$-dimensional contact manifold containing a contact
submanifold $M \subset V$ of codimension~$2$ with trivial normal bundle
such that $\xi \cap TM = \xi_+$.
For any $k$, let $(Y_k,\xiGT)$ denote the result of blowing down
$([0, 2k\pi] \times \SS^1 \times M,\xiGT)$ along the $\xiGT$-round hypersurface
$\{0\} \times \SS^1 \times M$
as described in \S\ref{subsec:round}.  We can then remove a small neighborhood
of~$M$ from $(V,\xi)$ and glue in a correspondingly small neighborhood of
$(Y_k,\xiGT) \subset (Y_{k+1},\xiGT)$ such that the contact structures
match up.  The resulting manifold is again diffeomorphic to~$V$ and determines
a new contact structure~$\xi_k$ up to  isotopy,
and we shall say that $\xi_k$ is obtained from $\xi$
by a $k$-fold \defin{Lutz-Mori twist along~$M$}.

\begin{theorem}
\label{thm:LutzMori}
Suppose $(V,\xi)$ contains a closed
codimension~$2$ contact submanifold $M \subset V$ with trivial normal bundle
such that $\xi \cap TM = \xi_+$ where $\xi_+ = \ker\alpha_+$ for some
Liouville pair $(\alpha_+,\alpha_-)$ on~$M$.  Then for any positive integer $k$,
one can modify~$\xi$ near~$M$ by a $k$-fold Lutz-Mori twist as described
above to define a contact structure~$\xi_k$ with the following properties:
\begin{enumerate}
\item $\xi_k$ and $\xi$ are homotopic through a family of 
almost contact structures. 
\item $(V,\xi_k)$ is $PS$-overtwisted
	(cf.~Definition~\ref{defn:PSOT}) and not weakly fillable (if $V$ is closed).
\end{enumerate}
\end{theorem}
\begin{proof}
Since the homotopy of almost contact structures in our proof of
Theorem~\ref{thm:LutzMoriTorsion} had compact support in
$(0,\epsilon) \times \SS^1 \times M$, the argument can be carried over
verbatim to the present context to prove the first statement.
The presence of a \BLOB can be deduced from the general
Proposition~\ref{prop:bLob_in_blown_down_domain}, 
but also much more directly for the
concrete examples we discussed in \S\ref{sec:Liouville}.
Indeed, one can check that the torus bundles on which we constructed Liouville
pairs $(\alpha_+, \alpha_-)$ always contain an $n$-torus $T$ on which 
both $\alpha_+$ and $\alpha_-$ vanish.  In $[0, 2\pi] \times \SS^1 \times M$,
the contact form $\lambdaGT$ induces on $[0, \pi] \times \SS^1 \times T$
the integrable $1$-form $(\sin s)\, dt$, whose kernel is singular exactly along
the boundary. Blowing down $\{0\} \times \SS^1 \times M$ turns this domain into
a plastikstufe inside $(Y_k,\xiGT)$.

Non-fillability can also be deduced directly from Theorem~\ref{thm:torsion2},
with the technical advantage that it does not require any semipositivity
assumption thanks to the polyfold technology for holomorphic 
spheres~\cite{HoferWZ_GW}.
(The corresponding technology for holomorphic disks remains under development.)
\end{proof}

\begin{remark}
In the $3$-dimensional case one can also define the so-called ``half-Lutz
twist'' along a positively transverse knot, which both changes the homotopy
class of the contact structure and makes it overtwisted, producing a negatively
transverse knot at the core of the inserted tube.  The equivalent operation here
would be defined by replacing a neighborhood of $(M,\xi_+)$ in $(V,\xi)$ with
the domain $([\pi,2\pi] \times \SS^1 \times M,\xiGT)$ blown down along 
$\{\pi\} \times \SS^1 \times M$.  A variation on the above argument shows that
the resulting contact manifold is also $PS$-overtwisted, and in this case the
submanifold~$M$ at the center of the inserted ``tube'' inherits the
\emph{negative} contact structure~$\xi_-$ instead of~$\xi_+$.
\end{remark}

It is not remotely clear under what circumstances in general 
one can say that the modification from $\xi$ to
$\xi_k$ or $\xi_\ell$ produces non-isomorphic contact structures for
$k \ne \ell$, though we will show in the next few subsections that this is
at least sometimes the case for Lutz-Mori twists along round hypersurfaces.
In light of the flexibility exhibited by overtwisted contact structures
in dimension three, the following natural question arises:

\begin{qu}
If $\xi_k$ and $\xi_\ell$ are obtained from the same contact structure
by a $k$-fold and $\ell$-fold Lutz-Mori twist respectively along a fixed
contact submanifold of codimension~$2$, when are they isomorphic?
\end{qu}

\begin{remark}
It should be emphasized that Lutz-Mori twists cannot be performed along
arbitrary round hypersurfaces or codimension~$2$ contact submanifolds:
we always need to assume that the contact structure restricted to the
submanifold~$M$ admits a contact form belonging to a Liouville pair.
This is a serious constraint, as there are many smooth manifolds that
are known to admit contact structures but not Liouville pairs: for instance,
by \cite{Etnyre_planar} and \cite{AlbersBramhamWendl},
this is the case for any $3$-manifold whose contact 
structures are all known to be planar (e.g.~$\SS^3$ and $\SS^1 \times \SS^2$) or
partially planar (e.g.~$\TT^3$), as these can never admit strong
symplectic semifillings with disconnected boundary.
The fact that $3$-dimensional Lutz twists
can be inserted along any contact submanifold (here transverse knots) 
can then be seen as a consequence of the fact that every contact form on every
closed $1$-dimensional manifold obviously belongs to a Liouville pair.
\end{remark}

\subsection{Liouville pairs and Reeb vector fields}

In this section we describe the Reeb vector fields corresponding to contact
forms coming from Liouville pairs.  

Lemma~\ref{lemma:LPchar} implies that for any Liouville pair $(\alpha_+,\alpha_-)$
and constants $C_+,C_- \ge 0$ that do not both vanish, the $2$-form
$C_+\, d\alpha_+ + C_-\, d\alpha_-$ has maximal rank.  Its kernel therefore
defines a nonsingular line field on~$M$.

\begin{definition}\label{defn:hypertightLP}
A Liouville pair $(\alpha_+,\alpha_-)$ is called \defin{hypertight}
if for every pair of constants $C_+,C_- \ge 0$ that are not both zero,
$M$ admits no contractible loops tangent to
$\ker\left(C_+\, d\alpha_+ + C_-\, d\alpha_-\right)$.
\end{definition}

In particular, this condition implies that $\alpha_+$ and~$\alpha_-$ 
each admit no contractible closed Reeb orbits.  As one can check,
nonzero left-invariant vector fields on the Lie groups $\GrsunOrig$ of
\S\ref{sec:Liouville} never have closed orbits, thus we have the following
useful observation:

\begin{proposition}
\label{prop:hypertightExamples}
All the Liouville pairs constructed in \S\ref{sec:Liouville} are 
hypertight.
\end{proposition}

The following computation will be useful for understanding Reeb vector
fields on our examples in the next two subsections.  As a simple
application, it immediately implies that whenever the pair
$(\alpha_+,\alpha_-)$ is hypertight, $\lambdaGT$ admits no contractible
closed Reeb orbits.

\begin{lemma}\label{lemma:Reeb}
Suppose $(\alpha_+,\alpha_-)$ is a hypertight
Liouville pair on a manifold~$M$ of
dimension $2n-3 \ge 1$, and $f,g,h \colon \RR \to \RR$ are smooth functions
such that $f$ and~$g$ are both nonnegative and never vanish simultaneously,
and the $1$-form on $\RR\times \SS^1 \times M$ defined by
\begin{equation*}
  \lambda := f(s)\, \alpha_+ + g(s)\, \alpha_- + h(s)\, dt
\end{equation*}
is contact.  Then the Reeb vector field $R_\lambda$ associated to~$\lambda$
has the form $R_\lambda(s,t,m) = X_s(m) + u(s)\, \p_t$, where
$u \colon \RR \to \RR$ is a smooth function and $X_s$ is a smooth $1$-parameter
family of vector fields on~$M$, each of which either vanishes identically
or has no contractible closed orbits.
\end{lemma}
\begin{proof}
Computing $\lambda \wedge (d\lambda)^{n-1}$, we find that the contact
condition implies
\begin{equation}
\label{eqn:contCond}
[ (hf'-h'f) \alpha_+ + (hg'-h'g) \alpha_- ] \wedge (f\, d\alpha_+
+ g\, d\alpha_-)^{n-2} \ne 0,
\end{equation}
thus there is for each $s \in \RR$ a unique vector field $X_s$ on~$M$ satisfying
the conditions
\begin{equation*}
\begin{split}
( h f' - h' f) \,\alpha_+(X_s) + ( h g' - h' g )
\,\alpha_-(X_s) &= - h',\\
f\, d\alpha_+(X_s,\cdot) + g\, d\alpha_-(X_s,\cdot) &= 0,
\end{split}
\end{equation*}
This vector field vanishes precisely when $h'(s)=0$, and otherwise it
has no contractible orbits due to the hypertightness assumption.
The relation \eqref{eqn:contCond} also implies that $h(s)$ and $h'(s)$ can 
never simultaneously vanish, thus one can define a
function $u \colon \RR \to \RR$ by
\begin{equation*}
  u(s) = \begin{cases}
    \frac{1}{h} \left[ 1 - f\, \alpha_+(X_s) - g\, \alpha_-(X_s) \right]
    & \text{ when $h(s) \ne 0$},\\
    -\frac{1}{h'} \left[ f'\, \alpha_+(X_s) + g'\, \alpha_-(X_s) \right]
    & \text{ when $h(s) = 0$.}
  \end{cases}
\end{equation*}
With these definitions, it is straightforward to check that
$d\lambda(X_s + u(s)\,\p_t,\cdot) = 0$ and $\lambda(X_s + u(s)\,\p_t) = 1$.
\end{proof}

\subsection{A sequence of contact structures on $\TT^2 \times M$}
\label{subsec:weakNotStrong}

In order to prove Theorems~\ref{thm:generalisationTori}
and~\ref{thm:exist_weak_not_strong} from the introduction, 
we now consider an example that generalizes
the well-known sequence of weakly but not strongly fillable contact structures
on~$\TT^3$ \cite{Giroux_plusOuMoins,Eliashberg3Torus}.
Assume as usual that
$M$ has a Liouville pair $(\alpha_+,\alpha_-)$, and define for each
positive integer $k$ a contact structure $\xi_k$ on $\TT^2 \times M$ by
identifying the latter with $(\RR / 2k\pi\ZZ) \times \SS^1 \times M$
and setting $\xi_k := \xiGT$ via \eqref{eqn:GT2}.  
Theorems~\ref{thm:generalisationTori} and~\ref{thm:exist_weak_not_strong} are
then consequences of the following result, together with
Example~\ref{ex:S1product} below.

\begin{theorem}\label{thm:weakNotStrong}
For any closed manifold~$M$ with a Liouville pair $(\alpha_+,\alpha_-)$,
the sequence of contact structures $\{ \xi_k\}_{k > 0}$ on
$\TT^2 \times M$ defined above has the following properties:
\begin{enumerate}
\item $(\TT^2 \times M,\xi_1)$ is exactly fillable.
\item $(\TT^2 \times M,\xi_k)$ is not strongly fillable for any $k \ge 2$.
\item For any $k, \ell$, $\xi_k$ and $\xi_\ell$ are homotopic through
a family of almost contact structures.
\item If $(\alpha_+,\alpha_-)$ is hypertight 
(see Definition~\ref{defn:hypertightLP}) then every $\xi_k$ for $k \in \NN$
is hypertight, and no two of these contact structures are isotopic.
If additionally $\pi_1(M)$ has trivial center, then no two of these contact
structures are contactomorphic.
\item Suppose additionally that $\SS^1 \times M$ admits a closed $2$-form
$\omega$ such that for some constants $c_+, c_- > 0$ and
all sufficiently small $\epsilon > 0$, $\epsilon \omega + 
c_+ d\alpha_+ + c_- d\alpha_-$ is symplectic on $\SS^1 \times M$.
Then $(\TT^2 \times M,\xi_k)$ is weakly fillable for every~$k$.
\end{enumerate}
In particular, the first four statements are true for all the examples of 
\S\ref{sec:Liouville} with $\dim M \ge 3$ and $s = 0$ (see
Proposition~\ref{prop:hypertightExamples} and Lemma~\ref{lemma:centerGk}), 
and the fifth statement is also true when $\dim M = 3$, so~$M$ may be any
$\TT^2$-bundle over $\SS^1$ with hyperbolic monodromy
(see the discussion following Proposition~\ref{prop:torusBundle}).
\end{theorem}

\begin{example}\label{ex:S1product}
We do not know any examples of Liouville pairs with $\dim M \ge 5$ for which
we can verify the last condition, and this is why 
Theorem~\ref{thm:exist_weak_not_strong} in the introduction is stated only for
dimension five.
For $\dim M = 1$, the condition is the trivial observation that $\TT^2$
admits an area form, and our argument will then reproduce Giroux's
construction \cite{Giroux_plusOuMoins} of weak fillings for the
tight contact structures on~$\TT^3$, which directly inspired our general case.
Theorem~\ref{thm:exist_weak_not_strong} depends on finding closed 
$3$-manifolds $M$ with Liouville pairs such that $\SS^1 \times M$ is
symplectic, and this is also not hard. Every $\TT^2$-bundle over $\SS^1$ with
hyperbolic monodromy admits a hypertight Liouville pair that can be written as
follows: on $\RR \times \RR^2$ with coordinates $(t,x,y)$ let
\begin{equation*}
  \alpha_\pm = \pm e^t \, dx + e^{-t}\, dy \;.
\end{equation*}
Then if $A \in \SL(2,\ZZ)$ has eigenvalues $e^{\pm\tau}$ for $\tau > 0$,
one can find a lattice $\Lambda_A \subset \RR^2$ which is preserved by
the linear transformation $(x,y) \mapsto (e^{-\tau}x,
e^\tau y)$, so that $\alpha_+$ and $\alpha_-$ both descend to the
mapping torus
\begin{equation*}
M_A := \left(\RR \times (\RR^2 / \Lambda_A)\right) \Big/ (t,x,y) \sim 
(t + \tau,e^{-\tau} x,e^\tau y) \;.
\end{equation*}
Since $M_A$ fibers over~$\SS^1$, $\SS^1 \times M_A$ admits a symplectic form,
and we can write it explicitly as
\begin{equation*}
\omega = d\phi \wedge dt + dx \wedge dy \;,
\end{equation*}
where $\phi$ denotes the additional $\SS^1$-coordinate.  This form
satisfies $\omega \wedge d\alpha_\pm = 0$, hence 
$\epsilon\omega + c_+\, d\alpha_+ + c_-\, d\alpha_-$ 
is symplectic for all constants $\epsilon > 0$ and $c_\pm \in \RR$.
\end{example}

Our argument for distinguishing the contact structures~$\xi_k$ for different
values of~$k$ will use cylindrical contact homology as sketched by
Eliashberg-Givental-Hofer \cite{SFT}, a theory which in its most general form
has not yet been rigorously defined due to the difficulty of achieving
transversality for multiply covered holomorphic curves.  In our situation
however, we are in the lucky position of being able to rule out multiply covered
curves topologically.  Suppose $(V,\xi)$ is a closed contact manifold and
$\bar{a}$ denotes a free homotopy class of loops $S^1 \to V$.  We shall say that
a contact form $\lambda$ for $(V,\xi)$ is \defin{$\bar{a}$-admissible} if all
its Reeb orbits in the homotopy class~$\bar{a}$ are Morse-Bott and their periods
are uniformly bounded, and there are no contractible Reeb orbits.  The idea
sketched in \cite{SFT} is that if $\lambda$ is nondegenerate, one should define
a chain complex generated by a certain class of Reeb orbits homotopic to
$\bar{a}$, with the differential counting rigid holomorphic cylinders in the
symplectization for a generic choice of almost complex structure
adapted to~$\lambda$.  The
resulting homology is meant to depend only on $(V,\xi)$ and $\bar{a}$ up to
natural isomorphisms, so we denote it by $HC_*^{\bar{a}}(V,\xi)$.  Bourgeois
\cite{Bourgeois_thesis} has also explained how to extend this definition to
Morse-Bott contact forms by counting so-called ``holomorphic cascades.''

\begin{lemma}
\label{lemma:SFTworks}
Suppose $\bar{a}$ is a free homotopy class of loops in $(V,\xi)$ which is
primitive, i.e.~it is not a positive multiple of any other homotopy class, and
suppose $(V,\xi)$ admits an $\bar{a}$-admissible contact form.  Then the
cylindrical contact homology $HC_*^{\bar{a}}(V,\xi)$ sketched in \cite{SFT} is
well defined and can be computed as described in \cite{Bourgeois_thesis} by
counting holomorphic cascades for generic data associated to any
$\bar{a}$-admissible contact form.
\end{lemma}
\begin{proof}
We only need to supplement the standard Floer-theoretic picture with the
following observations.  First, every Reeb orbit homotopic to $\bar{a}$ must be
simply covered, thus every holomorphic curve having only one positive end, which
is asymptotic to such an orbit, is guaranteed to be somewhere injective.
Transversality for these curves can therefore be achieved via a generic
perturbation of the almost complex structure, using the standard result of
Dragnev \cite{Dragnev} (see also the appendix of \cite{Bourgeois_homotopy}).
Secondly, if $\lambda$ has no contractible Reeb orbits and $\lambda'$ is a
sufficiently small nondegenerate perturbation of it as in
\cite{Bourgeois_thesis}, then one may assume every contractible Reeb orbit for
$\lambda'$ to have arbitrarily large period.  Then since the periods of Reeb
orbits homotopic to $\bar{a}$ are bounded, one can choose a
generic almost complex structure~$J$ adapted to $\lambda'$ and
define a subcomplex of the
usual complex for the data $(\lambda',J)$ by taking as generators all the Reeb
orbits up to a certain period, chosen so that all perturbations of the
Morse-Bott orbits homotopic to $\bar{a}$ are included but holomorphic planes can
never appear in the relevant compactifications because they have
too much energy.  For appropriate choices of the
period cutoff, the standard construction of natural isomorphisms (i.e.~by
counting rigid holomorphic cylinders in symplectic cobordisms) suffices to prove
that the homology is independent of auxiliary choices.
\end{proof}

\begin{proof}[Proof of Theorem~\ref{thm:weakNotStrong}]
Statements~(2) and~(3) in the theorem follow immediately from 
Corollary~\ref{cor:torsion} and Theorem~\ref{thm:LutzMoriTorsion}
respectively.  We shall now prove statements~(1), (5) and~(4), in that order.

\textsl{Proof of~(1).}
An exact filling of $(\TT^2 \times M,\xi_1)$ can be constructed as
the product of two Liouville domains of the form
$([-1,1] \times \SS^1, \sigma\, d\theta)$ and
$([-c,c] \times M,e^s \alpha_+ + e^{-s} \alpha_-)$ with rounded corners,
where $c > 0$ may be assumed arbitrarily large and
$(\sigma,\theta)$ denote
the natural coordinates on $[-1,1] \times \SS^1$.

\textsl{Proof of~(5).}
If $\SS^1 \times M$ also admits a $2$-form $\omega$ as in the
condition of statement~(5), 
then we can modify the exact filling constructed above 
to define weak fillings of
every $(\TT^2 \times M,\xi_k)$, using the fact that the latter is
naturally a $k$-fold cover of $(\TT^2 \times M,\xi_1)$.
Indeed, the assumption implies that we can find $s_0 \in (-1,1)$ such that
for any $\epsilon > 0$ sufficiently small, the $2$-form
\begin{equation}
\label{eqn:omegaEpsilon}
\epsilon\omega + e^{s_0}\, d\alpha_+ + e^{-s_0}\, d\alpha_-
\end{equation}
is symplectic on $\SS^1 \times M$.  Now observe that
since the weak filling condition is open with respect to the symplectic form,
$([-1,1] \times \SS^1 \times [-c,c] \times M,\omega_\epsilon)$
with rounded corners and
\begin{equation*}
  \omega_\epsilon := d\left[ e^s\, \alpha_+ + e^{-s}\, \alpha_- + 
    \sigma\, d\theta \right] + \epsilon \omega
\end{equation*}
is also a weak filling of $(\TT^2 \times M,\xi_1)$ if $\epsilon > 0$
is sufficiently small, and for any $\sigma_0 \in (-1,1)$ its restriction 
to the interior submanifold
\begin{equation*}
  X_0 := \{\sigma_0\} \times \SS^1 \times \{s_0\} \times M \subset
  ([-1,1] \times \SS^1 \times [-c,c] \times M,\omega_\epsilon)
\end{equation*}
is precisely \eqref{eqn:omegaEpsilon}.  Thus we have a weak filling of
$(\TT^2 \times M,\xi_1)$ diffeomorphic to $\DD^2 \times \SS^1 \times M$
and containing $\{0\} \times \SS^1 \times M$
as a symplectic submanifold.  For any $k$, the $k$-fold symplectic
branched cover of this, branched at $\{0\} \times \SS^1 \times M$,
gives a weak filling of $(\TT^2 \times M,\xi_k)$.

\textsl{Proof of~(4).}
Assume now that $(\alpha_+,\alpha_-)$ is a hypertight Liouville pair.
Lemma~\ref{lemma:Reeb} then implies that $\lambdaGT$ has no contractible
Reeb orbits.  

We next compute the cylindrical contact homology 
of $(\TT^2 \times M,\xi_k)$, which is a straightforward adaptation 
of the calculation for the
tight $3$-tori explained in \cite[\S 4.2]{Bourgeois_CH}.  
Let $\bar{a}$ denote the free
homotopy class of the loop $\SS^1 \to \SS^1 \times \SS^1 \times M : \phi \mapsto
(\text{const},\phi,\text{const})$.  Applying Lemma~\ref{lemma:Reeb}
again, the Reeb orbits of $\lambdaGT$ in
homotopy class $\bar{a}$ on $\RR / (2\pi k \ZZ) \times \SS^1 \times M$
consist of precisely $k$ Morse-Bott families foliating the submanifolds
$\{ \cos s = 0,\ \sin s = 1 \} \cong \SS^1 \times M$.
Moreover, all of these orbits
have the same period, thus our contact form is $\bar{a}$-admissible in
the sense of Lemma~\ref{lemma:SFTworks}.
Now for any choice of admissible almost complex
structure~$J$ on the symplectization of $(\TT^2 \times M,\xi_k)$, there can
never be any index~$1$ $J$-holomorphic cylinders connecting two orbits in
homotopy class~$\bar{a}$ since it would have zero energy.  After making a
nondegenerate perturbation as explained in \cite{Bourgeois_CH}, 
nondegenerate orbits in homotopy class~$\bar{a}$ are in
one-to-one correspondence with the critical points of a Morse function
on the parameter space of the Morse-Bott families, i.e.~on~$M$.
Similarly, the holomorphic cylinders for the perturbed data 
correspond to so-called
``holomorphic cascades'' for the unperturbed data, and in the absence of
actual holomorphic cylinders, these are in one-to-one correspondence
with gradient flow lines on~$M$.  We conclude that 
$HC_*^{\bar{a}}(\TT^2 \times M,\xi_k)$ is isomorphic (up to a shift in the
grading) to the direct sum of $k$ copies of the Morse homology of~$M$,
which is simply the singular homology $H_*(M)$.  

Observe also that if $\bar{b} \ne \bar{a}$ 
is any other free homotopy class of loops
in $\TT^2 \times M$ whose projections to~$M$ are contractible, then
there are no Reeb orbits homotopic to $\bar{b}$ at all, hence
$HC_*^{\bar{b}}(\TT^2 \times M, \xi_k)$ is trivial.

The above computation shows that if $k \ne \ell$, then there can be
no contactomorphism $(\TT^2 \times M, \xi_k) \to
(\TT^2 \times M, \xi_\ell)$ whose action on $\pi_1(\TT^2 \times M)$
preserves the subgroup
$$
G := \pi_1(\TT^2) \times \{1\} \subset \pi_1(\TT^2) \times
\pi_1(M) = \pi_1(\TT^2 \times M).
$$
Indeed, we have computed the cylindrical contact homology for all 
homotopy classes in this subgroup, and by Lemma~\ref{lemma:SFTworks},
these computations would have to match if such a contactomorphism
existed.  This already implies that $\xi_k$ and $\xi_\ell$
cannot be isotopic.  To show that they are not even diffeomorphic,
we add the assumption that $\pi_1(M)$ has trivial center: then the
center of $\pi_1(\TT^2 \times M)$ is~$G$, which is therefore
preserved by \emph{every} automorphism of $\pi_1(\TT^2 \times M)$.
\end{proof}

\subsection{Hypertight but not weakly fillable}
\label{subsec:tightNotWeak}

We now construct a family of examples in all dimensions that implies
Theorem~\ref{thm:stacking} from the introduction.  Throughout this section, we
denote by $\Sigma_g$ the closed oriented surface of genus~$g$, and by
$\Sigma_{g,m}$ the compact oriented surface with genus~$g$ and~$m$ boundary
components.

\begin{theorem}\label{thm:tightNotWeak}
Suppose $M$ is any closed $(2n-3)$-dimensional manifold admitting a
hypertight Liouville pair.
Then for any integer $g > 0$, $\Sigma_{2g} \times M$
admits a sequence of contact structures $\{ \xi_k\}_{k > 0}$
with the following properties:
\begin{enumerate}
\item $(\Sigma_{2g} \times M,\xi_1)$ is exactly fillable.
\item $(\Sigma_{2g} \times M,\xi_k)$ is not weakly fillable for any $k \ge 2$.
\item $(\Sigma_{2g} \times M,\xi_k)$ is hypertight for all~$k$.
\item For any $k \ne \ell$, $\xi_k$ and $\xi_\ell$ are homotopic through a
family of almost contact structures but are not isotopic.  If additionally
$\pi_1(M)$ has trivial center and is solvable, then they are not even
contactomorphic.
\end{enumerate}
In particular, all of these statements are true for the Liouville pairs 
defined from totally real number fields in \S\ref{sec:Liouville}.
\end{theorem}

The contact structures $\xi_k$ on $\Sigma_{2g} \times M$ will be constructed
using a simple generalization of the blow-down 
operation along round hypersurfaces that was introduced in \S\ref{subsec:round}.
To start with, we consider $(Z_k,\xiGT)$ where $Z_k :=
[0,(2k-1)\pi] \times \SS^1 \times M$, so the two boundary components
\begin{equation*}
  \p_+Z_k := \{0\} \times \SS^1 \times M, \qquad
  \p_-Z_k := \{(2k-1)\pi\} \times \SS^1 \times M
\end{equation*}
are $\xiGT$-round hypersurfaces
modeled on $(M,\xi_+)$ and $(-M,\xi_-)$ respectively.  At $\p_+Z_k$
in particular, we find by Lemma~\ref{lemma:round_neighborhood}
a collar neighborhood identified with 
$([0,\epsilon) \times \SS^1 \times M, \ker(\alpha_+ + s\, dt))$ for some
$\epsilon > 0$.
Now choose a Liouville form~$\beta$ on
$\Sigma_{g,1}$ such that $\int_{\p\Sigma_{g,1}} \beta = \epsilon$.
Then $\p\Sigma_{g,1}$ has a neighborhood $\nN(\p\Sigma_{g,1}) \subset
(\Sigma_{g,1},\beta)$ that can be identified with 
$((0,\epsilon] \times \SS^1, s\, dt)$, defining a natural embedding
$\Phi_+ \colon \nN(\p\Sigma_{g,1}) \times M \hookrightarrow (0,\epsilon] \times
\SS^1 \times M \subset Z_k$ with $\Phi_+^*\xiGT = \ker(\beta + \alpha_+)$.
Similarly, the other end of $\mathring{Z}_k$ admits 
an orientation preserving embedding 
$\Phi_- \colon \nN(\p\Sigma_{g,1}) \times (-M) \hookrightarrow \mathring{Z}_k$
such that $\Phi_-^*\xiGT = \ker(\beta + \alpha_-)$.  We can therefore glue
three pieces together to define
\begin{multline*}
(\Sigma_{2g} \times M,\xi_k) := (\Sigma_{g,1} \times M,\ker(\beta + \alpha_+))
\cup_{\Phi_+} (\mathring{Z}_k,\xiGT) \\ \cup_{\Phi_-} 
(\Sigma_{g,1} \times (-M), \ker(\beta + \alpha_-)).
\end{multline*}
Note that if $g=0$, this construction is equivalent to blowing down
$(Z_k,\xiGT)$ at both boundary components as defined in
\S\ref{subsec:round}, and we shall think of the more general operation
defined here as ``blowing down with genus~$g$.''

We now proceed to construct a model of $(\Sigma_{2g} \times M,\xi_k)$ with a
more tractable Reeb vector field.
The disadvantage of using $\lambdaGT$ for this purpose is that it cannot easily
be related to the normal forms
$\alpha_\pm + s\, dt$ coming from Lemma~\ref{lemma:round_neighborhood},
as for instance near $\p_+Z_k$, the $\alpha_-$-term
in $\lambdaGT$ is small but not identically vanishing.  The following
lemma allows us to eliminate it entirely after a small adjustment which
essentially replaces the Liouville form $e^s\alpha_+ + e^{-s}\alpha_-$ 
on $\RR \times M$ by one which is explicitly the completion of a Liouville
domain $[-c, c] \times M$.

\begin{lemma}\label{lemma:cutoff}
Choose a smooth cutoff function $\psi \colon \RR \to [0,1]$ that equals~$0$
on $(-\infty,0]$ and~$1$ on $[1,\infty)$.  Then
for any Liouville pair $(\alpha_+,\alpha_-)$ on a $(2n-1)$-dimensional
manifold~$M$, the $1$-form
\begin{equation*}
  \beta := \psi(c + s)\, e^s \alpha_+ + \psi(c - s)\, e^{-s} \alpha_-
\end{equation*}
is Liouville if $c > 0$ is a sufficiently large constant.
\end{lemma}
\begin{proof}
The claim is immediate whenever $\psi' = 0$, so it will suffice to examine
$d\beta$ on the segments $\{ -c \le s \le -c + 1 \}$ and
$\{ c - 1 \le s \le c \}$.  On the former, we have
$\beta = \psi_c(s) e^s\, \alpha_+ + e^{-s} \alpha_-$ where
$\psi_c(s) := \psi(c+s)$.  Thus
\begin{equation*}
\begin{split}
d\beta^n &= n\, ds \wedge (\psi_c e^s\alpha_+ - e^{-s}\alpha_- +
\psi_c' e^s \alpha_+) \wedge (\psi_c e^s\, d\alpha_+ + e^{-s}\, d\alpha_-)^{n-1} \\
&= n e^{-ns}\, ds \wedge \Big[ (\psi_c e^{2s} \alpha_+ - \alpha_-) \wedge
(\psi_c e^{2s}\, d\alpha_+ + d\alpha_-)^{n-1} \\
&\qquad + e^{2s} \psi_c' \, \alpha_+ \wedge (\psi_c e^{2s}\, d\alpha_+ +
d\alpha_-)^{n-1} \Big].
\end{split}
\end{equation*}
In this last expression, the first term in the brackets can be made
arbitrarily close to $-\alpha_- \wedge d\alpha_-^{n-1} > 0$ by assuming
$c > 0$ large, while the second term can be made arbitrarily close to~$0$,
hence the sum is positive.  A similar argument also works for the
segment $\{ c - 1 \le s \le c \}$.
\end{proof}

Combining this lemma with the reparametrization trick in the
proof of Proposition~\ref{prop:GTdomain}, we can now introduce a
convenient modification of the contact form $\lambdaGT$: on
$Z_k = [0,(2k-1)\pi] \times \SS^1 \times M$, there exists a contact form of type
\begin{equation*}
  \lambda_k = f(s)\, \alpha_+ + g(s)\, \alpha_- + h(s)\, dt
\end{equation*}
for some smooth functions $f,g,h \colon [0,(2k-1)\pi] \to \RR$, such that
for some small constant $\epsilon > 0$:
\begin{itemize}
\item $\lambda_k = \lambdaGT$ on $[2\epsilon,(2k-1)\pi - 2\epsilon] \times
\SS^1 \times M$,
\item $\lambda_k$ is everywhere $C^1$-close to~$\lambdaGT$,
\item $\lambda_k = \alpha_+ + s\, dt$ on $[0,\epsilon] \times \SS^1 \times M$,
\item $\lambda_k = \alpha_- + \left[(2k-1)\pi - s\right]\, dt$ on
$[(2k-1)\pi - \epsilon,(2k-1)\pi] \times \SS^1 \times M$.
\end{itemize}
Then if $\beta$ denotes the Liouville form on $\Sigma_{g,1}$ as described
above with collar neighborhood $\nN(\p\Sigma_{g,1}) = (0,\epsilon] \times \SS^1$
in which $\beta = s\, dt$, we can smoothly glue $\Sigma_{g,1} \times M$
with contact form $\lambda_k := \alpha_+ + \beta$ to the interior of
$(Z_k,\lambda_k)$ along $(0,\epsilon] \times \SS^1 \times M$.
Similarly, defining the auxiliary coordinate $s' := (2k-1)\pi - s \in
[0,\epsilon]$ on the opposite collar neighborhood, we can glue this
neighborhood to $\Sigma_{g,1} \times M$ with contact form
$\lambda_k := \alpha_- + \beta$ so that the coordinates $(s',t)$ match the
collar $\nN(\p\Sigma_{g,1}) = (0,\epsilon] \times \SS^1$.
The kernel of $\lambda_k$ is now isotopic to $\xi_k$.

\begin{proof}[Proof of Theorem~\ref{thm:tightNotWeak}]
The claim regarding almost contact structures follows by the same
argument as in Theorem~\ref{thm:LutzMoriTorsion}.  With this
understood, we shall
now proceed to prove items~(3) and~(4) from the statement of the theorem,
and after that prove items~(1) and~(2).

\textsl{Proof of~(3) and~(4)}.
The contact form $\lambda_k$ constructed above determines a Reeb vector
field $R_{\lambda_k}$ that is given by Lemma~\ref{lemma:Reeb}
on $[\epsilon,(2k-1)\pi - \epsilon] \times \SS^1 \times M$ and matches
the Reeb vector fields of $\alpha_+$ and $\alpha_-$ respectively on the
two copies of $\Sigma_{g,1} \times M$.  While this vector field does have
nullhomologous closed orbits, none of them are contractible if $g > 0$
since $\p\Sigma_{g,1}$ is not contractible in $\Sigma_{2g}$.  
Similarly, for $g > 0$ one can define the cylindrical contact homology
$HC_*^{\bar{a}}(\Sigma_{2g} \times M,\xi_k)$ for any primitive homotopy 
class~$\bar{a}$ due to Lemma~\ref{lemma:SFTworks}.
A repeat of the argument in the proof of Theorem~\ref{thm:weakNotStrong} then
shows that for $k \ne \ell$, there is no contactomorphism
$$
(\Sigma_{2g} \times M,\xi_k) \to (\Sigma_{2g} \times M,\xi_\ell)
$$
whose action on $\pi_1(\Sigma_{2g} \times M)$ preserves $\pi_1(\Sigma_{2g})$.
So in particular, $\xi_k$ and $\xi_\ell$ are not isotopic. Under the additional
assumption on $\pi_1(M)$, they are not even contactomorphic due to
Lemma~\ref{lemma:deCornulier} below.

\textsl{Proof of~(1)}.
An exact filling of $(\Sigma_{2g} \times M,\xi_1)$ can be constructed as the
product of the two Liouville domains $(\Sigma_{g,1},\beta)$ and
$([-c,c] \times M, e^s \alpha_+ + e^{-s} \alpha_-)$ for sufficiently
large~$c$.  

\textsl{Proof of~(2)}.
Corollary~\ref{cor:torsion} implies that
$(\Sigma_{2g} \times M,\xi_k)$ is not weakly fillable for $k \ge 3$;
note that here we need the fact that for any $1$-cycle $C$ in~$M$,
the $2$-cycle $\{\text{const}\} \times \SS^1 \times C$ in $Z_k \subset
\Sigma_{2g} \times M$ can be realized as the boundary of
$\Sigma_{g,1} \times M$ and is thus nullhomologous.

At this point we've proved everything except the fact that
$(\Sigma_{2g} \times M,\xi_2)$ is not weakly fillable.  Since this
already suffices to prove Theorem~\ref{thm:stacking}, and the
non-fillability of~$\xi_2$ doesn't quite follow from our previous results
as stated, we shall content ourselves with a sketch of the proof.
The idea is analogous to the proof of Theorem~\ref{thm:torsion2}, but using
a straightforward generalization of the surgery in
\S\ref{section:surgery} to accommodate boundary components that are,
in the terminology introduced above, blown down with genus.  In particular,
$(\Sigma_{2g} \times M,\xi_2)$ can be realized as a chain of three Giroux domains
$G_0 \cup G_1 \cup G_2$ glued end to end, with the dangling ends of
$G_0$ and $G_2$ blown down with genus~$g$.  Now if we perform surgery to
remove the interiors of $G_0$ and $G_1$, we obtain a symplectic cobordism
to a manifold with three connected components
\begin{equation*}
  (M \times \Sigma_g) \sqcup (M \times \SS^2) \sqcup (V',\xi') \;,
\end{equation*}
where $(V',\xi')$ is a weakly filled boundary component and the other two
components are foliated by symplectic submanifolds $\{*\} \times \Sigma_g$
and $\{*\} \times \SS^2$ respectively.  Then if $(\Sigma_{2g}\times M,\xi_2)$
is assumed to be weakly fillable, one can derive a contradiction as in
the proof of Theorem~\ref{thm:torsion2} by examining
the moduli space of holomorphic spheres that emerge from the symplectic
submanifolds $\{*\} \times \SS^2$.  This only involves one feature not
already present in the proof of 
Theorem~\ref{thm:torsion2}: the holomorphic spheres
cannot approach the boundary component $M \times \Sigma_g$.  This is guaranteed
if one uses a product complex structure near this
boundary component, because then every somewhere injective holomorphic curve
touching a neighborhood of it must be of the form $\{*\} \times \Sigma_g$,
and no sequence of holomorphic spheres can converge to any cover of these
curves since such a cover would necessarily have positive genus.
\end{proof}

In the above proof we used the following algebraic lemma, whose proof
was kindly explained to us by Yves de Cornulier.
\begin{lemma}
\label{lemma:deCornulier}
Suppose $\Sigma$ is a closed oriented surface of genus at least two. 
If $G$ is any solvable group with trivial center, then any automorphism of 
$\pi_1(\Sigma) \times G$ preserves  $\pi_1(\Sigma)$.
\end{lemma}
\begin{proof}
We set $H = \pi_1(\Sigma) \times G$.  Our goal will be to show that $G$ is the unique
maximal normal solvable subgroup of~$H$, thus $G$ is preserved by any
automorphism.  Since $G$ has trivial center, its centralizer in $H$ is
$\pi_1(\Sigma)$, which is therefore also preserved by any automorphism.

We now prove the claim about $G$. Suppose $G_1$ is a normal solvable 
subgroup of~$H$. The projection $p(G_1)$ of $G_1$ into $\pi_1(\Sigma)$ is 
normal in $\pi_1(\Sigma)$ and solvable.
We now view $\pi_1(\Sigma)$ as a Zariski dense subgroup of 
$\mathrm{PSL}(2, \RR)$. The Zariski closure of $p(G_1)$ 
is still solvable and
is normal in the closure of $\pi_1(\Sigma)$, hence trivial because
$\mathrm{PSL}(2, \RR)$ is simple and not solvable.  Thus $p(G_1)$ 
is trivial and $G_1 \subset G$. 
\end{proof}

\appendix

\section{Cotamed complex structures: Existence and convexity}

\subsection{Contractibility of the space of cotamed almost complex
  structures}
\label{section:contractibility_of_space_J}

To go from the linear situation to global existence results on a manifold we
will need the following result.

\newtheorem*{propositionContractibilityComplex}{Proposition~\ref{prop:space_cotamed_contractible}}

\begin{propositionContractibilityComplex}[Sévennec]
        The space of complex structures on a finite dimensional 
vector space tamed by two
        given symplectic forms is either empty or contractible.
\end{propositionContractibilityComplex}

Using the fact that the space of complex structures tamed by a symplectic
form is nonempty (which follows for instance by the linear Darboux theorem), 
and applying the proposition above twice to the same symplectic form, we
recover as a special case the classical result of Gromov that states
that the space of tamed complex structures is contractible.  The proof
of the proposition uses the following two lemmas, of which
the first is more or less standard.

\begin{lemma}[Cayley, Sévennec]
\label{lemma:cayley_sevennec}	
Let $V$ be a real finite dimensional 
vector space and $\jJ(V)$ the space of complex
structures on $V$.  We can define for any fixed $J_0 \in \jJ(V)$ a
map
\begin{equation*}
\mu_{J_0}\colon J \mapsto (J + J_0)^{-1} \cdot (J - J_0)
\end{equation*}
which is a diffeomorphism from
\begin{equation*}
\jJ_{J_0}^*(V) := \bigl\{ J \in \jJ(V)\ |\  J + J_0 \in \GL(V)\bigr\}
\end{equation*}
to 
\begin{equation*}
\mathcal{A}_{J_0}^*(V) := \bigl\{ A \in \End(V)\ |\  A J_0 = -J_0 A\text{
and } A - I\in \GL(V)\bigr\} \;.
\end{equation*}
The inverse of this map is given by $\mu_{J_0}^{-1} \colon A \mapsto
(A- I)J_0(A - I)^{-1}$.
\end{lemma}
\begin{proof}
  One can view $\mathcal{A}^*_{J_0}(V)$ as the set of $J_0$-complex
  antilinear maps that do not have any eigenvalue equal to $1$.
  Using the equations $(J-J_0)\,J_0 = -J\,(J-J_0)$ and $(J+J_0)\,J_0 =
  J\,(J+J_0)$, one sees that the image of $\mu_{J_0}$ consists of
  $J_0$-complex antilinear maps, and $\mu_{J_0}(J) - I =
  -2\,(J + J_0)^{-1}\, J_0$ is invertible.
\end{proof}

\begin{lemma}[Sévennec]
\label{lemma:image_sevennec_map_convex}
Let $(V, \omega)$ be a finite dimensional 
symplectic vector space and denote by
$\jJ_t(\omega) \subset \jJ(V)$ the space of complex structures tamed
by $\omega$.  Choosing any $J_0 \in \jJ_t(\omega)$, it follows that
$\jJ_t(\omega)$ lies in $\jJ_{J_0}^*(V)$, and the image of
$\jJ_t(\omega)$ under the associated map $\mu_{J_0}$ is a convex
domain in $\mathcal{A}_{J_0}^*(V)$.
\end{lemma}

We first explain how to prove
Proposition~\ref{prop:space_cotamed_contractible} using the above
lemma.  Suppose there is a complex structure $J_0$ tamed by $\omega_0$
and $\omega_1$.  The space of cotamed complex structures
$\jJ_t(\omega_0) \cap \jJ_t(\omega_1)$ is then diffeomorphic under the
map~$\mu_{J_0}$ to the intersection of the convex
subsets given by the lemma. This intersection is again convex and
hence contractible.

\begin{proof}[Proof of Lemma~\ref{lemma:image_sevennec_map_convex}]
For any complex structure $J$ tamed by $\omega$, the endomorphism $J + J_0$ is
invertible because for any nonzero $w$, we have
$\omega\bigl(w, (J + J_0)\,w\bigr) > 0$, so in
particular $(J + J_0)\,w$ is not zero.  This proves the first part of the lemma.

Now fix a nonzero vector $v\in V$, and let $C_v$ be the set of
$A\in \End(V)$ that anticommute with $J_0$, and that satisfy
\begin{equation*}
  \omega\bigl((A - I)\, v, (A- I)\,J_0v\bigr) =   -\omega\bigl((A - I)\, v,
  J_0(A+ I)\, v\bigr) > 0 \;.
\end{equation*}
We now prove that $C_v\subset \End(V)$ is convex. 
Every segment $A_s = (1-s)\,A_0 + s\,A_1$ with $s\in[0,1]$ for
arbitrary $A_0,A_1 \in C_v$ defines a polynomial of degree~$2$
\begin{equation*}
  P(s) = -\omega\bigl((A_s - I)\,v, J_0(A_s + I)\,v\bigr) \;,
\end{equation*}
and the above inequality corresponds to checking that $P(s)$ is
positive for all values $s\in[0,1]$.  The leading coefficient
$-\omega\bigl((A_1-A_0)\, v, J_0(A_1 - A_0)\, v\bigr)$ of $P(s)$ is
never positive, because $J_0$ tames $\omega$, so that $P(s)$ is either
a line or a parabola facing downward.  In both cases $P(s) \ge \min
\{P(0), P(1)\} > 0$ for all $s\in(0,1)$ so the inequality holds
for the whole segment $A_s$.

Note that $C_v\ne \emptyset$ since $\0 \in C_v$.  Define the
intersection
\begin{equation*}
  C^* := \bigcap_{v\ne 0} C_v \;,
\end{equation*}
which is a nonempty convex subset of $\End(V)$.  In fact, one has
$C^* \subset \mathcal{A}_{J_0}^*(V)$, because if
there were a matrix $A\in C^*$ with $\det (A-I) = 0$, then $A$ would
have an eigenvector $w\in V$ with eigenvalue~$1$, but then
$-\omega\bigl((A - I)\,w, J_0(A + I)\,w\bigr) = 0$ so that $A\notin
C_w$.

Since $C^*$ lies in the domain of $\mu_{J_0}^{-1}$ and $\jJ_t(\omega)$
lies in the domain of $\mu_{J_0}$, we have
$C^* = \mu_{J_0}\bigl(\jJ_t(\omega)\bigr)$, so that the image of the complex
structures tamed by $\omega$ is convex as we wanted to show.
\end{proof}

\subsection{Existence of a cotamed complex structure}
\label{section:linear_algebra}

In this appendix, we prove Proposition~\ref{prop:carac_cotamed}, which we now
recall:

\newtheorem*{propositionDominatingComplex}{Proposition~\ref{prop:carac_cotamed}}

\begin{propositionDominatingComplex}
  Let $V$ be a finite dimensional 
real vector space equipped with two symplectic forms
  $\omega_0$ and $\omega_1$.
  The following properties are equivalent:
  \begin{enumerate}
  \item the segment between $\omega_0$ and $\omega_1$ consists of
    symplectic forms
  \item the ray starting at $\omega_0$ and directed by $\omega_1$
    consists of symplectic forms
  \item there is a complex structure $J$ on $V$ tamed both by $\omega_0$
    and by $\omega_1$.
  \end{enumerate}
\end{propositionDominatingComplex}

The equivalence between (1) and (3) was explained to us by Jean-Claude Sikorav.
It relies on the simultaneous reduction of symplectic forms.
Specifically, we need \cite[Theorem~9.1]{LancasterRodman} which we shall state
(in a slightly weakened form) and reprove (in its full force) below as
Proposition~\ref{prop:LR}, since the very general context of
\cite{LancasterRodman} makes it hard to read for people interested only in
the symplectic case.

Recall that according to the linear Darboux theorem, any symplectic form on a
$2n$-dimensional vector space is represented in some basis by the standard
matrix 
\begin{equation*}
\Omega_{2n} =
\begin{pmatrix}
      \0 &  \mathbf{1} \\
      -\mathbf{1} & \0
 \end{pmatrix} \;.
\end{equation*}
We now want to understand what can be said for a pair of symplectic structures.
Below we give an approximate normal form which is sufficient for our purposes
and more pleasant to state than the precise result 
(cf.~\cite[Theorem~9.1]{LancasterRodman}), though the precise result can
also be extracted from the proof that we will give at the end of this
section.

\begin{proposition}
\label{prop:LR}
Let $\omega_0$ and $\omega_1$ be symplectic forms on a 
finite dimensional vector space
$V$.  There exists a matrix $A_1$ that splits into blocks of the form
\begin{equation*}
  \begin{pmatrix}
    0 &  \lambda \\
    -\lambda & 0
	\end{pmatrix} \in \mathcal{M}_2(\RR) \text{ and }
 \begin{pmatrix}
   0 & 0 & \mu & \nu  \\  0 & 0 & -\nu & \mu \\
   -\mu & \nu & 0 & 0 \\ -\nu & -\mu & 0 & 0
 \end{pmatrix} \in \mathcal{M}_4(\RR)
\end{equation*}
for $\lambda, \nu \ne 0$ with the following property: for any
$\epsilon >0$, there is a basis of $V$ such that $\omega_0$ is represented
by a block diagonal matrix with standard blocks $\Omega_{2k}$,
and $\omega_1$ is represented by a matrix which is $\epsilon$-close to~$A_1$.

If the linear segment between $\omega_0$ and $\omega_1$ consists of
symplectic forms, then the coefficients $\lambda$ in the $2\times
2$-blocks of $A_1$ described above cannot be negative.
\end{proposition}

The relation with cotamed complex structures will come from the following.

\begin{proposition}\label{prop:tamed_J_models} \ 
\begin{itemize}
  \item [(a)] Let $V = \RR^2$ with two antisymmetric bilinear forms
    $\omega_0$ and $\omega_1$ defined by $\omega_j(v,w) = v^t A_j w$,
    where
    \begin{equation*}
      A_0 =
      \begin{pmatrix}
        0 & 1 \\ -1 & 0
      \end{pmatrix}
\quad \text{ and } \quad
      A_1 =
      \begin{pmatrix}
        0 & \lambda \\ -\lambda & 0
      \end{pmatrix} \;.
    \end{equation*}
    If $\lambda > 0$, then $J =
    \begin{pmatrix}
      0 & -1 \\ 1 & 0
    \end{pmatrix}$ is tamed by both forms.
  \item [(b)] Let $V = \RR^4$, and let $\omega_0$ and $\omega_1$ be
    antisymmetric bilinear forms defined by the matrixes
    \begin{equation*}
      A_0 =
      \begin{pmatrix}
        0 & 0 & 1 & 0 \\ 0 & 0 & 0 & 1 \\ -1 & 0 & 0 & 0 \\ 0 & -1 & 0
        & 0
      \end{pmatrix}
      \quad \text{ and } \quad
      A_1 =
      \begin{pmatrix}
        0 & 0 & \lambda & \mu  \\  0 & 0 & -\mu & \lambda \\
        -\lambda & \mu & 0 & 0   \\   -\mu & -\lambda & 0 & 0  \\
      \end{pmatrix} \;,
    \end{equation*}
    with $\mu \ne 0$.  Then there exists a complex structure $J$ on
    $\RR^4$ that is tamed by both forms.
  \end{itemize}
\end{proposition}

\begin{proof}
  We only need to prove~(b).  For simplicity write $V$ as $\CC^2$, and
  the matrices $A_0$ and $A_1$ as
  \begin{equation*}
      A_0 =
      \begin{pmatrix}
        0 & 1 \\ -1 & 0
      \end{pmatrix}
      \quad \text{ and } \quad
      A_1 =
      \begin{pmatrix}
        0 & z \\ - \bar z & 0
      \end{pmatrix}
  \end{equation*}
  with $z = \lambda + i \mu = r e^{i\psi}$.  The matrices
  \begin{equation*}
    J_\phi = \begin{pmatrix}
      0 & e^{i\phi} \\ - e^{-i\phi} & 0
    \end{pmatrix}
  \end{equation*}
  define complex structures on $V$, and it follows that $A_0 J_\phi =
  - \Bigl( \begin{smallmatrix} e^{-i\phi} & 0 \\ 0 &
    e^{i\phi} \end{smallmatrix} \Bigr)$ is positive definite if
  $\cos\phi < 0$, and $A_1 J_\phi = -r\, \Bigl( \begin{smallmatrix}
    e^{i(\psi-\phi)} & 0 \\ 0 & e^{i(\phi-\psi)}
  \end{smallmatrix} \Bigr)$ is positive definite if $\cos(\psi-\phi) <
  0$.  As long as $\psi \ne \pi$ (which we have excluded by requiring
  that $\mu \ne 0$), it follows that we can choose $\phi$ such that
  $\phi \in (\pi/2, 3\pi/2)$ and $\phi - \psi \in (\pi/2, 3\pi/2) +
  2\pi\ZZ$.
\end{proof}

\begin{proof}[Proof of Proposition~\ref{prop:carac_cotamed}]
We first explain the easy equivalence between (1) and (2).
The (open) ray starting at $\omega_0$ and directed by $\omega_1$ and
the open interval between $\omega_0$ and $\omega_1$ span the same
cone in the space of antisymmetric bilinear forms.
Since being symplectic is invariant under nonzero scalar
multiplication, we have the equivalence.

The implication (3)~$\implies$~(1) is also direct because, for any
$t \in [0, 1]$, we have
\begin{equation*}
  \big((1 - t)\,\omega_0 + t\,\omega_1\big)(v, Jv) =
  (1 - t)\,\omega_0(v, Jv) + t\,\omega_1(v, Jv) , 
\end{equation*}
which is positive whenever $v\in V$ is nonzero.
So in particular, no such $v$ can be in the kernel of an element of
the segment between $\omega_0$ and $\omega_1$.

To prove (1) $\implies$ (3), we use the fact that by Proposition~\ref{prop:LR},
there is a matrix
$A_1'$ that splits into certain standard blocks, such that we can find
for any $\epsilon > 0$ a basis of $V$ for which $\omega_0$ is in
canonical form, and for which $\omega_1$ is represented by a matrix
that is $\epsilon$-close to~$A_1'$.

If condition~(1) holds, then the blocks of $A_1'$ correspond to the
ones described in Proposition~\ref{prop:tamed_J_models}, and we obtain
the existence of a complex structure $J$ on $V$ that is tamed both by
the standard symplectic form and by $A_1'$.  By choosing $\epsilon >
0$ sufficiently small, it follows that $J$ is also tamed by $\omega_0$
and $\omega_1$, because tameness is an open condition.
\end{proof}

\begin{proof}[Proof of Proposition~\ref{prop:LR}]
The proof will proceed in several steps.

\textbf{Decomposition into generalized eigenspaces.}
In the first step we shall decompose $V$ into suitable subspaces that are both
$\omega_0$- and $\omega_1$-orthogonal.

Let $\varphi_r \colon V \to V^*$ for $r=0,1$ be the isomorphisms defined
by $\varphi_r(v) := \omega_r(v, \cdot)$.
We consider the endomorphism $B = \varphi_0^{-1} \circ \varphi_1$ of
$V$ so that $\omega_1(v, w) = \omega_0(Bv, w)$. The endomorphism $B$ is
invertible and it is $\omega_0$-symmetric since:
\begin{equation*}
  \omega_0(Bv, w) = \omega_1(v, w) = -\omega_1(w, v) = -\omega_0(Bw, v)
  = \omega_0(v, Bw) \;.
\end{equation*}

To define the generalized eigenspaces of $B$, complexify the vector
space~$V$ to obtain $V^\CC$, and extend the $\omega_r$ to sesquilinear
forms $\omega_r^\CC$. 
A computation analogous to the preceding one shows that $B$ is
$\omega_0^\CC$-symmetric and we still have 
$\omega_0^\CC(v,Bw) = \omega_1^\CC(v,w)$.

The characteristic polynomial of $B$ splits over $\CC$ as
$P(X) = \prod_\lambda (X - \lambda)^{m_\lambda}$, so we can decompose
$V^\CC$ into generalized eigenspaces
\begin{equation*}
  V^\CC = \bigoplus_{\lambda \in Sp(B)} E^\CC_\lambda \quad;\quad 
  E^\CC_\lambda = \ker (B - \lambda)^{m_\lambda} \;.
\end{equation*}

\begin{lemma}
  If $\lambda$ and $\mu$ are eigenvalues of $B$ such that $\lambda
  \neq \bar \mu$, then $E^\CC_\lambda$ and $E^\CC_\mu$ are both
  $\omega^\CC_0$- and $\omega^\CC_1$-orthogonal.
\end{lemma}

\begin{proof}
We prove by induction on $k$ and $l$ that $\ker (B - \lambda)^k$ and
$\ker (B - \mu)^l$ are orthogonal.

To start the induction, note that if $v_\lambda \in \ker (B - \lambda)
$, and $v_\mu \in \ker (B - \mu)$, then
\begin{equation*}
  (\bar \lambda - \mu)\, \omega_0^\CC(v_\lambda, v_\mu) =
  \omega_0^\CC\bigl((B - \bar \mu)\,v_\lambda, v_\mu\bigr) =
  \omega_0^\CC\bigl(v_\lambda, (B-\mu)\, v_\mu\bigr) = 0 \;,
\end{equation*}
thus since $\lambda \ne \bar \mu$, it follows that
$\omega_0^\CC(v_\lambda, v_\mu) = 0$.  Similarly,
$\omega_1^\CC(v_\lambda, v_\mu) = \omega_0^\CC(v_\lambda, B v_\mu) =
\mu\, \omega_0^\CC(v_\lambda, v_\mu) = 0$.

Assume now it has already been shown for the integers $k$ and $l$ that
$\ker (B - \lambda)^k$ and $\ker (B - \mu)^l$ are both
$\omega_0^\CC$- and $\omega_1^\CC$-orthogonal.  Choose a vector
$v_\lambda' \in \ker (B - \lambda)^{k+1}$ and use the fact that $Bv_\lambda' =
\lambda\, v_\lambda' + w$ for some $w \in \ker (B - \lambda)^k$.  Then
we obtain for any $v_\mu \in \ker (B - \mu)^l$,
\begin{equation*}
  \begin{split}
    (\bar\lambda - \mu)^l\, \omega_0^\CC(v_\lambda', v_\mu) &= (\bar
    \lambda - \mu)^{l-1}\,\omega_0^\CC\bigl((B-\bar \mu)\, v_\lambda'
    - w, v_\mu\bigr) \\ &= (\bar \lambda - \mu)^{l-1}\,
    \omega_0^\CC\bigl((B - \bar\mu)\, v_\lambda', v_\mu\bigr) =
    \omega_0^\CC\bigl( v_\lambda', (B-\mu)^l\, v_\mu\bigr) = 0 \;,
  \end{split}
\end{equation*}
and also $\omega_1^\CC(v_\lambda', v_\mu) = \omega_0^\CC(B v_\lambda',
v_\mu) = \bar\lambda\,\omega_0^\CC(v_\lambda', v_\mu) +
\omega_0^\CC(w, v_\mu) = 0$, which proves the induction step from $(k, l)$
to $(k + 1, l)$. Since $\lambda$ and $\mu$ have completely symmetric roles, this
also explains how to go to $(k, l + 1)$.
\end{proof}

We now relate this decomposition of $V^\CC$ to the initial
real vector space $V$.
For a real eigenvalue $\lambda$, the intersection $V\cap E_\lambda^\CC$
defines a real subspace $E_\lambda$ with $\dim_\RR E_\lambda =
\dim_\CC E_\lambda^\CC$.  Complex conjugation defines an isomorphism
$E_\lambda^\CC \to E_{\bar\lambda}^\CC, v_\lambda \mapsto \bar
v_\lambda$, and we can write $V\cap \bigl(E_\lambda^\CC \oplus E_{\bar
  \lambda}^\CC\bigr)$ for $\lambda \in \CC\setminus \RR$ as the direct
sum of real subspaces $E_{\{\lambda,\bar \lambda\}} = \bigl\{v+\bar v
\bigm|\, v\in E_\lambda^\CC \bigr\} \oplus \bigl\{i\,(v-\bar v)
\bigm|\, v\in E_\lambda^\CC \bigr\}$.

This way we find a decomposition of $V$ into pairwise $\omega_0$- and
$\omega_1$-orthogonal subspaces
\begin{equation*}
  E_{\mu_1}\oplus \dotsm \oplus E_{\mu_k} \oplus
  E_{\{\lambda_1,\bar \lambda_1\}}
  \oplus \dotsm \oplus  E_{\{\lambda_l,\bar \lambda_l\}}
\end{equation*}
with $\mu_1, \dotsc, \mu_k \in \RR\setminus\{0\}$, and
$\lambda_1,\dotsc,\lambda_l \in \CC\setminus \RR$.

\textbf{Blocks with real eigenvalue.}
For the following considerations, we restrict to one of the
subspaces $E_{\lambda_j}$ with $\lambda_j \in \RR$, and denote $\lambda_j$ for
simplicity just by $\lambda$.
We will construct a basis of $E_\lambda$ such that $\omega_0$ and
$\omega_1$ have the particularly nice form described in the
proposition.
Note that $\omega_0$ and $\omega_1$ are both nondegenerate on
$E_\lambda$. 

Let $k+1$ be the nilpotency index of $B - \lambda$, i.e.~$(B -
\lambda)^{k+1} = 0$ and $(B - \lambda)^k \ne 0$.  Let $v_0$ be an
element of $E_\lambda$ not in $\ker(B - \lambda)^k$.  We set $v_j :=
\epsilon^{-j}(B - \lambda)^jv_0$ to define a collection of vectors
$v_0,\dotsc,v_k$.
Choose now a vector $w_k \in E_\lambda$ with $\omega_0(v_k,w_k) = 1$
and $\omega_0(v_j,w_k) = 0$ for every $j \ne k$, and define
inductively $w_{j-1} := \epsilon^{-1}\, (B-\lambda)\, w_j$, or
equivalently
\begin{equation*}
  B w_j = \lambda\, w_j + \epsilon\, w_{j-1}
\end{equation*}
for $j\ge 1$.

\begin{lemma}\label{lemma: real_eigenspaces_relations}
  The vectors $v_0,\dotsc, v_k, w_0,\dotsc, w_k$ are linearly
  independent and satisfy the relations $\omega_r(v_j,v_{j'}) =
  \omega_r(w_j, w_{j'}) = 0$ for all $r=0,1$, and $j,j'$, and
  \begin{equation*}
    \omega_0(v_j, w_{j'}) = \delta_{j,j'} \quad\text{ and }\quad \omega_1(v_j, w_{j'}) =
    \lambda\, \delta_{j,j'} + \epsilon\,  \delta_{j,j'-1} \;.
  \end{equation*}
\end{lemma}

\begin{proof}
We start by proving $\omega_r(v_j, v_{j'}) = 0$.
For this we will use an induction on $\abs{j - j'}$.
If $j - j' = 0$ then the statement follows directly from the
antisymmetry of $\omega_r$.
Suppose that the claim is true for $j - j' \le m$ and consider any $j$
and $j'$ with $j - j' = m + 1$ (in particular $j \ge 1$).  We have
\begin{align*}
  \epsilon\, \omega_0(v_j, v_{j'}) &= \omega_0\bigl((B - \lambda)\, v_{j -
    1}, v_{j'}\bigr) = \omega_1(v_{j - 1}, v_{j'}) -
  \lambda\,\omega_0(v_{j - 1}, v_{j'}) = 0 \intertext{by the induction
    hypothesis.  Using the fact that $B v_{j'} = \epsilon\, v_{j' + 1} + \lambda\,
    v_{j'}$, we compute} \omega_1(v_j, v_{j'}) &= \omega_0(v_j,
  Bv_{j'}) = \epsilon\, \omega_0(v_j, v_{j' + 1}) + \lambda\,\omega_0(v_j,
  v_{j'}) \;.
\end{align*}
The first term is zero by the induction hypothesis and the second one is zero
because of the preceding computation.
The proof of $\omega_r(w_j, w_{j'}) = 0$ follows the same lines, and
will be omitted.

Note that 
\begin{align*}
\omega_0(v_j,w_{j'}) &= \epsilon^{j'-k}\,
\omega_0\bigl(v_j, (B-\lambda)^{k-j'}\, w_k\bigr)\\
&= \epsilon^{j'-k}\, \omega_0\bigl((B-\lambda)^{k-j'}\, v_j, w_k\bigr)
= \omega_0(v_{k + j-j'}, w_k) = \delta_{j,j'} \;,
\end{align*}
and in particular this implies that $v_0,\dotsc,v_k, w_0,\dotsc, w_k$
are linearly independent vectors with respect to which $\omega_0$ has
standard form.

The remaining relation for $\omega_1$ can be obtained by
\begin{equation*}
  \omega_1(v_j,w_{j'}) = \omega_0(v_j,Bw_{j'}) =
  \lambda\,\omega_0(v_j,w_{j'}) + \epsilon\, \omega_0(v_j,w_{j'-1}) =
  \lambda\,\delta_{j,j'} + \epsilon\,\delta_{j,j'-1}\;. \qedhere
\end{equation*}
\end{proof}

If we restrict $\omega_0$ and $\omega_1$ to the subspace 
$E = \Span(v_0,\dotsc,v_k, w_0,\dotsc,w_k)$ and represent them in
this basis, we now find that $\omega_0$ is in standard
form $\Omega_{2k}$ and $\omega_1$ is represented by a matrix
$\epsilon$-close to $\lambda\, \Omega_{2k}$.

To continue the proof, restrict $\omega_0$, $\omega_1$, and $B$ to the
$\omega_0$-symplectic complement $E'$ of the space $E$.  Note that
$E'$ is stable under $B$ because for $u\in E'$,
\begin{equation*}
  \omega_0(v_j,Bu) = \omega_0(Bv_j,u) = \lambda \,\omega_0(v_j,u) +
  \epsilon\, \omega_0(v_{j-1},u) = 0 \;,
\end{equation*}
and similarly for
$\omega_0(w_j,Bu) = 0$.  We can thus proceed as before to reduce all
eigenspaces $E_\lambda$ with $\lambda\in\RR$ to $\omega_0$-symplectic blocks
in normal form.

\textbf{Blocks with complex eigenvalue.} 
We proceed now to the generalized
complex eigenspace $E^\CC_\lambda$ with $\lambda \in \CC\setminus
\RR$.  Let $k$ be the largest integer for which $E^\CC_\lambda \ne
\ker (B-\lambda)^k$, and construct as before a chain of vectors
$v_0,\dotsc,v_k \in E^\CC_\lambda$ by starting with an element $v_0
\in E^\CC_\lambda \setminus \ker (B-\lambda)^k$, and defining
inductively
\begin{equation*}
  v_{j+1} := \epsilon^{-1}\,(B-\lambda)\, v_j \;.
\end{equation*}
Using complex conjugation, we also find a chain $\bar v_0,\dotsc, \bar
v_k$ that lies in $E^\CC_{\bar\lambda}$.  Since $B$ is the
complexification of a real linear map, $\bar v_{j+1} :=
\epsilon^{-1}\,(B- \bar\lambda)\, \bar v_j$ holds.

Next, we define two chains $w_0,\dotsc, w_k$ in
$E^\CC_{\bar\lambda}$ and $\bar w_0,\dotsc, \bar w_k$ in
$E^\CC_\lambda$ by starting with a vector $w_k \in E^\CC_{\bar
  \lambda}$ with $\omega^\CC_0(v_k,w_k) = 1$ and
$\omega^\CC_0(v_j,w_k) = 0$ for every $j \ne k$, and defining
$w_{j-1} := \epsilon^{-1}\, (B-\bar \lambda)\, w_j$, or equivalently
\begin{equation*}
  B w_j = \bar \lambda \, w_j + \epsilon\, w_{j-1}
\end{equation*}
for $j\ge 1$.  Similarly, we obtain $\bar w_{j-1} = \epsilon^{-1}\,
(B- \lambda)\, \bar w_j$.

\begin{lemma} \ 
\begin{itemize}
\item [(a)] The space spanned by $v_0,\dotsc, v_{k - 1}, \bar v_0,
	\dotsc, \bar v_{k - 1}$ and the one spanned by $w_0,\dotsc, 
	w_{k - 1}, \bar w_0, \dotsc, \bar w_{k - 1}$ are each isotropic with respect
	to both $\omega_0$ and $\omega_1$.
\item [(b)] The $\omega_0^\CC$-pairings for these vectors are given
	by
	\begin{align*}
		\omega_0^\CC(v_j, \bar w_{j'}) = 0 \;, \qquad&
		\omega_0^\CC(v_j, w_{j'}) = \delta_{j,j'}\;, \\
		\omega_0^\CC(\bar v_j, w_{j'}) = 0 \;, \qquad& \omega_0^\CC(\bar
		v_j, \bar w_{j'}) = \delta_{j,j'} \;.
	\end{align*}
\item [(c)] The $\omega_1^\CC$-pairings for these vectors are given
	by
	\begin{align*}
		\omega_1^\CC(v_j, \bar w_{j'}) = 0 \;, \qquad& \omega_1^\CC(v_j,
		w_{j'}) = \lambda\, \delta_{j,j'} + \epsilon\,
		\delta_{j,j'-1}  \;, \\
		\omega_1^\CC(\bar v_j, w_{j'}) = 0 \;, \qquad& \omega_1^\CC(\bar
		v_j, \bar w_{j'}) = \bar \lambda\, \delta_{j,j'} + \epsilon\,
		\delta_{j,j'-1} \;.
	\end{align*}
\end{itemize}
\end{lemma}
\begin{proof}
To prove (a) note that since $\lambda \ne \bar \lambda$, the spaces
$E^\CC_\lambda$ and $E^\CC_{\bar\lambda}$ are both $\omega_0^\CC$-
and $\omega_1^\CC$-isotropic, so we only need to show that
$\omega_r^\CC (\bar v_j, v_{j'}) = \omega_r^\CC (\bar w_j, w_{j'}) =
0$ for all $j,j'$, and for $r=0,1$.  If $j=j'$, we write $v_j$ as $v_x
+ i v_y$, and we use sesquilinearity as follows:
\begin{align*}
  \omega_0^\CC(\bar v_j, v_j) &= \omega_0^\CC(v_x, v_x) +
  \omega_0^\CC(v_x, i v_y)
  -\omega_0^\CC(i v_y, v_x) - \omega_0^\CC(i v_y, i v_y) \\
  &= \omega_0(v_x, v_x) + i\omega_0(v_x, v_y)
  +i\omega_0(v_y, v_x) - \omega_0( v_y,  v_y) \\
  &= 0 .
\end{align*}
By the same computation, $\omega_1^\CC(\bar v_j, v_j) = 0$.

If the statement is true for $j'-j = m \ge 0$, then
\begin{align*}
  \epsilon\, \omega_0^\CC(\bar v_j, v_{j'+1}) &=
  \omega_0^\CC\bigl(\bar v_j, (B-\lambda)\,v_{j'}\bigr) =
  \omega_1^\CC(\bar v_j, v_{j'}) - \lambda\, \omega_0^\CC (\bar v_j,
  v_{j'})\\
  &= 0\\[-.2cm]
  \intertext{and} \omega_1^\CC(\bar v_j, v_{j'+1}) &= \omega_0^\CC(B
  \bar v_j, v_{j'+1}) = \omega_0^\CC(\bar \lambda\,
  \bar v_j + \epsilon\, \bar v_{j+1}, v_{j'+1}) \\
  &= 0 \;,
\end{align*}
which finishes the induction.  The argument for $\omega_r^\CC (\bar
w_j, w_{j'})$ is identical.

To prove (b), note first that the second two equations are the complex
conjugate of the first two.  Since $v_j, \bar w_{j'} \in
E_\lambda^\CC$, it also follows immediately that $\omega_0^\CC(\bar
v_j, \bar w_{j'}) = 0$, so that we are only left with showing
$\omega_0^\CC(v_j, w_{j'}) = \delta_{j,j'}$, but the required
computation is identical to the one used to show the analogous
relation in the proof of Lemma~\ref{lemma:
  real_eigenspaces_relations}.

The equalities for~(c) follow similarly.
\end{proof}

We will now intersect the complex subspace spanned by the chains
defined above with the initial real vector space $V$ to finish the
proof of the proposition.
For this, define for all $j \le k$ the real vectors
\begin{align*}
  v_j^+ = \frac{1}{\sqrt{2}}\,(v_j + \bar v_j), \qquad& v_j^- =
  \frac{i}{\sqrt{2}}\, (v_j - \bar v_j) \intertext{and} w_j^+ =
  \frac{1}{\sqrt{2}}\,(w_j + \bar w_j), \qquad& w_j^- =
  \frac{i}{\sqrt{2}}\, (w_j - \bar w_j)
\end{align*}
which all lie in $E_{\lambda, \bar \lambda}$.  Using the results deduced
above, we obtain for all $r=0,1$, and $j,j'$ the equations
$\omega_r(v_j^+, v_{j'}^\pm) = \omega_r(v_j^-, v_{j'}^\pm) = 0$ and
$\omega_r(w_j^+, w_{j'}^\pm) = \omega_r(w_j^-, w_{j'}^\pm) = 0$, and finally
\begin{align*}
  2\,\omega_0(v_j^+, w_{j'}^+) &= \omega_0^\CC(v_j, w_{j'} + \bar
  w_{j'}) +
  \omega_0^\CC(\bar v_j, w_{j'} + \bar w_{j'}) = 2\,\delta_{j,j'}\;, \\
  2\,\omega_0(v_j^+, w_{j'}^-) &= i\,\omega_0^\CC(v_j, w_{j'} - \bar
  w_{j'}) +
  i\,\omega_0^\CC(\bar v_j, w_{j'} - \bar w_{j'}) = 0\;, \\
  2\,\omega_1(v_j^+, w_{j'}^+) &= \omega_1^\CC(v_j, w_{j'} + \bar
  w_{j'}) + \omega_1^\CC(\bar v_j, w_{j'} + \bar w_{j'}) \\
  & = \omega_0^\CC(v_j, Bw_{j'}) + \omega_0^\CC(\bar v_j, B\bar
  w_{j'}) = \bar \lambda \, \omega_0^\CC(v_j, w_{j'}) + \epsilon\,
  \omega_0^\CC(v_j, w_{j'-1}) \\
  & + \lambda \, \omega_0^\CC(\bar v_j, \bar w_{j'})
  +   \epsilon\, \omega_0^\CC(\bar v_j,  \bar w_{j-1}) \\
  & = (\lambda + \bar \lambda) \, \delta_{j, j'} +
  2\epsilon\,\delta_{j, j' - 1}
\end{align*}
and similar computations for the other matrix elements, which prove the 
desired result with $\mu = \RealPart\lambda$ and $\nu = \ImaginaryPart\lambda$.

\textbf{Sign of real eigenvalues.} 
Assume that all $2$-forms in the
family
\begin{equation*}
  \omega_t := (1-t)\,\omega_0 + t\,\omega_1
\end{equation*}
for $t\in [0,1]$ are nondegenerate.  The $\lambda$-coefficients in
the $2\times 2$-blocks of $A_1'$ correspond to the real eigenvalues
of the map $B$, so that if $\lambda < 0$ with eigenvector $v$, then we
have $\omega_1(v,\cdot) = \omega_0(Bv,\cdot) = \lambda\,
\omega_0(v,\cdot)$, and it follows that $\omega_t(v,\cdot) = (1-t +
t\lambda)\,\omega_0(v,\cdot)$ has to vanish for a certain value
$t_0\in (0,1)$, so that $\omega_{t_0}$ is degenerate.
\end{proof}

\bibliographystyle{amsalphaurl}
\bibliography{main}

\newcommand{\etalchar}[1]{$^{#1}$}
\providecommand{\bysame}{\leavevmode\hbox to3em{\hrulefill}\thinspace}
\providecommand{\MR}{\relax\ifhmode\unskip\space\fi MR }
\providecommand{\MRhref}[2]{%
  \href{http://www.ams.org/mathscinet-getitem?mr=#1}{#2}
}
\providecommand{\href}[2]{#2}
\begin{thebibliography}{GVHM09}

\bibitem[ABW10]{AlbersBramhamWendl}
P.~Albers, B.~Bramham, and C.~Wendl, \emph{On non-separating contact
  hypersurfaces in symplectic $4$--manifolds}, Algebr. Geom. Topol. \textbf{10}
  (2010), 697--737, URL: \url{http://dx.doi.org/10.2140/agt.2010.10.697}.

\bibitem[BEH{\etalchar{+}}03]{BourgeoisCompactness}
F.~Bourgeois, Y.~Eliashberg, H.~Hofer, K.~Wysocki, and E.~Zehnder,
  \emph{Compactness results in symplectic field theory}, Geom. Topol.
  \textbf{7} (2003), 799--888 (electronic), URL:
  \url{http://dx.doi.org/10.2140/gt.2003.7.799}.

\bibitem[Ben83]{Bennequin}
D.~Bennequin, \emph{Entrelacements et équations de {P}faff}, Third
  {S}chnepfenried geometry conference, {V}ol. 1 ({S}chnepfenried, 1982),
  Astérisque, vol. 107, Soc. Math. France, Paris, 1983, pp.~87--161.

\bibitem[Bou02a]{Bourgeois_thesis}
F.~Bourgeois, \emph{A {M}orse-{B}ott approach to contact homology}, ProQuest
  LLC, Ann Arbor, MI, 2002, Thesis (Ph.D.)--Stanford University, URL:
  \url{http://proquest.umi.com/pqdlink?did=726452491&Fmt=7&clientId
  =79356&RQT=309&VName=PQD}.

\bibitem[Bou02b]{BourgeoisTori}
\bysame, \emph{Odd dimensional tori are contact manifolds}, Int. Math. Res.
  Not. \textbf{2002} (2002), no.~30, 1571--1574, URL:
  \url{http://dx.doi.org/10.1155/S1073792802205048}.

\bibitem[Bou06]{Bourgeois_homotopy}
\bysame, \emph{Contact homology and homotopy groups of the space of contact
  structures}, Math. Res. Lett. \textbf{13} (2006), no.~1, 71--85.

\bibitem[Bou09]{Bourgeois_CH}
\bysame, \emph{A survey of contact homology}, New perspectives and challenges
  in symplectic field theory, CRM Proc. Lecture Notes, vol.~49, Amer. Math.
  Soc., Providence, RI, 2009, pp.~45--71.

\bibitem[BvK10]{BourgeoisContactHomologyLeftHanded}
F.~Bourgeois and O.~van Koert, \emph{Contact homology of left-handed
  stabilizations and plumbing of open books}, Commun. Contemp. Math.
  \textbf{12} (2010), no.~2, 223--263, URL:
  \url{http://dx.doi.org/10.1142/S0219199710003762}.

\bibitem[CGH09]{ColinGH}
V.~Colin, E.~Giroux, and K.~Honda, \emph{Finitude homotopique et isotopique des
  structures de contact tendues}, Publ. Math. Inst. Hautes Études Sci.
  \textbf{109} (2009), no.~1, 245--293, URL:
  \url{http://dx.doi.org/10.1007/s10240-009-0022-y}.

\bibitem[CV10]{Cieliebak_Volkov_SHS}
K.~Cieliebak and E.~Volkov, \emph{First steps in stable {H}amiltonian
  topology}, preprint, 2010, arXiv:1003.5084v3.

\bibitem[DG12]{Geiges_circle_bundles}
F.~Ding and H.~Geiges, \emph{Contact structures on principal circle bundles},
  Bull. London Math. Soc. (2012), URL:
  \url{http://dx.doi.org/10.1112/blms/bds042}.

\bibitem[Dra04]{Dragnev}
D.~Dragnev, \emph{Fredholm theory and transversality for noncompact
  pseudoholomorphic maps in symplectizations}, Comm. Pure Appl. Math.
  \textbf{57} (2004), no.~6, 726--763, URL:
  \url{http://dx.doi.org/10.1002/cpa.20018}.

\bibitem[EG91]{EliashbergGromov_convex}
Y.~Eliashberg and M.~Gromov, \emph{Convex symplectic manifolds}, Several
  complex variables and complex geometry, {P}art 2 ({S}anta {C}ruz, {CA},
  1989), Proc. Sympos. Pure Math., vol.~52, Amer. Math. Soc., Providence, RI,
  1991, pp.~135--162.

\bibitem[EGH00]{SFT}
Y.~Eliashberg, A.~Givental, and H.~Hofer, \emph{Introduction to symplectic
  field theory}, Geom. Funct. Anal. (2000), no.~Special Volume, Part II,
  560--673, GAFA 2000 (Tel Aviv, 1999).

\bibitem[EH02]{EtnyreHonda_tightNonfillable}
J.~Etnyre and K.~Honda, \emph{Tight contact structures with no symplectic
  fillings}, Invent. Math. \textbf{148} (2002), no.~3, 609--626, URL:
  \url{http://dx.doi.org/10.1007/s002220100204}.

\bibitem[Eli89]{Eliashberg_Overtwisted}
Y.~Eliashberg, \emph{Classification of overtwisted contact structures on
  $3$--manifolds}, Invent. Math. \textbf{98} (1989), no.~3, 623--637, URL:
  \url{http://dx.doi.org/10.1007/BF01393840}.

\bibitem[Eli90]{Eliashberg_filling}
\bysame, \emph{Filling by holomorphic discs and its applications}, Geometry of
  low-dimensional manifolds, 2 ({D}urham, 1989), London Math. Soc. Lecture Note
  Ser., vol. 151, Cambridge Univ. Press, Cambridge, 1990, pp.~45--67, URL:
  \url{http://dx.doi.org/10.1017/CBO9780511629341.006}.

\bibitem[Eli91]{EliashbergContactProperties}
\bysame, \emph{On symplectic manifolds with some contact properties}, J.
  Differential Geom. \textbf{33} (1991), no.~1, 233--238, URL:
  \url{http://projecteuclid.org/getRecord?id=euclid.jdg/1214446036}.

\bibitem[Eli96]{Eliashberg3Torus}
\bysame, \emph{Unique holomorphically fillable contact structure on the
  $3$--torus}, Internat. Math. Res. Notices (1996), no.~2, 77--82, URL:
  \url{http://dx.doi.org/10.1155/S1073792896000074}.

\bibitem[EP11]{EtnyreGeneralizedLutzTwists}
J.~Etnyre and D.~Pancholi, \emph{On generalizing {L}utz twists}, J. London
  Math. Soc. \textbf{84} (2011), no.~3, 670--688, URL:
  \url{http://dx.doi.org/10.1112/jlms/jdr028}.

\bibitem[Etn04]{Etnyre_planar}
J.~Etnyre, \emph{Planar open book decompositions and contact structures}, Int.
  Math. Res. Not. (2004), no.~79, 4255--4267, URL:
  \url{http://dx.doi.org/10.1155/S1073792804142207}.

\bibitem[Gay06]{Gay_GirouxTorsion}
D.~Gay, \emph{Four-dimensional symplectic cobordisms containing three-handles},
  Geom. Topol. \textbf{10} (2006), 1749--1759 (electronic), URL:
  \url{http://dx.doi.org/10.2140/gt.2006.10.1749}.

\bibitem[Gei94]{Geiges_disconnected}
H.~Geiges, \emph{Symplectic manifolds with disconnected boundary of contact
  type}, Internat. Math. Res. Notices (1994), no.~1, 23--30, URL:
  \url{http://dx.doi.org/10.1155/S1073792894000048}.

\bibitem[Gei95]{Geiges_disconnected4}
\bysame, \emph{Examples of symplectic {$4$}--manifolds with disconnected
  boundary of contact type}, Bull. London Math. Soc. \textbf{27} (1995), no.~3,
  278--280, URL: \url{http://dx.doi.org/10.1112/blms/27.3.278}.

\bibitem[Gei08]{GeigesBook}
\bysame, \emph{An introduction to contact topology}, Cambridge Studies in
  Advanced Mathematics, vol. 109, Cambridge University Press, Cambridge, 2008,
  URL: \url{http://dx.doi.org/10.1017/CBO9780511611438}.

\bibitem[GH08]{GhigginiHonda_twisted}
P.~Ghiggini and K.~Honda, \emph{Giroux torsion and twisted coefficients},
  preprint, 2008.

\bibitem[Gir91]{Giroux_91}
E.~Giroux, \emph{Convexité en topologie de contact}, Comment. Math. Helv.
  \textbf{66} (1991), no.~4, 637--677, URL:
  \url{http://dx.doi.org/10.1007/BF02566670}.

\bibitem[Gir94]{Giroux_plusOuMoins}
\bysame, \emph{Une structure de contact, même tendue, est plus ou moins
  tordue}, Ann. Sci. École Norm. Sup. (4) \textbf{27} (1994), no.~6, 697--705,
  URL: \url{http://www.numdam.org/item?id=ASENS_1994_4_27_6_697_0}.

\bibitem[Gir99]{Giroux_infinity}
\bysame, \emph{Une infinité de structures de contact tendues sur une infinité
  de variétés}, Invent. Math. \textbf{135} (1999), no.~3, 789--802, URL:
  \url{http://dx.doi.org/10.1007/s002220050301}.

\bibitem[Gir00]{Giroux_2000}
\bysame, \emph{Structures de contact en dimension trois et bifurcations des
  feuilletages de surfaces}, Invent. Math. \textbf{141} (2000), no.~3,
  615--689, URL: \url{http://dx.doi.org/10.1007/s002220000082}.

\bibitem[Gir02]{Giroux_ICM}
\bysame, \emph{Géométrie de contact: de la dimension trois vers les
  dimensions supérieures}, Proceedings of the {I}nternational {C}ongress of
  {M}athematicians, {V}ol. {II} (Beijing), Higher Ed. Press, 2002,
  pp.~405--414.

\bibitem[Gir10]{Giroux_Bourbaki}
\bysame, \emph{Sur la g\'eom\'etrie et la dynamique des transformations de
  contact (d'apr\`es {Y}. {E}liashberg, {L}. {P}olterovich et al.)},
  Ast\'erisque (2010), no.~332, Exp. No. 1004, viii, 183--220, S{\'e}minaire
  Bourbaki. Volume 2008/2009. Expos{\'e}s 997--1011.

\bibitem[Gro85]{Gromov_HolCurves}
M.~Gromov, \emph{Pseudo holomorphic curves in symplectic manifolds}, Invent.
  Math. \textbf{82} (1985), 307--347, URL:
  \url{http://dx.doi.org/10.1007/BF01388806}.

\bibitem[GVHM09]{Ghiggini_VHM}
P.~Ghiggini and J.~Van Horn-Morris, \emph{Tight contact structures on the
  brieskorn spheres {$-\Sigma(2,3,6n-1)$} and contact invariants}, preprint,
  2009.

\bibitem[Hof93]{HoferWeinstein}
H.~Hofer, \emph{Pseudoholomorphic curves in symplectizations with applications
  to the {W}einstein conjecture in dimension three}, Invent. Math. \textbf{114}
  (1993), no.~3, 515--563, URL: \url{http://dx.doi.org/10.1007/BF01232679}.

\bibitem[Hof06]{Hofer_polyfoldSurvey}
\bysame, \emph{A general {F}redholm theory and applications}, Current
  developments in mathematics, 2004, Int. Press, Somerville, MA, 2006,
  pp.~1--71.

\bibitem[HWZ11]{HoferWZ_GW}
H.~Hofer, K.~Wysocki, and E.~Zehnder, \emph{Applications of polyfold theory
  {I}: {G}romov--{W}itten theory}, preprint, 2011, arXiv:1107.2097.

\bibitem[Ler01]{Lerman_ContactCuts}
E.~Lerman, \emph{Contact cuts}, Israel J. Math. \textbf{124} (2001), 77--92,
  URL: \url{http://dx.doi.org/10.1007/BF02772608}.

\bibitem[LR05]{LancasterRodman}
P.~Lancaster and L.~Rodman, \emph{Canonical forms for symmetric/skew-symmetric
  real matrix pairs under strict equivalence and congruence}, Linear Algebra
  Appl. \textbf{406} (2005), 1--76, URL:
  \url{http://dx.doi.org/10.1016/j.laa.2005.03.035}.

\bibitem[LW11]{LatschevWendl}
J.~Latschev and C.~Wendl, \emph{Algebraic torsion in contact manifolds}, Geom.
  Funct. Anal. \textbf{21} (2011), no.~5, 1144--1195, With an appendix by M.
  Hutchings., URL: \url{http://dx.doi.org/10.1007/s00039-011-0138-3}.

\bibitem[Mar77]{Marcus_numberFields}
D.~A. Marcus, \emph{Number fields}, Springer-Verlag, New York, 1977,
  Universitext.

\bibitem[Mas08]{Massot_GCS}
P.~Massot, \emph{Geodesible contact structures on 3-manifolds}, Geom. Topol.
  \textbf{12} (2008), no.~3, 1729--1776.

\bibitem[McD91]{McDuff_contactType}
D.~McDuff, \emph{Symplectic manifolds with contact type boundaries}, Invent.
  Math. \textbf{103} (1991), no.~3, 651--671, URL:
  \url{http://dx.doi.org/10.1007/BF01239530}.

\bibitem[Mit95]{Mitsumatsu_Anosov}
Y.~Mitsumatsu, \emph{Anosov flows and non-{S}tein symplectic manifolds}, Ann.
  Inst. Fourier (Grenoble) \textbf{45} (1995), no.~5, 1407--1421, URL:
  \url{http://www.numdam.org/item?id=AIF_1995__45_5_1407_0}.

\bibitem[Mor09]{Mori_Lutz}
A.~Mori, \emph{{R}eeb foliations on {$S^5$} and contact $5$--manifolds
  violating the {T}hurston-{B}ennequin inequality}, preprint, 2009,
  arXiv:0906.3237.

\bibitem[MS04]{McDuffSalamonJHolo}
D.~McDuff and D.~Salamon, \emph{{$J$}--holomorphic curves and symplectic
  topology}, {Colloquium Publications. American Mathematical Society 52.
  Providence, RI: American Mathematical Society (AMS).}, 2004.

\bibitem[Nie06]{NiederkruegerPlastikstufe}
K.~Niederkrüger, \emph{The plastikstufe - a generalization of the overtwisted
  disk to higher dimensions}, Algebr. Geom. Topol. \textbf{6} (2006),
  2473--2508, URL: \url{http://dx.doi.org/10.2140/agt.2006.6.2473}.

\bibitem[NP10]{NiederkruegerPresas_neighborhoods}
K.~Niederkrüger and F.~Presas, \emph{Some remarks on the size of tubular
  neighborhoods in contact topology and fillability}, Geom. Topol. \textbf{14}
  (2010), no.~2, 719--754, URL: \url{http://dx.doi.org/10.2140/gt.2010.14.719}.

\bibitem[NvK07]{vanKoertPSOvertwistedEverywhere}
K.~Niederkrüger and O.~van Koert, \emph{Every contact manifold can be given a
  nonfillable contact structure}, Int. Math. Res. Not. IMRN (2007), no.~23,
  Art. ID rnm115, 22, URL: \url{http://dx.doi.org/10.1093/imrn/rnm115}.

\bibitem[NW11]{NiederkrugerOvertwistedAnnulus}
K.~Niederkrüger and C.~Wendl, \emph{Weak symplectic fillings and holomorphic
  curves}, Ann. Sci. École Norm. Sup. (4) \textbf{44} (2011), no.~5, 801--853.

\bibitem[Pol91]{PolterovichLagrangianSurgery}
L.~Polterovich, \emph{The surgery of {L}agrange submanifolds}, Geom. Funct.
  Anal. \textbf{1} (1991), no.~2, 198--210, URL:
  \url{http://dx.doi.org/10.1007/BF01896378}.

\bibitem[Pre07]{PresasExamplesPlastikstufes}
F.~Presas, \emph{A class of non-fillable contact structures}, Geom. Topol.
  \textbf{11} (2007), 2203--2225, URL:
  \url{http://dx.doi.org/10.2140/gt.2007.11.2203}.

\bibitem[Sal99]{Salamon_lecture_notes}
D.~Salamon, \emph{Lectures on {F}loer homology}, Symplectic geometry and
  topology ({P}ark {C}ity, {UT}, 1997), IAS/Park City Math. Ser., vol.~7, Amer.
  Math. Soc., Providence, RI, 1999, pp.~143--229.

\bibitem[Tis70]{Tischler}
D.~Tischler, \emph{On fibering certain foliated manifolds over {$S\sp{1}$}},
  Topology \textbf{9} (1970), 153--154.

\bibitem[vdG88]{vdGeer}
G.~van~der Geer, \emph{Hilbert modular surfaces}, Ergebnisse der Mathematik und
  ihrer Grenzgebiete (3) [Results in Mathematics and Related Areas (3)],
  vol.~16, Springer-Verlag, Berlin, 1988.

\bibitem[Wen10a]{WendlCobordisms}
C.~Wendl, \emph{Non-exact symplectic cobordisms between contact
  $3$--manifolds}, preprint, 2010, arXiv:1008.2456.

\bibitem[Wen10b]{WendlGirouxTorsion}
\bysame, \emph{Strongly fillable contact manifolds and {$J$}--holomorphic
  foliations}, Duke Math. J. \textbf{151} (2010), no.~3, 337--384, URL:
  \url{http://dx.doi.org/10.1215/00127094-2010-001}.

\end{thebibliography}

\end{document}